\documentclass[12pt]{article}
\usepackage[a4paper,text={170mm,250mm},centering]{geometry}
\usepackage[affil-it]{authblk}
\usepackage{amsmath,amsthm,amssymb}
\usepackage{graphicx}
\usepackage{enumitem}
\usepackage{indentfirst}
\usepackage[title,titletoc]{appendix}
\usepackage{textgreek}
\usepackage{hyperref}
\usepackage[symbols,nogroupskip,sort=none]{glossaries-extra}

\newtheorem{theorem}{Theorem}[section]
\newtheorem{corollary}[theorem]{Corollary}
\newtheorem{lemma}[theorem]{Lemma}
\newtheorem{proposition}[theorem]{Proposition}
\theoremstyle{definition}
\newtheorem{definition}{Definition}[section]
\newtheorem{assumption}{Assumption}
\theoremstyle{remark}
\newtheorem{remark}{Remark}[section]

\numberwithin{equation}{section}

\allowdisplaybreaks[4]

\def\one{\mathbf 1}
\def\bE{\mathbf E}

\def\bP{\mathbf P}

\def\R{\mathbb R}

\def\FF{\mathcal F}

\def\LL{\mathcal L}
\def\MM{\mathcal M}

\def\PP{\mathcal P}

\def\WW{\mathcal W}

\def\xa{\alpha}
\def\xb{\beta}
\def\xg{\gamma}

\def\xd{\delta}
\def\xD{\Delta}
\def\xe{\varepsilon}
\def\xf{\varphi}
\def\xk{\kappa}
\def\xl{\lambda}

\def\xo{\omega}
\def\xO{\Omega}

\def\xs{\sigma}
\def\xS{\Sigma}
\def\xt{\tau}

\def\as{\mathrm{a.s.}}

\def\d{\mathop{}\!\mathrm{d}}

\def\e{\mathrm e}

\def\st{\mathrm{st}}

\def\str#1{#1^{\ast}}

\def\abs#1{\left\lvert #1 \right\rvert}
\def\norm#1{\left\lVert #1 \right\rVert}

\def\dif#1#2{\frac{\d #1}{\d #2}}

\def\p{\partial}

\def\8{\infty}

\def\xht{\hat x}
\def\muht{\hat{\mu}}

\def\nutd{\tilde{\nu}}

\def\up{\uparrow}
\def\down{\downarrow}
\def\goto{\xrightarrow}

\hypersetup{hidelinks,colorlinks=true,linkcolor=red,citecolor=blue}

\renewcommand{\glossarysection}[2][]{}
\newglossarystyle{mylong}{%
     {\begin{longtable}[l]{@{}p{\dimexpr 0.1\textwidth -\tabcolsep}p{\dimexpr 0.9\textwidth -\tabcolsep}@{}}}%
     {\end{longtable}}%
  \renewcommand*{\glsgroupheading}[1]{}%
  \ifglsnogroupskip
  \else
  \fi
}
\glsdisablehyper
\glsxtrnewsymbol
[description={The set of probability measures on $\R^d$}]
{ProbSp}
{\ensuremath{\PP(\R^d)}}

\glsxtrnewsymbol
[description={The $p$-th root of the $p$-th moment of a probability measure $\mu$, $\norm{\mu}_p^p:=\int\abs{x}^p\d\mu(x)$}]
{Normp}
{\ensuremath{\norm{\mu}_p}}

\glsxtrnewsymbol
[description={The set of probability measures on $\R^d$ with finite $p$-th moments}]
{ProbSpp}
{\ensuremath{\PP_p(\R^d)}}

\glsxtrnewsymbol
[description={The $p$-Wasserstein distance between two probability measures $\mu,\nu\in\PP_p(\R^d)$}]
{Wp}
{\ensuremath{\WW_p(\mu,\nu)}}

\glsxtrnewsymbol
[description={The Hausdorff semi-distance between two bounded sets $A,B\subset\PP_p(\R^d)$ defined by $\WW_p(A,B):=\sup_{\mu\in A}\inf_{\nu\in B}\WW_p(\mu,\nu)$}]
{Wp-sets}
{\ensuremath{\WW_p(A,B)}}

\glsxtrnewsymbol
[description={The set of probability measures on $\R^d$ with finite $p$-th moments for all $p\geq1$}]
{ProbSpInfty}
{\ensuremath{\PP_{\8}(\R^d)}}

\glsxtrnewsymbol
[description={The stochastic order between probability measures}]
{StOrd}
{\ensuremath{\leq_{\mathrm{st}}}}

\glsxtrnewsymbol
[description={The order interval consisting of probability measures $\mu\in\PP_2(\R)$ with $\mu_1\leq_{\mathrm{st}}\mu\leq_{\mathrm{st}}\mu_2$}]
{order-interval}
{\ensuremath{[\mu_1,\mu_2]}}

\glsxtrnewsymbol
[description={The open ball consisting of probability measures $\nu\in\PP_2(\R)$ with $\WW_2(\mu,\nu)<r$}]
{open-ball}
{\ensuremath{B(\mu,r)}}

\begin{document}

\title{Global Dynamics of Granular Media Equations via Stochastic Order}
\author[1,2]{Baoyou Qu\thanks{qubaoyou@sdu.edu.cn}}
\author[2]{Jinxiang Yao\thanks{jxyao@mail.ustc.edu.cn}}
\author[2]{Yanpeng Zhi\thanks{yanpeng.zhi@durham.ac.uk}}

\affil[1]{{\small Research Center for Mathematics and Interdisciplinary Sciences, Shandong University, Qingdao, 266237, P.R. China}}
\affil[2]{{\small Department of Mathematical Sciences, Durham University, DH1 3LE, U.K.}}

\date{}

\maketitle

\begin{abstract}
This paper studies the rich dynamics of one-dimensional granular media equations with attractive quadratic interactions.
Building on the monotone dynamical systems framework developed in an earlier work, we allow for multiplicative noise, in contrast to most existing results restricted to additive noise.
Within this framework, we show that, in the one-dimensional setting, invariant measures are totally ordered with respect to the stochastic order.
The basins of attraction of the minimal and maximal invariant measures contain unbounded open sets in the $2$-Wasserstein space,  which is vacant in previous research even for additive noises.
Also, our main results address the global convergence to the order interval enclosed by the minimal and maximal invariant measures, and an alternating arrangement of invariant measures in terms of stability (locally attracting) and instability (as the backward limit of a connecting orbit).
Our theorems cover a wide range of classical granular media equations, such as double-well and multi-well landscapes.
Specific values for the parameter ranges, explicit descriptions of attracting sets and phase diagrams are provided.

\medskip

\noindent
{\bf MSC2020 subject classifications:} Primary 60H10, 60B10; secondary 37C65, 60E15


\end{abstract}

{
\hypersetup{linkcolor=black}
\tableofcontents
}

\section{Introduction}
The category of McKean-Vlasov SDEs refers to a class of SDEs whose coefficients depend on laws of solutions.
McKean \cite{McKean1966} introduced it as a generalization of It\^o's diffusion processes by considering a nonlinear version of Fokker-Planck equations.
Among McKean–Vlasov SDEs, one of the most important and widely studied classes with concrete applications is given by granular media models:
\begin{equation}\label{eq:mvsystem}
\d X_t=-V'(X_t)\d t-(W'*\LL(X_t))(X_t)\d t+\xs(X_t)\d B_t,
\end{equation}
where $\{B_t\}_{t\in\R}$ is a standard two-sided Brownian motion on $\R$, $\LL(X_t)$ is the law of the random variable $X_t$, $V:\R\to\R$ is a confining potential function, $W:\R\to\R$ is an interaction potential, $\xs:\R\to\R$ is a diffusion term and
\begin{equation*}
    (W'*\mu)(x):=\int_{\R}W'(x-y)\d\mu(y), \ \text{ for all probability measure $\mu$ on $\R$.}
\end{equation*}
\noindent The law $\rho_t=\mathcal L(X_t)$ of the solution to \eqref{eq:mvsystem} formally solves the associated nonlinear Fokker-Planck equation 
\begin{equation}
\label{eq:mvpde}
\p_t\rho_t=\frac12\p_{xx}\big(\xs^2\rho_t\big)+\p_x\big[(V'+W'*\rho_t)\rho_t\big],
\end{equation}
which describes the evolution of the particle density.
From another point of view, The equation \eqref{eq:mvsystem} can be deduced as the mean-field limit (or hydrodynamic limit) of the following interacting particle system,
\begin{equation}
\label{eq:mvparticle}
\d X_t^{i,N}=-V'\big(X_t^{i,N}\big)\d t-\frac{1}{N}\sum_{j=1}^N W'\big(X_t^{i,N}-X_t^{j,N}\big)\d t+\xs\big(X_t^{i,N}\big)\d B_t^i,\quad i=1,2,\dots,N,
\end{equation}
where $\{B_t^i\}_{i=1}^{\8}$ is a sequence of independent Brownian motions.
As the particle number $N\to\8$, a single particle $X_t^{i,N}$ or the empirical measure $\frac{1}{N}\sum_{i=1}^N \xd_{X_t^{i,N}}$ converges to the continuum process $X_t$ in some sense.
This is called ``propagation of chaos'', and we refer to \cite{Chaintron-Diez2022i,Chaintron-Diez2022ii,McKean1967,Sznitman1991} on this topic.
Due to their origin as infinite-particle limits, equations \eqref{eq:mvsystem}-\eqref{eq:mvpde} provide natural mathematical models for large-scale phenomena.
They are widely applied in various fields, such as 
battery models \cite{Guhlke-Gajewski-Maurelli-Friz-Dreyer2018},  optimization \cite{Carrillo-Choi-Totzeck-Tse2018} and machine learning \cite{Mei-Montanari-Nguyen2018}.

In this paper, we focus on the typical attractive quadratic interaction potential, namely
\[
W(x)=\frac{\theta}{2}x^2,\quad\theta>0,
\]
which arises naturally in a variety of application contexts, including muscle contraction \cite{Kometani75}, chemical kinetics \cite{Horsthemke77}, statistical physics \cite{Haken77}, and large economic systems \cite{Aoki1980}.
A salient feature of this class of equations is the occurrence of multiple invariant measures if the noise is not too strong.
For the double-well landscapes with additive noise, Dawson \cite{Dawson1983} and Tugaut \cite{Tugaut2014} obtain phase transitions in the number of invariant measures. 
For the multi-well case, we refer to Alecio \cite{Alecio2023} and the references therein. 
Roughly speaking, this phenomenon can be understood through the geometry of the confining potential. 
It is common to consider well-shaped confining potentials $V$ in granular media equations, which means $V(x)\to\infty$ as $x\to\pm\infty$. 
If its derivative $V'$ only has simple zeros, then basically $V'$ has odd $(2n-1)$ zeros. 
In such a setting, the equation has exactly $(2n-1)$ invariant measures when noise is small enough, and has a unique invariant measure when noise is large.

Consequently, the study of the dynamics of \eqref{eq:mvsystem} is often divided into two regimes: the case of a unique invariant measure and the case of multiple invariant measures. In the former, convergence to the unique invariant measure has been extensively studied. 
For instance, Carrillo-McCann-Villani \cite{Carrillo-McCann-Villani2003} employed generalized logarithmic Sobolev inequalities and mass transportation inequalities, while Cattiaux-Guillin-Malrieu \cite{Cattiaux-Guillin-Malrieu2008} used a uniform-in-time propagation of chaos method to demonstrate convergence to the unique invariant measure. For more on this topic, see \cite{Carrillo-McCann-Villani2006,Durmus-Eberle-Guillin-Zimmer2020,Guillin2022,Wang2023} and the references therein.

In the presence of multiple invariant measures, a complete characterization of the basins of attraction is in general a highly challenging problem.
Current results in this direction are mainly  restricted to  establishing the local convergence around certain stable invariant measures, showing that solutions starting from initial distributions sufficiently close to them converge to the corresponding invariant measures. 
In Tugaut \cite{Tugaut2025}, local convergence is established under the assumptions that the initial condition admits a finite higher-order moment, a $C^{\infty}$-smooth density, and finite entropy.
Monmarché-Reygner \cite{Monmarche-Reygner2025} applied a local version of the log-Sobolev inequality to obtain local convergence within a sharp range of the noise intensity $\sigma$ for double-well landscapes.
Zhang \cite{Zhang2025} achieved local convergence by linearizing the nonlinear Markov semigroup, applicable to double-well landscapes.
For early works using the free energy function, see Tamura \cite{Tamura1984,Tamura1987}, while approaches relying on Lions derivatives can be found in Cormier \cite{Cormier2024}.
From a global perspective, Tugaut \cite{Tugaut2013} showed that, for double-well granular media equations with qualitatively sufficiently small noise, solutions starting from initial conditions with finite higher-order moment, a $C^{\infty}$-smooth density, and finite entropy will converge to one of the three invariant measures.

In \cite{Liu-Qu-Yao-Zhi2024}, a monotone dynamical systems framework for cooperative McKean–Vlasov SDEs was developed.
More precisely, under a local dissipativity assumption, the existence of multiple order related invariant measures in the Wasserstein space was established, together with monotone connecting (heteroclinic) orbits between them with respect to the stochastic order.

Building on the main result of \cite{Liu-Qu-Yao-Zhi2024}, the present paper goes further in the one-dimensional setting.
By analyzing the local dissipativity condition and the number of invariant measures, we identify connecting orbits for a broader class of one-dimensional granular media equations with multiplicative noise,  including, in particular, polynomial-type confining potentials of even degree.
Relying on the existence of such connecting orbits together with a structural theorem for local attractors of order-preserving semigroups, we prove local convergence in the $2$-Wasserstein space.
Furthermore, by combining a structural result on global attractors with an analysis of the stochastic order, we show that the basins of attraction of the minimal and maximal invariant measures contain unbounded open sets (see our main results Theorem \ref{thm:global-convergence}-Corollary \ref{coro:segment-convergence}).
To illustrate the applicability of our abstract results, we present below a theorem for the case where $V'$ has finitely many simple zeros.
\begin{theorem}\label{T:general V}
    Suppose Assumption \ref{asp:potential} holds, and $V'$  has finite zeros and they are all simple. Then there exist $\theta_0>0, \ \sigma_0>0$  such that for all $\theta\geq \theta_0$ and $\overline{\sigma}\leq \xs_0$,
    \begin{enumerate}[label=\textnormal{(\roman*)}]
\item \textnormal{(Total Orderedness)} there are exactly $(2n-1)$ invariant measures $\{\nu_k\}_{k=1}^{2n-1}$ with $\nu_1<_{\st}\nu_2<_{\st}\cdots<_{\st}\nu_{2n-1}$ for some integer $n\geq 1$;
\item \textnormal{(Connecting Orbits)} for all $k=1,2,\dots,n-1$, there are a decreasing connecting orbit $\{\nu^{k,\downarrow}_t\}_{t\in \R}$ from $\nu_{2k}$ to $\nu_{2k-1}$ and an increasing connecting orbit $\{\nu^{k,\uparrow}_t\}_{t\in \R}$ from $\nu_{2k}$ to $\nu_{2k+1}$;
\item \textnormal{(Locally Attracting)} there exist $r>0$ and $a_1<\cdots<a_k$ such that for all $1\leq k\leq n$, $\nu_{2k-1}\in B(\xd_{a_k},r)$  and
\[
\WW_2(P_t^*B(\xd_{a_k},r),\nu_{2k-1})=\sup_{\mu\in B(\xd_{a_k},r)}\WW_2(P_t^*\mu,\nu_{2k-1})\to0\text{ as }t\to\8;
\]
\item \textnormal{(Attracting Unbounded Open Neighbourhoods)}
if $\mu\leq_{\st}\nu$ for some $\nu\in B(\xd_{a_1},r)$, then
\[
\WW_2(P_t^*\mu,\nu_1)\to0\text{ as }t\to\8.
\]
If $\mu\geq_{\st}\nu$ for some $\nu\in B(\xd_{a_n},r)$, then
\[
\WW_2(P_t^*\mu,\nu_{2n-1})\to0\text{ as }t\to\8.
\]
In particular,
\[
\WW_2(P_t^*\mu,\nu_1)\to0\text{ as }t\to\8\ \text{ for all $\mu\in\bigcup_{a\leq0}B(\xd_{a_1+a},r)$},
\]
\[
\WW_2(P_t^*\mu,\nu_{2n-1})\to0\text{ as }t\to\8\ \text{ for all $\mu\in\bigcup_{a\geq0}B(\xd_{a_n+a},r)$}.
\]
\end{enumerate}
\end{theorem}
Our theorem clearly includes confining potentials of polynomial type, in particular those with
\[
V'(x)=a\prod_{k=1}^{2n-1}(x-\alpha_k),
\qquad a>0,\ \alpha_k\in\mathbb R\ \text{and}\ \alpha_i\neq\alpha_j\ (i\neq j).
\]
We stress that our approach is expected to extend to more general interaction potentials, for instance those satisfying $W''> 0$, as well as to higher-dimensional systems, including a multi-species population models; see \cite{Duong-Pavliotis-Tugaut2025,DuongHongTugaut20}.

Our approach begins with  the order-preserving property of the law-evolution semigroup generated by equation \eqref{eq:mvsystem} on the $2$-Wasserstein space $\PP_2(\R)$ (see e.g., \cite{Huang-Liu-Wang2018,Liu-Qu-Yao-Zhi2024,panpanren22}).
More precisely, we consider the stochastic order $\leq_{\mathrm{st}}$ defined by
\[ 
\mu\leq_{\st}\nu\text{ if and only if }\int_{\R}f\d\mu\leq\int_{\R}f\d\nu\ \ \text{for all bounded increasing functions $f\colon\R\to\R$}.
\]
Under a hyper-dissipativity assumption, we develop a dynamical result describing the global attractor structure of order-preserving semigroups (Proposition \ref{T:attractor}), based on a characterization of order-boundedness (Proposition \ref{prop:order-bounded-equivalent}) and the coming down from infinity property (Proposition  \ref{thm:mvsde-attractor}).
This yields a partial description of basins of attraction in the sense of the stochastic order (Theorem \ref{thm:global-convergence}, Corollary \ref{coro:unique-convergence}).
Building upon this global attractor analysis, and under the additional assumptions on local dissipativity and zero-crossing number, we further derive local attractor structures by tracing solutions backward along connecting orbits (Section \ref{subsec:local attractor}). 
As a consequence, we obtain local convergence in the topological sense, as well as the existence of unbounded open subsets contained in the corresponding basins of attraction (Theorem \ref{thm:local-convergence}, Corollary \ref{coro:segment-convergence}).  
It is worth noting that this phenomenon is far from obvious.
Indeed, any order interval has empty interior in $\PP_2(\R)$ (see Proposition \ref{prop:no-strong-order}), and resembles a `thin slice'' as shown in Figure \ref{fig:double-well}, \ref{fig:multi-well}. 
So attracting open sets by the method of preserving such an unsatisfactory partial order relation is counter-intuitive.

Finally, our results suggest the presence of a saddle-point structure in the double-well granular media equation.
This leads us to formulate, at the end of the paper, a conjecture on the existence of an unordered invariant separatrix containing the unstable invariant measure, which separates the basins of attraction of another two stable invariant measures.

This paper is organized as follows.
In Section \ref{sec:main-results}, we state the main results beyond Theorem \ref{T:general V}.
Section \ref{sec:stochastic-order} is about properties of the stochastic order, including an equivalent characterization of order boundedness in $\PP_2(\R)$ and a generalization of the lattice property.
Section \ref{sec:global-convergence} is devoted to the proof of Theorem \ref{thm:global-convergence} and Corollary \ref{coro:unique-convergence}.
Section \ref{sec:local-convergence} serves the proof of the remaining theorems and the formulation of a saddle-point conjecture in the double-well granular media equation.
In Appendix \ref{sec:no strong order}, we show that, in $\PP_2(\R)$, any probability measure has no order relation with any open set, which denies the existence of strong ordered pairs for the stochastic order.

\subsection*{Notations}

\printunsrtglossary[type=symbols,style=mylong]

\subsection{Assumptions and Main Results}
\label{sec:main-results}

In this section, we state our assumptions, some necessary definitions and main results.

\begin{assumption}\label{asp:potential}
(i) The function $V\in C^2(\R;\R)$ such that the continuous function $V':\R\to\R$ is one-sided Lipschitz, i.e., there exists $L>0$ such that, for all $x,y\in\R$,
\[
-(x-y)(V'(x)-V'(y))\leq L\abs{x-y}^2.
\]
(ii) The function $V':\R\to\R$ has polynomial growth, i.e., there exist $L>0$, $\xk\geq 1$ such that, for all $x\in\R$,
\[
\abs{V'(x)}\leq L(1+|x|^{\xk}).
\]
(iii) The function $V':\R\to\R$ is hyper-dissipative, i.e., there exist $\xa>0$, $\xb>0$, $\xd>0$ such that, for all $x,y\in\R$,
\[
-(x-y)(V'(x)-V'(y))\leq-\xa\abs{x-y}^{2+\xd}+\xb.
\]
(iv) $W(x)=\frac{\theta}{2}x^2$ with $\theta>0$.\\
(v) The function $\xs:\R\to\R$ is Lipschitz continuous and non-degenerate, i.e., there exist $0<\underline{\xs}<\overline{\xs}$, such that, for all $x\in\R$,
\[
\underline{\xs}^2\leq\xs^2(x)\leq\overline{\xs}^2.
\]
\end{assumption}

Next, we emphasize the following definitions that are aforementioned.

\begin{definition}\label{def:stochastic-order}
The stochastic order ``$\leq_{\st}$'' on $\PP_2(\R)$ is a partial order, defined by
\[
\mu\leq_{\st}\nu\text{ if and only if }\int_{\R}f\d\mu\leq\int_{\R}f\d\nu\text{ for all bounded increasing functions $f\colon\R\to\R$}.
\]
\end{definition}

\begin{definition}\label{def:connecting-orbits-measure}
For a semigroup $\str P_t$ on $\PP_2(\R)$,
\begin{enumerate}[label=(\roman*)]
\item a measure $\nu\in\PP_2(\R)$ is called an invariant measure if $P_t^*\nu=\nu$ for all $t\geq0$;
\item a family $\{\nu_t\}_{t\in\mathbb{R}}\subset\PP_2(\R)$ is called a connecting orbit from an invariant measure $\nu$ to another invariant measure $\tilde{\nu}$ if
\begin{align*}
&\str P_t\nu_s=\nu_{t+s}\text{ for any $t\geq0$, $s\in\mathbb{R}$};\\
&\WW_2(\nu_t,\nu)\to0\text{ as $t\to-\8$ and }\WW_2(\nu_t,\tilde{\nu})\to0\text{ as $t\to\8$}.
\end{align*}
Furthermore, if 
$$\nu_s\leq_{\st}\nu_t\ \ (\nu_s\geq_{\st}\nu_t)\text{ for all }s\leq t,$$
then we call $\{\nu_t\}_{t\in\mathbb{R}}$ an increasing (decreasing) connecting orbit from $\nu$ to $\tilde{\nu}$.
\end{enumerate}
\end{definition}

\begin{remark}
A family $\{\nu_t\}_{t\in\R}\subset\PP_2(\R)$ is called an entrance measure if $P_t^*\nu_s=\nu_{t+s}$ for all $t\geq0$, $s\in\R$, and particularly, a connecting orbit from $\nu$ to $\nutd$ is an entrance measure.
\end{remark}

Now we state our main theorems.
Theorem \ref{thm:global-convergence}, \ref{thm:local-convergence} and Corollary \ref{coro:unique-convergence}, \ref{coro:segment-convergence} are abstract results for the equation \eqref{eq:mvsystem}, while  Theorem \ref{thm:double-well}-\ref{thm:double-well-vanish} deal with concrete equations.
In the sequel, we denote by $P_t^*:\PP_2(\R)\to\PP_2(\R)$ the semigroup associated to the equation \eqref{eq:mvsystem}:
\begin{equation}\label{eq:semigroup P}
    P_t^*\mu:=\LL(X_t^{\xi}) \ \text{ with } \ \LL(\xi)=\mu,
\end{equation}
where $X_t^{\xi}$ is the unique solution of \eqref{eq:mvsystem} with initial data $\xi$ at time 0.

\begin{theorem}\label{thm:global-convergence}
Under Assumption \ref{asp:potential}, the following statements hold for the equation \eqref{eq:mvsystem}.
\begin{enumerate}[label=\textnormal{(\roman*)}]
\item \textnormal{(Total Orderedness and Finiteness)} The set of invariant measures $\MM$ is a finite set, totally ordered with respect to the stochastic order, and $\MM\subset\PP_{\8}(\R)$;
\item \textnormal{(Global Convergence to Order Interval)} $\inf_{\leq_{\st}}\MM$ and $\sup_{\leq_{\st}}\MM$ exist.
We denote them by $\underline{\nu}$ and $\overline{\nu}$ respectively.
Then for all bounded sets $B\subset\PP_2(\R)$,
\[
\WW_2(P_t^*B,[\underline{\nu},\overline{\nu}])\to0\ \text{ as $t\to\8$};
\]
\item \textnormal{(Globally Attracting from Below/Above)} For all $\mu\leq_{\st}\underline{\nu}$,
\[
\WW_2(P_t^*\mu,\underline{\nu})\to0\text{ as }t\to\8,
\]
and for all $\mu\geq_{\st}\overline{\nu}$,
\[
\WW_2(P_t^*\mu,\overline{\nu})\to0\text{ as }t\to\8.
\]
\end{enumerate}
\end{theorem}

\begin{corollary}\label{coro:unique-convergence}
\textnormal{(Global Convergence)} Under Assumption \ref{asp:potential}, if the equation \eqref{eq:mvsystem} has a unique invariant measure $\nu$, then for all bounded sets $B\subset\PP_2(\R)$, $\WW_2(P_t^*B,\nu)\to0$ as $t\to\8$.
\end{corollary}

\begin{definition}\label{def:zero-crossing}
The zero-crossing number of a continuous function $f:\R\to\R$, denoted by $Z(f)$, is the maximal integer $N$ such that, there exists an increasing sequence $x_1<x_2<\cdots<x_N<x_{N+1}$ satisfying $f(x_n)f(x_{n+1})<0$ for all $n=1,2,\dots,N$.
\end{definition}

\begin{definition}\label{def:locally dissipative}
The equation \eqref{eq:mvsystem} is called locally dissipative at $a\in\R$ with configuration $(r,g)$, if
\begin{enumerate}[label=(\roman*)]
\item the function $g:\R_+\times\R_+\to\R$ satisfies, for all $x\in\R$, $\mu\in\PP_2(\R)$,
\begin{equation*}
-2xV'(x+a)-2x(W'*\mu)(x)+\xs^2(x+a)\leq -g(x^2,\|\mu\|_2^2);
\end{equation*}
\item the number $r>0$ and the function $g$ satisfy
\begin{equation*}
\begin{split}
&g(\cdot,r^2) \text{ is continuous and convex};\\
&\inf_{0\leq w\leq r^2}g(z,w)=g(z,r^2);\\
&\inf_{z\geq r^2}g(z,r^2)>0.
\end{split}
\end{equation*}
\end{enumerate}
\end{definition}

\begin{remark}
It is straightforward to verify that Definition \ref{def:locally dissipative} provides an equivalent formulation of \cite[Definition 1.2]{Liu-Qu-Yao-Zhi2024} in the one-dimensional setting. 
The concept of local dissipation was introduced by Zhang \cite{Zhang2023}, and was later applied by Feng-Qu-Zhao \cite{Feng-Qu-Zhao2023b} to the time-inhomogeneous McKean-Vlasov equation and by Bao-Wang \cite{Bao-Wang2025} to the McKean–Vlasov SDEs with jumps.
\end{remark}

\begin{theorem}\label{thm:local-convergence}
Suppose Assumption \ref{asp:potential} holds, and suppose the zero-crossing number $Z(V')$ of $V':\R\to\R$ is $(2n-1)$.
If the equation \eqref{eq:mvsystem} is locally dissipative at $a_1<a_2<\cdots<a_n$ with configurations $\{(r_k,g_k)\}_{k=1}^n$, and if
\begin{equation}\label{eq:local-convergence-condition}
r_k+r_{k+1}\leq a_{k+1}-a_k, \ \text{ for all $k=1,2,\dots,n-1$},
\end{equation}
then
\begin{enumerate}[label=\textnormal{(\roman*)}]
\item \textnormal{(Total Orderedness)} there are exactly $(2n-1)$ invariant measures $\{\nu_k\}_{k=1}^{2n-1}$ with $\nu_1<_{\st}\nu_2<_{\st}\cdots<_{\st}\nu_{2n-1}$;
\item \textnormal{(Positively Invariant Open Neighbourhoods)} the open balls $\{B(\xd_{a_k},r_k)\}_{k=1}^n$ are pairwise disjoint, and there hold
\begin{equation}\label{eq:nu in ball}
\nu_{2k-1}\in B(\xd_{a_k},r_k), \ \text{ for all } k=1,2,\dots,n,
\end{equation}
and
\begin{equation}\label{eq:nu notin ball}
\nu_{2k}\notin \overline{B(\xd_{a_k},r_k)}\bigcup\overline{B(\xd_{a_{k+1}},r_{k+1})}, \ \text{ for all } k=1,2,\dots,n-1.
\end{equation}
Moreover, $P^*_tB(\xd_{a_k},r_k)\subset B(\xd_{a_k},r_k)$ for all $k=1,2,\dots,n$ and $t\geq 0$;
\item \textnormal{(Connecting Orbits)} for all $k=1,2,\dots,n-1$, there are a decreasing connecting orbit $\{\nu^{k,\downarrow}_t\}_{t\in \R}$ from $\nu_{2k}$ to $\nu_{2k-1}$ and an increasing connecting orbit $\{\nu^{k,\uparrow}_t\}_{t\in \R}$ from $\nu_{2k}$ to $\nu_{2k+1}$;
\item \textnormal{(Locally Attracting)} for all $k=1,2,\dots,n$,
\[
\WW_2(P_t^*B(\xd_{a_k},r_k),\nu_{2k-1})=\sup_{\mu\in B(\xd_{a_k},r_k)}\WW_2(P_t^*\mu,\nu_{2k-1})\to0\text{ as }t\to\8;
\]
\item \textnormal{(Attracting Unbounded Open Neighbourhoods)}
if $\mu\leq_{\st}\nu$ for some $\nu\in B(\xd_{a_1},r_1)$, then
\[
\WW_2(P_t^*\mu,\nu_1)\to0\text{ as }t\to\8.
\]
If $\mu\geq_{\st}\nu$ for some $\nu\in B(\xd_{a_n},r_n)$, then
\[
\WW_2(P_t^*\mu,\nu_{2n-1})\to0\text{ as }t\to\8.
\]
In particular,
\[
\WW_2(P_t^*\mu,\nu_1)\to0\text{ as }t\to\8\ \text{ for all $\mu\in\bigcup_{a\leq0}B(\xd_{a_1+a},r_1)$},
\]
\[
\WW_2(P_t^*\mu,\nu_{2n-1})\to0\text{ as }t\to\8\ \text{ for all $\mu\in\bigcup_{a\geq0}B(\xd_{a_n+a},r_n)$}.
\]
\end{enumerate}
\end{theorem}

\begin{remark}\label{R:invariant measures cannot attract bdd sets}
(i). By Assumption \ref{asp:potential} (i)(iii), the function $V':\R\to\R$ must have an odd zero-crossing number whenever finite.

(ii). 
By definition, a connecting orbit $\{\nu_t\}_{t\in\R}$ from $\nu$ to $\nutd$ as a bounded set itself cannot be attracted by its forward limit $\nutd$, i.e., $\WW_2(P_t^*\{\nu_s:s\in\R\},\nutd)\nrightarrow0$, although $\WW_2(P_t^*\nu_s,\nutd)\to0$ as $t\to\8$ for any $s\in\R$.
So, the set of  invariant measures cannot attract all bounded subsets of $\PP_2(\R)$.
In other word, the order interval $[\underline{\nu},\overline{\nu}]$ in Theorem \ref{thm:global-convergence} (ii) cannot be replaced by  the set of  invariant measures.
\end{remark}

We next present a corollary providing the order position of invariant measures and also explicit expressions for certain densities in the basins of attraction.
Consider a family of probability measures $\{\mu_m\}_{m\in\R}$ on $\R$, whose densities are given by
\begin{equation}\label{eq:density-function}
\dif{\mu_m}{x}\propto\frac{1}{\xs^2(x)}\exp\bigg\{-2\int_0^x\frac{V'(y)+\theta y-\theta m}{\xs^2(y)}\d y\bigg\}.
\end{equation}

\begin{corollary}\label{coro:segment-convergence}
Suppose the assumptions in Theorem \ref{thm:local-convergence} hold, and let $\{\mu_m\}_{m\in\R}$ be probability measures given in \eqref{eq:density-function}.
Then
\begin{enumerate}[label=\textnormal{(\roman*)}]
\item there are increasing values $m_k$, $k=1,2,\dots,2n-1$, such that $\nu_{k}=\mu_{m_k}$;
\item set $m_0=-\8$ and $m_{2n}=\8$.
For all $k=1,2,\dots,n$, and for all $m\in(m_{2k-2},m_{2k})$,
\[
\WW_2(P_t^*\mu_m,\nu_{2k-1})\to0\text{ as }t\to\8.
\]
\end{enumerate}
\end{corollary}

\begin{remark}
 For each $k=1,2,\dots,2n-2$, although the set $\{\mu_m\}_{m\in(m_k,m_{k+1})}$ is a totally ordered arc with endpoints $\nu_k$, $\nu_{k+1}$, it is not a connecting orbit between invariant measures.
One can check it from the Fokker-Planck equation.
\end{remark}

As an application of our abstract theorems, we present below the classical double-well and multi-well examples.
In contrast to Theorem \ref{T:general V}, the confining potential $V$ is given explicitly in these cases, which allows us to determine explicit parameter ranges and to provide  explicit descriptions of attracting sets.
On the other hand, Proposition \ref{prop:number-invariant-measure} provides a detailed analysis of the number of invariant measures under multiplicative noise, thereby verifying the a priori assumption in the main result of \cite{Liu-Qu-Yao-Zhi2024}. 
Based on this analysis, the subsequent examples not only address basins of attraction with respect to both the partial order and the topology, but also strengthen the results on connecting orbits for the corresponding examples considered in \cite{Liu-Qu-Yao-Zhi2024}. 
In particular, even under multiplicative noise, the classical double-well equation \eqref{eq:double-well} admits connecting orbits; for the multi-well equation \eqref{eq:multi-well} we identify exactly five invariant measures together with a precise heteroclinic structure connecting them via monotone connecting orbits.

\begin{theorem}[Double-well landscapes, Figure \ref{fig:double-well}]\label{thm:double-well}
Consider the following one-dimensional McKean-Vlasov SDE,
\begin{equation}\label{eq:double-well}
\d X_t=-\left[X_t(X_t-1)(X_t+1)+\theta\left(X_t-\bE [X_t]\right)\right]\d t+\xs(X_t)\d B_t.
\end{equation}
If the parameter $\theta$ and the Lipschitz function $\xs\colon\R\to\R$ satisfy
\begin{equation}\label{eq:con-double-well}
    \theta\geq\frac{27(9+\sqrt{17})}{128},\quad 0<\inf_{x\in\R}\xs^2(x)\leq \sup_{x\in\R}\xs^2(x)<\frac{51\sqrt{17}-107}{256},
\end{equation}
and let $r=\frac{9-\sqrt{17}}{8}$, then
\begin{enumerate}[label=\textnormal{(\roman*)}]
\item There are exactly three invariant measures, $\nu_{-1}, \nu_0, \nu_1\in\PP_{\8}(\R)$ with $\nu_{-1}<_{\st}\nu_0<_{\st}\nu_1$;
\item For all bounded sets $B\subset\PP_2(\R)$, $\WW_2(P_t^*B,[\nu_{-1},\nu_1])\to0$ as $t\to\8$;
\item There are a decreasing connecting orbit from $\nu_0$ to $\nu_{-1}$ and an increasing connecting orbit from $\nu_0$ to $\nu_1$;
\item There hold $\nu_{-1}\in B(\xd_{-1},r)$, $\nu_1\in B(\xd_1,r)$, and
\[
\WW_2(P_t^*\mu,\nu_{-1})\to0\ \text{ as $t\to\8$ for all $\mu\in\bigcup_{a\leq0}B(\xd_{-1+a},r)$},
\]
\[
\WW_2(P_t^*\mu,\nu_1)\to0\ \text{ as $t\to\8$ for all $\mu\in\bigcup_{a\geq0}B(\xd_{1+a},r)$}.
\]
\end{enumerate}
\end{theorem}

\begin{figure}[htbp]
\centering
\includegraphics[width=8cm]{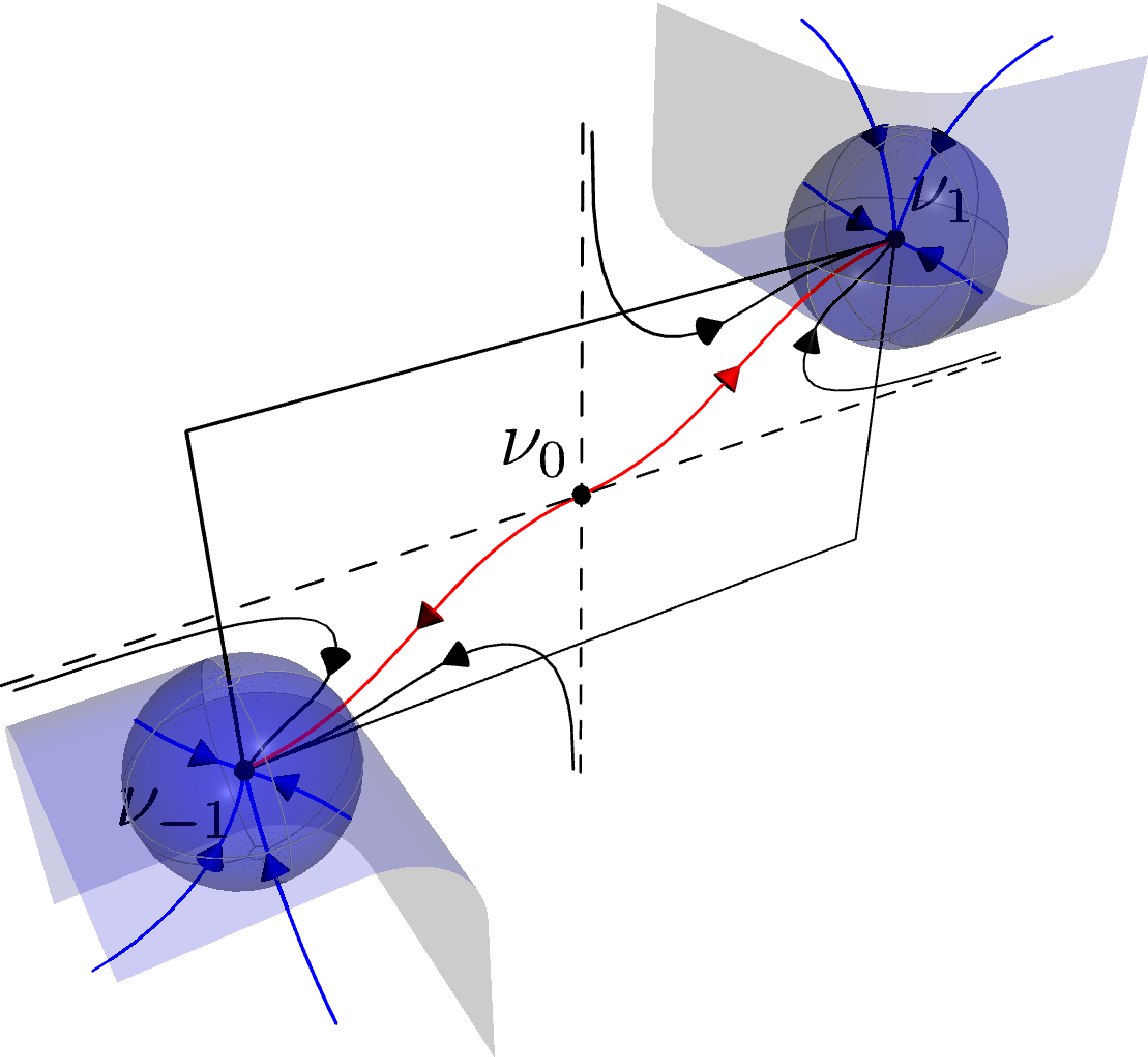}
\caption{Phase diagram for double-well landscapes}
\label{fig:double-well}
\end{figure}

\begin{theorem}[Multi-well landscapes, Figure \ref{fig:multi-well}]\label{thm:multi-well}
Consider the following one-dimensional McKean-Vlasov SDE,
\begin{equation}\label{eq:multi-well}
\d X_t=-\left[X_t(X_t-1)(X_t+1)(X_t-2)(X_t+2)+\theta\left(X_t-\bE [X_t]\right)\right]\d t+\xs(X_t)\d B_t.
\end{equation}
If the parameter $\theta$ and the Lipschitz function $\xs\colon\R\to\R$ satisfy
\[
\theta\geq8\sqrt{5+\sqrt{13}},\quad 0<\inf_{x\in\R}\xs^2(x)\leq \sup_{x\in\R}\xs^2(x)<\frac{4(13\sqrt{13}-35)}{27},
\]
and let $r=\frac{\sqrt{15-3\sqrt{13}}}{3}$, then 
\begin{enumerate}[label=\textnormal{(\roman*)}]
\item There are exactly five invariant measures, $\nu_{-2}, \nu_{-1}, \nu_0, \nu_1, \nu_2\in\PP_{\8}(\R)$ with $\nu_{-2}<_{\st}\nu_{-1}<_{\st}\nu_0<_{\st}\nu_1<_{\st}\nu_2$;
\item For all bounded sets $B\subset\PP_2(\R)$, $\WW_2(P_t^*B,[\nu_{-2},\nu_2])\to0$ as $t\to\8$;
\item There are a decreasing connecting orbit from $\nu_{-1}$ to $\nu_{-2}$ and an increasing connecting orbit from $\nu_{-1}$ to $\nu_0$;\\
there are a decreasing connecting orbit from $\nu_1$ to $\nu_0$ and an increasing connecting orbit from $\nu_1$ to $\nu_2$;
\item There hold $\nu_{-2}\in B(\xd_{-2},r)$, $\nu_0\in B(\xd_0,r)$, $\nu_2\in B(\xd_2,r)$, and
\[
\WW_2(P_t^*\mu,\nu_{-2})\to0\ \text{ as $t\to\8$ for all $\mu\in\bigcup_{a\leq0}B(\xd_{-2+a},r)$},
\]
\[
\WW_2(P_t^*\mu,\nu_{0})\to0\ \text{ as $t\to\8$ for all $\mu\in B(\xd_{0},r)$},
\]
and
\[
\WW_2(P_t^*\mu,\nu_2)\to0\ \text{ as $t\to\8$ for all $\mu\in\bigcup_{a\geq0}B(\xd_{2+a},r)$}.
\]
\end{enumerate}
\end{theorem}

\begin{figure}[htbp]
\centering
\includegraphics[width=8cm]{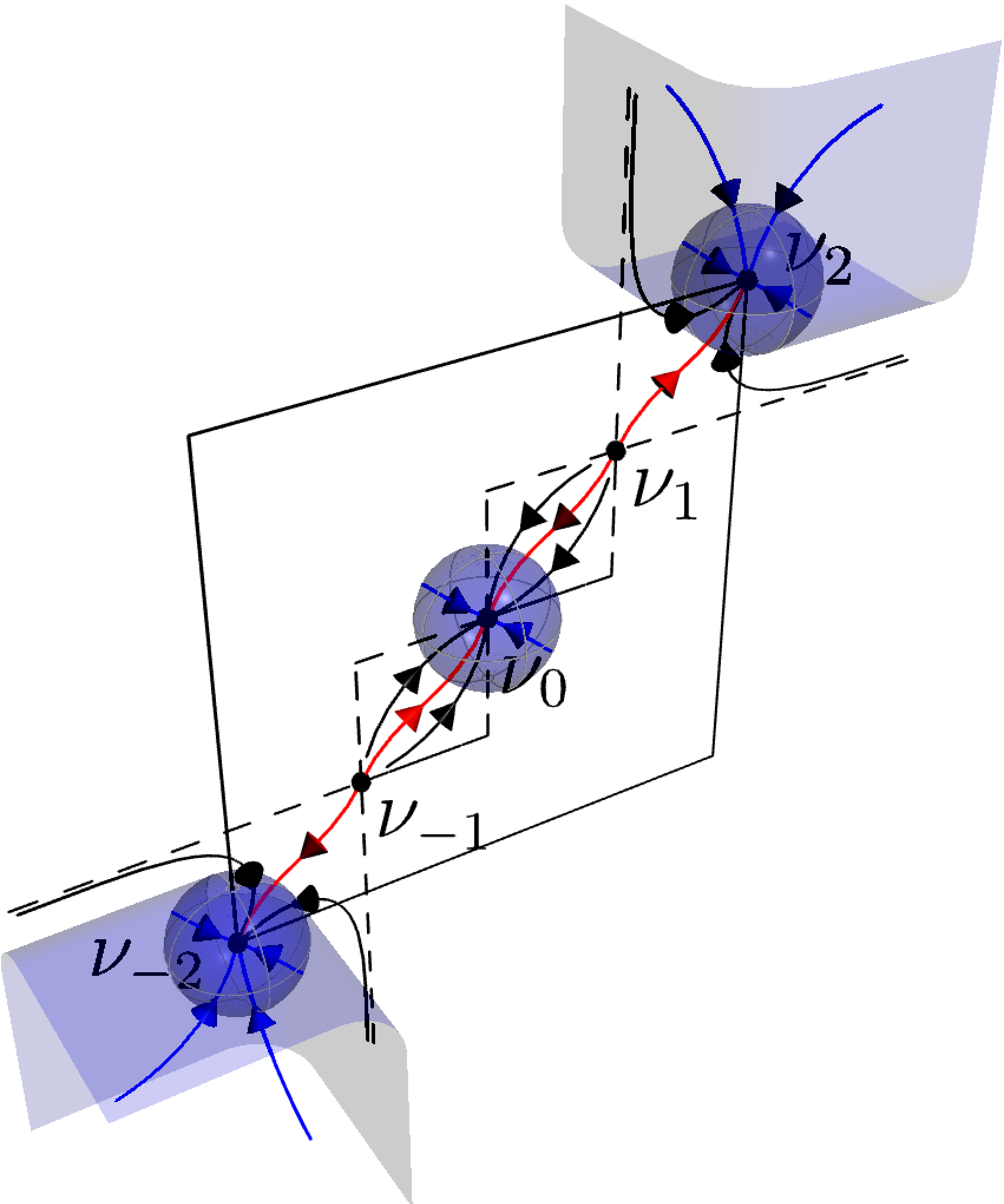}
\caption{Phase diagram for multi-well landscapes}
\label{fig:multi-well}
\end{figure}

Owing to our result on the number of invariant measures (Proposition \ref{prop:number-invariant-measure}), our framework is also capable of handling the case where $V'$ vanishes on an interval.

\begin{theorem}[Double-well with flat bottoms, Figure \ref{fig:double-well}]\label{thm:double-well-vanish}
Consider the following one-dimensional McKean-Vlasov SDE,
\begin{equation}\label{eq:double-well-vanish}
\d X_t=-\left[V'(X_t)+\theta\left(X_t-\bE [X_t]\right)\right]\d t+\xs(X_t)\d B_t.
\end{equation}
where
\[
V'(x)=
\begin{cases}
    x\big(x-\frac{31}{32}\big)\big(x+\frac{31}{32}\big), & x\in \big(-\frac{31}{32},\frac{31}{32}\big),\\
    0, & x\in \big(-1,-\frac{31}{32}\big]\cup\big[\frac{31}{32},1\big),\\
    x(x-1)(x+1), & x\in (-\8,-1]\cup [1, \8).
\end{cases}
\]
If the parameter $\theta$ and the Lipschitz function $\xs\colon\R\to\R$ satisfy
\begin{equation}\label{eq:para-double-well-vanish}
    \theta\geq\frac{27(9+\sqrt{17})}{128},\quad 0<\inf_{x\in\R}\xs^2(x)\leq \sup_{x\in\R}\xs^2(x)<\frac{51\sqrt{17}-171}{256},
\end{equation}
and let $r=\frac{9-\sqrt{17}}{8}$, then the exactly same results as in Theorem \ref{thm:double-well} are valid.
\end{theorem}

\section{Properties of Stochastic Order}\label{sec:stochastic-order}

In this section, some new properties of the stochastic order are examined. 
We first summarize some preliminary materials involving stochastic order in $\PP_2(\R)$. Then we give an equivalent characterization of order boundedness in $\PP_2(\R)$ (see Proposition \ref{prop:order-bounded-equivalent}), which plays a crucial role in the proof of our main results.
As a consequence, we show any bounded subset of $\PP_p(\R)$, $p>2$, is order bounded in $\PP_2(\R)$ (see Corollary \ref{coro:bounded-order-bounded}), and any order bounded subset of $\PP_2(\R)$ is relatively compact (see Proposition \ref{prop:order-bounded-precompact}).

\subsection{Preliminaries of stochastic order}\label{sec:preliminaries of stochastic-order}

Let $\PP(\R^d)$ be the set of probability measures on $\R^d$, and let $\PP_p(\R^d)$ be the set of those having finite $p$-th moments,
\[
\PP_p(\R^d):=\left\{\mu\in\PP(\R^d): \norm{\mu}_p=\left(\int_{\R^d}\abs x^p\d\mu(x)\right)^{1/p}<\8\right\}.
\]
For $p\in[1,\8)$, equip $\PP_p(\R^d)$ with the $p$-Wasserstein metric defined as below,
\[
\WW_p(\mu,\nu):=\inf\Big\{\big(\bE[\abs{X-Y}^p]\big)^{1/p}:\LL(X)=\mu,\, \LL(Y)=\nu\Big\},
\]
where $\LL(\cdot)$ represents the law of a random variable.
Then $(\PP_p(\R^d),\WW_p)$ is a complete metric space (see Villani \cite[Theorem 6.18]{Villani2009}).

Next, we review the so-called ``stochastic order'', which is also known as stochastic domination; we refer to,
e.g., Kamae-Krengel-O’Brien \cite{Kamae1977}, Lindvall \cite[Chapter IV]{Lindvall1992} for a detailed discussion.
Give $\R^d$ an order relation by coordinate-wise comparison, i.e., $x\leq y$ if and only if $x_i\leq y_i$ for all $i=1,2,\dots,d$. 
For $\mu,\nu\in\PP_2(\R^d)$, the stochastic order is defined by the following,
\[
\mu\leq_{\st}\nu\text{ if and only if }\int_{\R^d}f\d\mu\leq\int_{\R^d}f\d\nu\text{ for all bounded increasing functions $f\colon\R^d\to\R$},
\]
where ``increasing'' means $f(x)\leq f(y)$ whenever $x\leq y$.
Then $\leq_{\st}$ is a partial order on $\PP_2(\R^d)$.
Indeed, the reflexivity and transitivity are obvious, and the antisymmetry can be obtained by taking $f=\one_{(x,\8)^d}$ for any $x\in \R^d$. 
Moreover, as shown in \cite[Lemma 2.1]{Liu-Qu-Yao-Zhi2024}, the stochastic order is a closed partial order relation in $\PP_2(\R^d)$, namely, the subset $\{(\mu,\nu):\mu\leq_{\st}\nu\}$ is closed in the product space $\PP_2(\R^d)\times\PP_2(\R^d)$. That is to say, $\mu\leq_{\st}\nu$ whenever $\WW_2(\mu_n,\mu)\to0$ and $\WW_2(\nu_n,\nu)\to0$ as $t\to\8$ and $\mu_n\leq_{\st}\nu_n$ for all $n$.

If $\mu\leq_{\st}\nu$ and $\mu\neq\nu$, we write $\mu<_{\st}\nu$. 
The reversed signs $\geq_{\st}$, $>_{\st}$ are used in the usual way. 
For $\mu,\nu\in\PP_2(\R^d)$ with $\mu\leq_{\st}\nu$, the \emph{order interval} is defined by $[\mu,\nu]=\{\rho\in\PP_2(\R^d):\mu\leq_{\st} \rho\leq_{\st} \nu\}$, $(\mu,\nu]=\{\rho\in\PP_2(\R^d):\mu<_{\st} \rho\leq_{\st} \nu\}$  and similarly we can define $[\mu,\nu)$ and $(\mu,\nu)$.
Measures $\mu,\nu\in\PP_2(\R^d)$ are said to be \emph{order related}, if $\mu\leq_{\st}\nu$ or $\nu\leq_{\st}\mu$ holds.
Otherwise, they are \emph{order unrelated}.
Given two subsets $A$ and $B$ of $\PP_2(\R^d)$, we write $A\leq_{\st} B$ ($A<_{\st}B$) when $\mu\leq_{\st}\nu$ ($\mu<_{\st}\nu$) holds for each choice of $\mu\in A$ and $\nu\in B$.
Sets $A, B\subset\PP_2(\R^d)$ are said to be \emph{order related}, if $A\leq_{\st} B$ or $B\leq_{\st} A$ holds.
A \emph{totally ordered set} (which is also referred as \emph{chain}) $A$ in $\PP_2(\R^d)$ means that, any two measures $\mu$ and $\nu$ in $A$  are order related.

A measure $\mu\in\PP_2(\R^d)$ is an \emph{upper bound} (\emph{lower bound}) of a set $A\subset\PP_2(\R^d)$ if $\mu\geq_{\st} A$ ($\mu\leq_{\st} A$).
If the set $A$ has an upper (lower) bound, it is said to be \emph{bounded from above} (\emph{below}).
A set $A$ which is bounded from both above and below is said to be \emph{order bounded}.
An upper bound $\mu_0$ of $A$ is said to be the \emph{supremum} of $A$, denoted by $\mu_0=\sup_{\leq_{\st}}A$, if any other upper bound $\mu$ satisfies $\mu\geq_{\st} \mu_0$. 
Similarly, a lower bound $\mu_0$ of $A$ is said to be the \emph{infimum} of $A$, denoted by $\mu_0=\inf_{\leq_{\st}}A$, if any other lower bound $\mu$ satisfies $\mu\leq_{\st} \mu_0$. 
By the antisymmetry of the stochastic order, if the supremum (infimum) exists, it must be unique. 

There are some complementary material of the stochastic order that we will use later.
Firstly, a useful equivalent characterization of the stochastic order is Strassen's theorem (see e.g., Lindvall \cite[Theorem IV.2.4]{Lindvall1992} or \cite[equation (3)]{Lindvall1999}), which says $\mu\leq_{\mathrm{st}}\nu$ if and only if there exist two random variables $X,Y$ such that
\[
\LL(X)=\mu,\quad\LL(Y)=\nu,\quad X\leq Y\ \as
\]
Secondly, according to Hiai-Lawson-Lim \cite[Proposition 3.11]{Hiai2018}, test functions $f$ in Definition \ref{def:stochastic-order} can be restricted as bounded continuous increasing.
If $\mu$, $\nu$ both have finite first moments, it is equivalent for test functions ranging over all $1$-Lipschitz increasing functions (see Fritz-Perrone \cite[Theorem 4.2.1]{Fritz2020}), where a $1$-Lipschitz function $f$ means $\abs{f(x)-f(y)}\leq\abs{x-y}$ for all $x,y\in\R^d$.
We may confine test functions in such different sets whenever convenient.
Additionally, an important observation regarding the definition of $\mu\leq_{\st}\nu$, is that when we take the test function as $f=-\one_{(-\infty,x]^d}$ or $f=\one_{(x,\infty)^d}$, we have $\mu\big((-\infty,x]^d\big)\geq\nu\big((-\infty,x]^d\big)$ and $\mu\big((x,\infty)^d\big)\leq\nu\big((x,\infty)^d\big)$, for any $x\in \R^d$.

\subsection{Equivalent characterization of order boundedness}\label{subsec:order-bounded-equivalent}
We begin with  an equivalent characterization of the stochastic order in $\PP_2(\R)$ by comparison of cumulative distribution functions.

\begin{lemma}\label{lem:stochastic-order-distribution}
Let $\mu_1,\mu_2\in\PP_2(\R)$ be two probability measures on $\R$, and let $F_1,F_2:\R\to[0,1]$ be their cumulative distribution functions.
Then $\mu_1\leq_{\st}\mu_2$ if and only if $F_1(x)\geq F_2(x)$ for all $x\in\R$.
\end{lemma}

\begin{proof}
($\Rightarrow$)
Note that $F_i(x)=\mu_i\big((-\infty,x]\big)$, $i=1,2$.
Take the test function $f=-\one_{(-\infty,x]}$ in the definition of $\mu_1\leq_{\st}\mu_2$, it follows that $F_1(x)\geq F_2(x)$, for any $x\in \R$.

($\Leftarrow$) 
By Hiai-Lawson-Lim \cite[Proposition 3.11]{Hiai2018}, it suffices to show $\int f\d\mu_1\leq\int f\d\mu_2$ for all bounded continuous increasing functions $f$.
Let $f:\R\to\R$ be such a function, and we may assume $f$ is non-negative.
Set
\[
f_n:=\sum_{k=1}^{n2^n}\frac1{2^n}\one_{\{f>k/2^n\}},
\]
and we have $f_n\up f$ as $n\to\8$.
Note that the set $\{f>k/2^n\}$ is an open interval of the form $(x_k,\8)$, where $(\8,\8):=\emptyset$, since $f$ is increasing continuous.
So we have 
\begin{equation*}
    \int f_n\d \mu_i=\sum_{k=1}^{n2^n}\frac{1}{2^n}\mu_i\big((x_k,\8)\big), \ \ i=1,2.
\end{equation*}
By the assumption that $F_1(x)\geq F_2(x)$ for all $x\in \R$, we conclude that for any $x\in \R$,
\begin{equation*}
\mu_1\big((x,\8)\big)=1-F_1(x)\leq 1-F_2(x)=\mu_2\big((x,\8)\big).
\end{equation*}
Then
\[
\int_{\R}f_n\d\mu_1\leq\int_{\R}f_n\d\mu_2\quad\text{ for all $n\geq 1$}.
\]
The monotone convergence theorem gives $\int f\d\mu_1\leq\int f\d\mu_2$.
\end{proof}

\begin{remark}\label{rem:lattice}
Lemma \ref{lem:stochastic-order-distribution} implies the lattice property holds in $\PP_2(\R)$, that is to say, $\sup_{\leq_{\st}}\{\mu_1,\mu_2\}$ exists for any two probability measures $\mu_1, \mu_2\in\PP_2(\R)$.
However, there is no lattice property in $\PP_2(\R^d)$ for $d\geq2$.
See counterexamples in  Kamae-Krengel-O’Brien \cite{Kamae1977} or  M\"{u}ller-Scarsini \cite{Muller2006}.
\end{remark}

We next give an equivalent characterization of order boundedness in $\PP_2(\R)$, which turns out to be very crucial for our approach.

\begin{proposition}\label{prop:order-bounded-equivalent}
Let $M$ be a subset of $\PP_2(\R)$.
Then the following statements are true.
\begin{enumerate}[label=\textnormal{(\roman*)}]
\item $M$ is bounded from above in $\PP_2(\R)$ with respect to the stochastic order if and only if
\begin{equation}\label{eq:order-bounded-equivalent-1}
\int_0^{\8}\bigg(\sup_{\mu\in M}\mu\big((u,\8)\big)\bigg)u\d u<\8.
\end{equation}
\item $M$ is bounded from below in $\PP_2(\R)$ with respect to the stochastic order if and only if
\begin{equation}\label{eq:order-bounded-equivalent-2}
\int_0^{\8}\bigg(\sup_{\mu\in M}\mu\big((-\8,-u]\big)\bigg)u\d u<\8.
\end{equation}
\end{enumerate}
\end{proposition}

\begin{proof}
We only prove (i), and the proof of (ii) is in a similar fashion.

($\Rightarrow$)
Suppose $M\leq_{\st}\overline{\mu}$ for some $\overline{\mu}\in\PP_2(\R)$.
Then by the definition of stochastic order, we have for all $u\in\R$,
\[
\sup_{\mu\in M}\mu\big((u,\8)\big)\leq\overline{\mu}\big((u,\8)\big).
\]
So, it follows
\begin{equation}\label{pf:order-bounded-equivalent-1}
\int_0^{\8}\bigg(\sup_{\mu\in M}\mu\big((u,\8)\big)\bigg)u\d u
\leq\int_0^{\8}\overline{\mu}\big((u,\8)\big)u\d u
=\int_0^{\8}\int_{\R}\one_{\{x>u\}}\d\overline{\mu}(x)u\d u.
\end{equation}
By Tonelli's theorem, we switch the integrals and obtain
\begin{equation}\label{pf:order-bounded-equivalent-2}
\int_0^{\8}\int_{\R}\one_{\{x>u\}}\d\overline{\mu}(x)u\d u=\int_{\R}\one_{\{x>0\}}\int_0^xu\d u\d\overline{\mu}(x)=\frac12\int_{\R}x^2\one_{\{x>0\}}\d\overline{\mu}(x)<\8.
\end{equation}
Combining \eqref{pf:order-bounded-equivalent-1} and \eqref{pf:order-bounded-equivalent-2} finishes the ``only if'' part.

($\Leftarrow$)
Note $\sup_{\mu\in M}\mu\big((u,\8)\big)$ is decreasing in $u$, and then by condition \eqref{eq:order-bounded-equivalent-1}, we have
\begin{equation}\label{pf:order-bounded-equivalent-3}
\sup_{\mu\in M}\mu\big((u,\8)\big)\to0\ \text{ as $u\to\8$}.
\end{equation}
Let $F_{\mu}$ be the cumulative distribution function of $\mu$, and define a function $F:\R\to[0,1]$ by
\[
F(x):=\inf_{\mu\in M}F_{\mu}(x).
\]
We want to show $F$ is also a cumulative distribution function.
It is obvious that $F$ is increasing, and $\lim_{x\to-\8}F(x)=0$.
Also, \eqref{pf:order-bounded-equivalent-3} gives $\lim_{x\to\8}F(x)=1$.
Since $F_{\mu}$ is right continuous for every $\mu\in M$, we claim $F$ is also right continuous.
In fact, by the monotonicity of $F$ and right continuity of $F_{\mu}$, we have
\[
F(x)\leq\lim_{x_n\down x}F(x_n)=\lim_{x_n\down x}\inf_{\mu\in M}F_{\mu}(x_n)\leq\inf_{\mu\in M}\lim_{x_n\down x}F_{\mu}(x_n)=\inf_{\mu\in M}F_{\mu}(x)=F(x).
\]
Thus, $F$ is a cumulative distribution function of some probability measure $\nu$.
Then, we observe for all $x\in\R$,
\begin{equation}\label{pf:order-bounded-equivalent-4}
\nu\big((x,\8)\big)=1-F(x)=1-\inf_{\mu\in M}\mu\big((-\8,x]\big)=\sup_{\mu\in M}\big(1-\mu\big((-\8,x]\big)\big)=\sup_{\mu\in M}\mu\big((x,\8)\big).
\end{equation}

Now we show $\nu\in\PP_2(\R)$.
Split the second moment of $\nu$ into two parts,
\begin{equation}\label{pf:order-bounded-equivalent-5}
\int_{\R}\abs x^2\d\nu(x)=\int_{\R}\abs x^2\one_{\{x\leq0\}}\d\nu(x)+\int_{\R}\abs x^2\one_{\{x>0\}}\d\nu(x).
\end{equation}
On the one hand, using Tonelli's theorem and \eqref{pf:order-bounded-equivalent-4}, we see for some fixed $\mu\in M\subset\PP_2(\R)$,
\begin{equation}\label{pf:order-bounded-equivalent-6}
\begin{aligned}
\int_{\R}\abs x^2\one_{\{x\leq0\}}\d\nu(x)
&=2\int_{\R}\int_{\R}(-u)\one_{\{x\leq u\leq0\}}\d u\cdot\one_{\{x\leq0\}}\d\nu(x)\\
&=2\int_{\R}\int_{\R}\one_{\{x\leq u\}}\d\nu(x)\cdot(-u)\one_{\{u\leq0\}}\d u\\
&=2\int_{\R}\nu\big((-\8,u]\big)\cdot(-u)\one_{\{u\leq0\}}\d u\\
&\leq2\int_{\R}\mu\big((-\8,u]\big)\cdot(-u)\one_{\{u\leq0\}}\d u\\
&=\int_{\R}\abs x^2\one_{\{x\leq0\}}\d\mu(x)<\8.
\end{aligned}
\end{equation}
On the other hand, use Tonelli's theorem again,
\begin{equation}\label{pf:order-bounded-equivalent-7}
\int_{\R}\abs x^2\one_{\{x>0\}}\d\nu(x)
=2\int_{\R}\int_{\R}u\one_{\{0<u<x\}}\d u\cdot\one_{\{x>0\}}\d\nu(x)
=2\int_{\R}\nu\big((u,\8)\big)\cdot u\one_{\{u>0\}}\d u.
\end{equation}
Combine \eqref{eq:order-bounded-equivalent-1}, \eqref{pf:order-bounded-equivalent-4}, \eqref{pf:order-bounded-equivalent-7}, and we see
\begin{equation}\label{pf:order-bounded-equivalent-8}
\int_{\R}\abs x^2\one_{\{x>0\}}\d\nu(x)
=2\int_0^{\8}\bigg(\sup_{\mu\in M}\mu\big((x,\8)\big)\bigg)\d u<\8.
\end{equation}
It follows from \eqref{pf:order-bounded-equivalent-5}, \eqref{pf:order-bounded-equivalent-6} and \eqref{pf:order-bounded-equivalent-8} that $\int_{\R}\abs x^2\d\nu(x)<\8$, so $\nu\in\PP_2(\R)$.

It remains to show that $\nu\geq_{\st}M$.
Since $F_{\nu}=F=\inf_{\mu\in M}F_{\mu}$, we have $F_{\nu}\leq F_{\mu}$ for all $\mu\in M$.
Lemma \ref{lem:stochastic-order-distribution} yields that $\nu\geq_{\st}\mu$ for all $\mu\in M$, which completes the proof.
\end{proof}

\begin{remark}\label{rem:order-sup-inf}
Actaully, the condition \eqref{eq:order-bounded-equivalent-1} guarantees the existence of $\sup_{\leq_{\st}}M$ in $\PP_2(\R)$.
Let $\nu\in\PP_2(\R)$ be the probability measure constructed in the proof of Proposition \ref{prop:order-bounded-equivalent}, and we can show $\nu=\sup_{\leq_{\st}}M$.
Suppose $\nutd\geq_{\st}M$, and we need to show $\nutd\geq_{\st}\nu$.
Using Lemma \ref{lem:stochastic-order-distribution}, we have $F_{\nutd}\leq F_{\mu}$ for all $\mu\in M$.
Hence, $F_{\nutd}\leq\inf_{\mu\in M}F_{\mu}=F_{\nu}$, then we apply Lemma \ref{lem:stochastic-order-distribution} again to obtain $\nutd\geq_{\st}\nu$.
This verifies $\nu=\sup_{\leq_{\st}}M$.
Likewise, the condition \eqref{eq:order-bounded-equivalent-2} implies the existence of $\inf_{\leq_{\st}}M$ in $\PP_2(\R)$.
\end{remark}

\begin{remark}
    We remark here that in higher dimensional case, there is also an equivalent characterization of order boundedness for a subset of $\PP_2(\R^d)$: 

    \emph{A subset $M\subset \PP_2(\R^d)$ is order bounded from above (below) in $\PP_2(\R^d)$ if and only if the marginals of $M$ on each component are order bounded from above (below).}

    We give a brief proof here. 
    ($\Rightarrow$) is obvious, and we only show ($\Leftarrow$).
We will find a lower bound $\nu\in\PP_2(\R^d)$ of $M$, and the construction of an upper bound is similar, which is omitted.

For $x\in\R^d$, write $\xht=\min\{x_1,\dots,x_d\}$.
Define a function $\xf:\R^d\to\R^d$ by $x\mapsto(\xht,\dots,\xht)$.
If $\mu\in\PP_2(\R^d)$, define $\muht:=\mu\circ\xf^{-1}$.
Then it is easy to see $\muht\leq_{\st}\mu$, and that $\muht$ is supported on the diagonal $\xD=\{x\in\R^d:x_1=\cdots=x_d\}$.
Since this diagonal $\xD$ is a totally ordered line in $\R^d$, we can regard $\muht$ as a probability measure on a one-dimensional space.
That is to say, for $u\in\R$,
\begin{equation}\label{pf:higher-dimension-order-bounded-equivalent-1}
\muht\big((-\8,u]\big)=\mu\big(\R^d\setminus(u,\8)^d\big)\leq\sum_{i=1}^d\mu\big(\R^{i-1}\times(-\8,u]\times\R^{d-i}\big).
\end{equation}
Given that marginals of $M$ on each component are order bounded below, Proposition \ref{prop:order-bounded-equivalent} (ii) implies, for each $i=1,2,\dots,d$,
\begin{equation}\label{pf:higher-dimension-order-bounded-equivalent-2}
\int_0^{\8}\bigg(\sup_{\mu\in M}\mu\big(\R^{i-1}\times(-\8,-u]\times\R^{d-i}\big)\bigg)u\d u<\8.
\end{equation}
Combining \eqref{pf:higher-dimension-order-bounded-equivalent-1}\eqref{pf:higher-dimension-order-bounded-equivalent-2} yields
\[
\int_0^{\8}\bigg(\sup_{\mu\in M}\muht\big((-\8,-u]\big)\bigg)u\d u\leq\sum_{i=1}^d\int_0^{\8}\bigg(\sup_{\mu\in M}\mu\big(\R^{i-1}\times(-\8,-u]\times\R^{d-i}\big)\bigg)u\d u<\8,
\]
so there is some $\nu\in\PP_2(\xD)$ such that $\nu\leq_{\st}\muht$ for all $\mu\in M$.
Now we observe $\nu\in\PP_2(\R^d)$, and $\nu\leq_{\st}\muht\leq_{\st}\mu$ for all $\mu\in M$.
This finishes the proof.
\end{remark}

\begin{corollary}\label{coro:bounded-order-bounded}
If $M$ is a bounded subset of $\PP_p(\R)$ for some $p>2$, then $M$ is order bounded in $\PP_2(\R)$.
\end{corollary}

\begin{proof}
It is sufficient to check the conditions \eqref{eq:order-bounded-equivalent-1} and \eqref{eq:order-bounded-equivalent-2}.
We first notice
\begin{equation}\label{pf:order-bounded-1}
\int_0^{\8}\bigg(\sup_{\mu\in M}\mu\big((u,\8)\big)\bigg)u\d u
=\int_0^1\bigg(\sup_{\mu\in M}\mu\big((u,\8)\big)\bigg)u\d u
+\int_1^{\8}\bigg(\sup_{\mu\in M}\mu\big((u,\8)\big)\bigg)u\d u.
\end{equation}
Then we estimate these two terms separately. Note that
\begin{equation}\label{pf:order-bounded-2}
\int_0^1\bigg(\sup_{\mu\in M}\mu\big((u,\8)\big)\bigg)u\d u
\leq\int_0^1u\d u=\frac12,
\end{equation}
and by Chebyshev's inequality, for $u\geq1$,
\[
\mu\big((u,\8)\big)\leq\frac{1}{u^p}\int_{\R}\abs{x}^p\d\mu(x).
\]
It follows
\begin{equation}\label{pf:order-bounded-3}
\int_1^{\8}\bigg(\sup_{\mu\in M}\mu\big((u,\8)\big)\bigg)u\d u
\leq\sup_{\mu\in M}\norm{\mu}_p^p\cdot\int_1^{\8}\frac{\d u}{u^{p-1}}<\8.
\end{equation}
Combining \eqref{pf:order-bounded-1}-\eqref{pf:order-bounded-3} verifies \eqref{eq:order-bounded-equivalent-1}. The condition \eqref{eq:order-bounded-equivalent-2} can be verified in a similar way.
\end{proof}

By the equivalent characterization of order bounded sets in $\PP_2(\R)$, we also have the following result.

\begin{proposition}\label{prop:order-bounded-precompact}
If $M$ is an order bounded subset of $\PP_2(\R)$, then $M$ is relatively compact in $\PP_2(\R)$.
In particular, any order interval $[\mu,\nu]$ is compact in $\PP_2(\R)$.
\end{proposition}

\begin{proof}
According to Ambrosio-Gigli \cite[Theorem 2.7]{Ambrosio-Gigli2013}, it suffices to show that $M$ is tight and $2$-uniformly integrable, i.e.,
\begin{equation}\label{pf:order-bounded-precompact-1}
\lim_{N\to\8}\sup_{\mu\in M}\int_{\{x:|x|\geq N\}}|x|^2\d \mu(x)=0,
\end{equation}
    
Since $\sup_{\mu\in M}\mu\big((u,\8)\big)$ is decreasing in $u$, and by Proposition \ref{prop:order-bounded-equivalent} (i), we have
\begin{equation}\label{pf:order-bounded-precompact-2}
\sup_{\mu\in M}\mu\big((u,\8)\big)\to0\ \text{ as $u\to\8$}.
\end{equation}
Similarly, we observe $\sup_{\mu\in M}\mu\big((-\8,-u]\big)$ is decreasing in $u$, so Proposition \ref{prop:order-bounded-equivalent} (ii) entails
\begin{equation}\label{pf:order-bounded-precompact-3}
\sup_{\mu\in M}\mu\big((-\8,-u]\big)\to0\ \text{ as $u\to\8$}.
\end{equation}
Then \eqref{pf:order-bounded-precompact-2} and \eqref{pf:order-bounded-precompact-3} together imply $M$ is tight.
    
It remains to show that $M$ is 2-uniformly integrable, i.e., \eqref{pf:order-bounded-precompact-1} holds.
Let $\underline{\mu}, \overline{\mu}\in\PP_2(\R)$ be a lower and an upper bounds of $M$ with respect to the stochastic order.
Note that for any $\mu\in M$, by choosing $f^+_N(x)=|x|^2\one_{\{x:x\geq N\}}$ and $f^-_N(x)=-|x|^2\one_{\{x:x\leq -N\}}$ respectively, we deduce
\[
\int_{\{x:x\geq N\}}|x|^2\d \mu(x)\leq \int_{\{x:x\geq N\}}|x|^2\d \overline{\mu}(x), \ \text{ and } \ \int_{\{x:x\leq -N\}}|x|^2\d \mu(x)\leq \int_{\{x:x\leq -N\}}|x|^2\d \underline{\mu}(x).
\]
Thus, for all $\mu\in M$,
\[
\int_{\{x:|x|\geq N\}}|x|^2\d \mu(x)\leq \int_{\{x:x\geq N\}}|x|^2\d \overline{\mu}(x)+\int_{\{x:x\leq -N\}}|x|^2\d \underline{\mu}(x).
\]
Then we have
\[
\lim_{N\to\8}\sup_{\mu\in M}\int_{\{x:|x|\geq N\}}|x|^2\d \mu(x)\leq \lim_{N\to\8}\Bigg(\int_{\{x:x\geq N\}}|x|^2\d \overline{\mu}(x)+\int_{\{x:x\leq -N\}}|x|^2\d \underline{\mu}(x)\Bigg)=0.
\]
Hence, $M$ is $2$-uniformly integrable in $\PP_2(\R)$.

In particular, an order interval $[\mu,\nu]$ is closed due to the closedness of the stochastic order, so it is compact in $\PP_2(\R)$.
\end{proof}

\begin{remark}\label{R:compact not imply bounds}
If $M$ is merely a relatively compact subset of $\PP_2(\R)$, then $M$ may not have upper or lower bounds in $\PP_2(\R)$ with respect to the stochastic order.
In the following, we give an example of a relatively compact subset of $\PP_2(\R)$, but it fails to satisfy the equivalent conditions \eqref{eq:order-bounded-equivalent-1}\eqref{eq:order-bounded-equivalent-2} of order boundedness.
Let $M=\{\mu_n: n\geq2\}\subset\PP_2(\R)$ be given by
\[
\mu_n=\frac1{n^2\log n}\xd_{-n}+\left(1-\frac2{n^2\log n}\right)\xd_0+\frac1{n^2\log n}\xd_n.
\]
Since $\WW_2(\mu_n,\xd_0)=(2/\log n)^{1/2}\to0$ as $n\to\8$, the set $\{\mu_n\}_{n=2}^{\8}$ is relatively compact in $\PP_2(\R)$.
Note that
\begin{align*}
\inf_{n\geq2}F_{\mu_n}(x)=
\left\{
\begin{aligned}
&0, &&x<0,\\
&1-\frac{1}{2^2\log2}, &&0\leq x<2,\\
&1-\frac{1}{k^2\log k}, &&k-1\leq x<k,\ k\geq3.
\end{aligned}
\right.
\end{align*}
It gives
\[
\int_{0}^{\8}\bigg(\sup_{n\geq2}\mu_n\big((u,\8)\big)\bigg)u\d u=\int_0^{\8}\big[1-\inf_{n\geq2}F_{\mu_n}(u)\big]u\d u\geq\sum_{k=2}^{\8}\frac{k}{(k+1)^2\log(k+1)}=\8,
\]
so the condition \eqref{eq:order-bounded-equivalent-1} does not hold, and $M$ has no upper bound in $\PP_2(\R)$ with respect to the stochastic order.
Likewise, we can check $M$ does not satisfy \eqref{eq:order-bounded-equivalent-2}, and hence, it has no lower bound.
\end{remark}

\begin{corollary}
\label{coro:order-interval-metric-equivalence}
Let $[\mu,\nu]$ be an order interval and $\{\rho_n\}_{n\geq1}\subset[\mu,\nu]$, $\rho\in[\mu,\nu]$. Then the following convergences are equivalent.
\begin{enumerate}[label=\textnormal{(\roman*)}]
\item $\WW_2(\rho_n,\rho)\to0$ as $n\to\8$;
\item $\WW_p(\rho_n,\rho)\to0$ as $n\to\8$, for all $1\leq p<2$;
\item $\rho_n\to\rho$ weakly as $n\to\8$.
\end{enumerate}
\end{corollary}

\begin{proof}
It is obvious that (i)$\Rightarrow$(ii)$\Rightarrow$(iii), and we only need to show (iii)$\Rightarrow$(i).

By Proposition \ref{prop:order-bounded-precompact}, we have $[\mu,\nu]$ is compact in $\PP_2(\R^d)$. Then for any subsequence $\{\rho_{n_k}\}_{k\geq 1}\subset \{\rho_n\}_{n\geq 1}\subset [\mu,\nu]$, there exists a subsequence $\{\rho_{n_{k_j}}\}_{j\geq 1}$ of  $\{\rho_{n_k}\}_{k\geq 1}$ and a $\rho_0\in \PP_2(\R^d)$ such that $\WW_2(\rho_{n_{k_j}},\rho_0)\to0$ as $j\to\8$. Thus, $\rho_{n_{k_j}}\to\rho_0$ weakly as $j\to\8$. By (iii) and the uniqueness of the weak limit, we conclude $\rho_0=\rho$. Therefore, $\WW_2(\rho_n,\rho)\to0$ as $n\to\8$.
\end{proof}

\section{Global Convergence to Order Interval}\label{sec:global-convergence}
This section aims to prove Theorem \ref{thm:global-convergence} and Corollary \ref{coro:unique-convergence}.
In Section \ref{subsec:global attractor}, we target on the structure of global attractors of general semigroups on $\PP_2(\R)$ that preserve the stochastic order  (namely, order-preserving semigroups).
The coming down from infinity property of hyper-dissipative McKean-Vlasov SDEs is shown in Section \ref{subsec:coming-down-from-infinity}.
At last, the proof of Theorem \ref{thm:global-convergence} and Corollary \ref{coro:unique-convergence} is given in Section \ref{subsec:proof-global-convergence}.

\subsection{Global attractors of order-preserving semigroups on \texorpdfstring{$\PP_2(\R)$}{P2(R)}}
\label{subsec:global attractor}
The notion of order-preserving  has already appeared in the study of uniqueness of invariant measures, ergodicity, and synchronization phenomena for certain classes of Markov processes; see, for instance, \cite{Butkovsky2020,RobertRichard1996,RobertTweedie2000}. 
In this subsection, we instead utilize concepts from the theory of dissipative dynamical systems (see, e.g., Hale \cite{Ha88}) to study the asymptotic behaviour of a general semigroup on $\PP_2(\R)$ that preserves the stochastic order, allowing in particular for the presence of multiple invariant measures.
Based on the existence of a globally attracting order bounded set, we prove that a global attractor of an order-preserving semigroup must be encompassed by two order related invariant measures, and these two invariant measures attract all solutions initiated from the regions above and below them with respect to the stochastic order, respectively (see Proposition \ref{T:attractor}). 

Firstly, we give the definition of a semiflow on $\PP_2(\R)$.
\begin{definition}\label{D:semiflow}
A map $P^*: 
\mathbb{R_+}\times\PP_2(\R)\rightarrow\PP_2(\R)$ is said to be a jointly continuous semigroup, or a semiflow on $\PP_2(\R)$, if 
\begin{enumerate}[label=(\roman*)]
    \item  $P^*: \mathbb{R_+}\times\PP_2(\R)\rightarrow\PP_2(\R)$ is continuous;
    \item $P_0^*\mu=\mu$  for all $\mu\in\PP_2(\R)$;
    \item $P_t^*\circ P_s^*=P_{t+s}^*$ for all $t,s\in \mathbb{R_+}$.
\end{enumerate}
\end{definition}
Next, we introduce some necessary notations and definitions originated from the standard theory of dissipative dynamical systems.
For any $\mu\in\PP_2(\R)$, the \emph{positive orbit} $O_+(\mu)$ of $\mu$ is $\{P^*_t\mu: t\geq0\}$. 
A measure $\mu\in\PP_2(\R)$ is an \emph{invariant measure} of $P^*$, if $P^*_t\mu=\mu$ for any $t\geq0$. 
A subset $A\subset\PP_2(\R)$ is said to be \emph{positively invariant}, if $P^*_tA\subset A$, for any $t\geq 0$, and  \emph{invariant}, if $P^*_tA=A$, for any $t\geq 0$. 
$A\subset\PP_2(\R)$ is said to be \emph{eventually positively invariant}, if there exists $t_0\geq0$ such that, $P^*_tA\subset A$, for any $t\geq t_0$.
A compact invariant set $A\subset\PP_2(\R)$ is said to be the \emph{maximal} compact invariant set if every compact invariant set is contained in $A$. 
\begin{definition}\label{def:attract-set}
A set $K\subset\PP_2(\R)$ is said to attract a set $B\subset\PP_2(\R)$ if $\WW_2(P^*_tB,K)\to0$, as $t\to\infty$.
\end{definition}
Obviously, if a measure $\mu\in\PP_2(\R)$ is attracted by a compact set $K\subset\PP_2(\R)$, then the positive orbit $O_+(\mu)$ of $\mu$ is relatively compact in $\PP_2(\R)$.
\begin{definition}\label{D:global attractor}
A compact invariant set $A\subset\PP_2(\R)$ is said to be the global attractor, if $A$ attracts each bounded subset of $\PP_2(\R)$. 
\end{definition}
Clearly, if the global attractor exists, it must be unique, maximal and all the invariant measures are contained in it.
\begin{definition}\label{D:omega limit set}
For any set $B\subset\PP_2(\R)$, the $\omega$-limit set of $B$ is defined by
$$\omega(B):=\mathop{\bigcap}\limits_{s\geq0}\overline{\bigcup\limits_{t\geq s}P^*_tB},$$
where $\overline{D}$ means the closure of $D$ in $\PP_2(\R)$.
\end{definition}
Now, we give some basic properties of $\omega$-limit sets and dissipative dynamics.

\begin{lemma}\label{L:omega-limit-set}
Given a semiflow $P^*$ on $\PP_2(\R)$, then
\begin{enumerate}[label=\textnormal{(\roman*)}]
    \item a measure $\nu\in\omega(B)$ if and only if there exist two sequences $t_k\to\infty$, and $\nu_k\in B$ such that $\WW_2(P^*_{t_k}\nu_k,\nu)\to0$, as $k\to\infty$;
    \item if $B$ is connected in  $\PP_2(\R)$, then $\omega(B)$ is connected in $\PP_2(\R)$;
    \item if there exists a compact set $K\subset\PP_2(\R)$ attracting a bounded set $B\subset\PP_2(\R)$, then $\omega(B)$ is a nonempty, compact, invariant subset of $K$ and $\omega(B)$ attracts $B$;
    \item if there exists a nonempty compact set $K\subset\PP_2(\R)$ that attracts each compact subset of $\PP_2(\R)$, then $\mathop{\bigcap}\limits_{t\geq0}P^*_tK$ is the maximal compact invariant set.
\end{enumerate}
\end{lemma}

\begin{proof}
(i) is a direct consequence of the definition of $\xo$-limit sets. (ii)-(iv) are standard results for dissipative systems on  complete metric spaces, for their proof one can see Hale \cite[Lemma 3.1.1, Lemma 3.2.1 and Theorem 3.4.2 (ii)]{Ha88} respectively.
\end{proof}

Clearly, Lemma \ref{L:omega-limit-set} (iii) implies that, if the global attractor $A$ exists, then $\omega(B)\subset A$ for any bounded set $B\subset\PP_2(\R)$.
In the following, we give the definition of \emph{order-preserving} semiflows. 
\begin{definition}\label{D:order-preserving}
A semiflow $P^*$ on $\PP_2(\R)$ is called order-preserving (or, monotone), if $\mu\leq_{\st}\nu$ implies $P^*_t\mu\leq_{\st} P^*_t\nu$ for any $t>0$. 
\end{definition}

Now, we give our main theorem in this subsection concerning the structure of the global attractors of order-preserving semigroups on $\PP_2(\R)$.
\begin{proposition}\label{T:attractor}
{\rm (Global attractor).} Let $P^*$ be an order-preserving semiflow on $\PP_2(\R)$. Assume that there exists an eventually positively invariant, order bounded, closed set $K\subset\PP_2(\R)$ such that $K$ attracts each bounded subset of $\PP_2(\R)$. Then,  
\begin{enumerate}[label=\textnormal{(\roman*)}]
    \item $\omega(K)=\mathop{\bigcap}\limits_{t\geq0}P^*_tK$ is the global attractor. Furthermore, if $K$ is connected in $\PP_2(\R)$, $\omega(K)$ is connected in $\PP_2(\R)$;
\end{enumerate}
\begin{enumerate}[resume,label=\textnormal{(\roman*)}]
    \item $\omega(K)\subset[\underline{\nu},\overline{\nu}]$, where $\underline{\nu}, \overline{\nu}\in\omega(K)$ are invariant measures;
    \item $\underline{\nu}$ is globally attracting from below and $\overline{\nu}$ is globally attracting from above i.e., for any $\rho\leq_{\st}\underline{\nu}$, we have $\WW_2(P^*_t\rho,\underline{\nu})\to0$, as $t\to\infty$; for any $\rho\geq_{\st}\overline{\nu}$, we have $\WW_2(P^*_t\rho,\overline{\nu})\to0$, as $t\to\infty$;
    \item If $\underline{\nu}$ attracts a set $U\subset\PP_2(\R)$, then $\underline{\nu}$ attracts each measure in $\bigcup\limits_{\mu\in U}\{\rho\in\PP_2(\R): \rho\leq_{\st} \mu\}$. Correspondingly, if $\overline{\nu}$ attracts a set $V\subset\PP_2(\R)$, then $\overline{\nu}$ attracts each measure in $\bigcup\limits_{\mu\in V}\{\rho\in\PP_2(\R): \rho\geq_{\st}\mu\}$;
    \item If there exists $\rho_1<_{\st}\overline{\nu}$ (resp. $\rho_2>_{\st}\underline{\nu}$) such that $\WW_2(P^*_t\rho_1,\overline{\nu})\to0$ (resp. $\WW_2(P^*_t\rho_2,\underline{\nu})\to0$) as $t\to\infty$, then for any $\rho\geq_{\st} \rho_1$ (resp. $\rho\leq_{\st}\rho_2$), one has $\WW_2(P^*_t\rho,\overline{\nu})\to0$ (resp. $\WW_2(P^*_t\rho,\underline{\nu})\to0$), as $t\to\infty$. 
    \item If $P^*$ has a unique invariant measure $\nu$, then the singleton set $\{\nu\}$ is the global attractor. That is to say, $\nu$ attracts all the bounded subsets of $\PP_2(\R)$.
\end{enumerate}
\end{proposition}

\begin{proof}
(i). By Proposition \ref{prop:order-bounded-precompact}, the order bounded, closed set $K$ is compact in $\PP_2(\R)$. Since $K$ is eventually positively invariant, by definition of $\omega(K)$ we have $\omega(K)\subset K$. 
Since $K$ attracts $K$, Lemma \ref{L:omega-limit-set} (iii) implies that $\omega(K)$ is invariant. 
Then $\omega(K)=P^*_t\omega(K)\subset P^*_tK$ for any $t\geq0$. 
Hence $\omega(K)\subset\mathop{\bigcap}\limits_{t\geq0}P^*_tK.$ 
On the other hand, the definition of $\omega(K)$ implies that $\mathop{\bigcap}\limits_{t\geq0}P^*_tK\subset\omega(K).$ 
Hence, $\omega(K)=\mathop{\bigcap}\limits_{t\geq0}P^*_tK.$
By virtue of the fact that $K$ attracts every bounded subset of $\PP_2(\R)$, it follows from Lemma \ref{L:omega-limit-set} (iv) that $\omega(K)$ is the maximal compact invariant set. 
To prove $\omega(K)$ is the global attractor, we only need to show that $\omega(K)$ attracts every bounded subset $B$ of $\PP_2(\R)$. 
In fact, by Lemma \ref{L:omega-limit-set} (iii), $\omega(B)$ is an invariant subset of $K$ and $\omega(B)$ attracts $B$. 
Since $\omega(B)$ is an invariant subset of $K$, one has $\omega(B)\subset\omega(K)$.  
Therefore, $\omega(K)$ also attracts $B$. 
The connectedness of $\omega(K)$ follows directly from the connectedness of $K$ and Lemma \ref{L:omega-limit-set} (ii). 
This proves (i).

(ii). Since $K$ is order-bounded, assume that there exists $\mu_1$, $\mu_2\in\PP_2(\R)$ such that $\mu_1\leq_{\st} K$ and $\mu_2\geq_{\st} K$. We only prove that $\sup_{\leq_{\st}}\omega(K)$ exists, $\sup_{\leq_{\st}}\omega(K)\in\omega(K)$ and $\sup_{\leq_{\st}}\omega(K)$ is an invariant measure.
The proof for $\inf_{\leq_{\st}}\omega(K)$ is similar, we omit it here. 
Since $\mu_2\geq_{\st} K$ and $\omega(K)\subset K$, one has $\mu_2\geq_{\st}\omega(K)$. As is shown in (i), $\omega(K)$ is the global attractor. 
Thus, $\WW_2(P^*_t\mu_2,\omega(K))\to0$ as $t\to\infty$. Since $\omega(K)$ is compact, one can choose $t_k\to\infty$ and $\overline{\nu}\in\omega(K)$, such that $\WW_2(P^*_{t_k}\mu_2,\overline{\nu})\to0$ as $k\to\infty$. 
For any $k\geq1$, the order-preservation implies $P^*_{t_k}\mu_2\geq_{\st} P^*_{t_k}\omega(K)=\omega(K)$. Therefore, the closedness of the stochastic order entails that $\overline{\nu}\geq_{\st}\omega(K)$. 
Together with the fact $\overline{\nu}\in\omega(K)$, we have $\overline{\nu}=\sup_{\leq_{\st}}\omega(K)$. Fix any $t\geq0$. On the one hand, $P^*_t\overline{\nu}\geq_{\st} P^*_t\omega(K)=\omega(K)$ and $\overline{\nu}\in\omega(K)$ entail that $P^*_t\overline{\nu}\geq_{\st}\overline{\nu}$. 
On the other hand, $P^*_t\overline{\nu}\in P^*_t\omega(K)=\omega(K)$ and $\overline{\nu}=\sup_{\leq_{\st}}\omega(K)$ imply that $P^*_t\overline{\nu}\leq_{\st}\overline{\nu}$. 
Thus, by the antisymmetry of $\leq_{\st}$, $P^*_t\overline{\nu}=\overline{\nu}$, for any $t\geq 0$. That is to say, $\overline{\nu}$ is an invariant measure. Hence, we have proved (ii).

(iii). We only prove that $\overline{\nu}$ is globally attracting from above i.e., for any  $\rho\geq_{\st}\overline{\nu}$, we have $\WW_2(P^*_t\rho,\overline{\nu})\to0$, as $t\to\infty$. 
The proof for $\underline{\nu}$ is similar, we omit it here. Note that $\{P_t^*\rho\}_{t\geq0}$ is relatively compact in $\PP_2(\R)$.
For any $\pi\in\PP_2(\R)$ such that, $\WW_2(P^*_{t_k}\rho,\pi)\to0$ as $k\to\infty$ for some sequence $t_k\to\infty$, we only need to prove $\pi=\overline{\nu}$. 
In fact, since the global attractor $\omega(K)$ attracts $\rho$, one has $\pi\in\omega(K)$. 
For any $k\geq1$, we have $P^*_{t_k}\rho\geq_{\st} P^*_{t_k}\overline{\nu}=\overline{\nu}$. 
Therefore, the closedness of the stochastic order entails that $\pi\geq_{\st}\overline{\nu}$. 
Together with the facts $\pi\in\omega(K)$ and $\overline{\nu}=\sup_{\leq_{\st}}\omega(K)$, we have $\pi=\overline{\nu}$.
This proves (iii). 

(iv). We only give the proof of the case of $\overline{\nu}$ (the case of $\underline{\nu}$ is similar). In fact, if $\rho\geq_{\st}\mu$ for some $\mu\in V$, by Lemma \ref{L:omega-limit-set} (iii), it suffice to prove that $\omega(\rho)=\{\overline{\nu}\}$. 
Let $\pi\in\omega(\rho)$, Lemma \ref{L:omega-limit-set} (i) entails that there exists a sequence $t_k\to\infty$ such that $\WW_2(P^*_{t_k}\rho,\pi)\to0$, as $k\to\infty$. 
Since $P^*$ is order-preserving and $\rho\geq_{\st}\mu$, we have $P^*_{t_k}\rho\geq_{\st} P^*_{t_k}\mu$ for any $k\geq 1$. 
Together with the closedness of the stochastic order, $\WW_2(P^*_{t_k}\mu,\overline{\nu})\to0$ as $k\to\infty$, one has  $\pi\geq_{\st}\overline{\nu}$. 
Since the global attractor $\omega(K)\subset[\underline{\nu},\overline{\nu}]$, one has $\pi\in\omega(\rho)\subset\omega(K)\subset[\underline{\nu},\overline{\nu}]$.
Therefore, $\pi=\overline{\nu}$.
This proves (iv).

(v). We only give the proof of the case of $\overline{\nu}$ (the case of $\underline{\nu}$ is similar). 
In fact, for any $\rho\geq_{\st}\rho_1$, define $\pi:=\sup_{\leq_{\st}}\{\rho,\overline{\nu}\}\in\PP_2(\R)$ (the existence of $\pi$ follows from Remark \ref{rem:lattice}). 
(iii) entails that $\WW_2(P^*_t\pi,\overline{\nu})\to0$, as $t\to\infty$. 
On the other hand, we also have $\WW_2(P^*_t\rho_1,\overline{\nu})\to0$ as $t\to\infty$. 
Noticing that $\rho\in[\rho_1,\pi]$, the order-preserving property of $P^*$ entails that $P^*_{t}\rho_1\leq_{\st} P^*_{t}\rho\leq_{\st} P^*_{t}\pi$ for any $t\geq0$.
Thus, the relative compactness of $\{P_t^*\rho\}_{t\geq0}$, the closedness and antisymmetry of the stochastic order entail that $\WW_2(P^*_t\rho,\overline{\nu})\to0$, as $t\to\infty$. Now, we obtain (v). 

(vi). By virtue of (ii), one has $\nu\leq_{\st}\omega(K)\leq_{\st} \nu$. Thus, the antisymmetry of the stochastic order entails that $\omega(K)=\{\nu\}$. Hence, $\{\nu\}$ is the global attractor. This proves (vi).
\end{proof}

\begin{remark}\label{R:attractor-change space-history}
(i). If we consider a semiflow preserving the stochastic order acting only on a closed subspace $B$ of $\PP_2(\R)$ (rather than the entire space $\PP_2(\R)$), by repeating the same proof, Proposition  \ref{T:attractor} still holds, where the order interval within $B$ is the restriction of the order interval to $B$.
In fact, for semiflow preserving some closed partial order relation on a general complete metric space, by assuming the atttracting set $K$ is compact and order bounded, Proposition \ref{T:attractor} still holds. 

(ii). 
For semiflows preserving some closed partial order relation on a strongly ordered space, Hirsch \cite[Theorem 3.3]{Hirsch84} implies the unique equilibrium must be globally attracting (Proposition \ref{T:attractor} (vi)). However,  Proposition \ref{prop:no-strong-order} shows that the strong ordering relation generated by $\leq_{\st}$ is empty. So, \cite[Theorem 3.3]{Hirsch84} is not applicable here.
\end{remark}

\subsection{Coming down from infinity}
\label{subsec:coming-down-from-infinity}
In this subsection, we prove hyper-dissipative McKean-Vlasov SDEs have the coming down from infinity property.
It means, there exists a compact subset of $\PP_2(\R)$ attracting all bounded sets.
As a byproduct, we show $P_t^*:\PP_2(\R)\to\PP_2(\R)$ generated by the equation \eqref{eq:mvsystem} is a compact map.

Now let us go back to the McKean-Vlasov equation \eqref{eq:mvsystem}.
The driving noise $\{B_t\}_{t\in\R}$ is a two-sided one-dimensional standard Brownian motion defined on a probability space $(\xO,\FF,\bP)$, and let $\{\FF_t\}_{t\in\R}$ be the natural filtration generated by $B$, that is, $\FF_t=\xs\{B_u-B_v:u,v\leq t\}$.
Under Assumption \ref{asp:potential}, for any $s\in \R$ and $\xi\in L^2(\xO,\FF_s,\bP)$, it is well-known that \eqref{eq:mvsystem} has a unique solution $\{X_t^{s,\xi}\}_{t\geq s}$ starting at time $s$ with initial data $\xi$, i.e., $\{X_t^{s,\xi}\}_{t\geq s}$ solves the following equation (see e.g., Wang \cite{Wang18})
\begin{equation*}
\d X_t=-V'(X_t)\d t-(W'*\LL(X_t))(X_t)\d t+\xs(X_t)\d B_t, \ \ X_s=\xi, \ \ t\geq s.
\end{equation*}
When the starting time $s=0$, we denote the solution $\{X_t^{0,\xi}\}_{t\geq 0}$ by $\{X_t^{\xi}\}_{t\geq 0}$.

Suppose Assumption \ref{asp:potential} (ii)-(iv) hold. Then for any $x\in \R$, $\mu\in \PP_2(\R)$, we have
\begin{equation}\label{eq:0802-1}
    \begin{split}
        -xV'(x)-x(W'*\mu)(x)&=-xV'(x)-\theta x^2+\theta x \int_{\R}y\d \mu(y)\\
        &\leq -\alpha |x|^{2+\delta}-\theta|x|^2+\theta|x|\|\mu\|_2+L|x|+\beta.
    \end{split}
\end{equation}
The following lemma provides $p$-th moment estimation of the solution $\{X_t^{\xi}\}_{t\geq 0}$.

\begin{lemma}\label{lem:p-moment estimate}
    Suppose Assumption \ref{asp:potential} holds. For any $p\geq 2$ and $\xi\in L^p(\xO,\FF_0,\bP)$, there exists $C(p)=C(p,\alpha,\beta,\delta,\overline{\xs}, L)>0$ such that the unique solution $\{X_t^{\xi}\}_{t\geq 0}$ of \eqref{eq:mvsystem} satisfies
    \begin{equation}\label{eq:p-moment estimate}
        \mathbf{E}[|X_t^{\xi}|^p]\leq e^{-pt}\mathbf{E}[|\xi|^p]+C(p), \ \text{ for all } \ t\geq 0.
    \end{equation}
\end{lemma}
\begin{proof}
    Applying It{\^o}'s formula to $|X_t^{\xi}|^2$ together with \eqref{eq:0802-1}, we have
    \begin{equation}\label{eq:0802-2}
        \begin{split}
            \d |X_t^{\xi}|^2&=\big(-2X_t^{\xi}V'(X_t^{\xi})-2\theta|X_t^{\xi}|^2 +2\theta X_t^{\xi}\mathbf{E}[X_t^{\xi}]+\xs^2(X_t^{\xi})\big)\d t+2X_t^{\xi}\sigma(X_t^{\xi})\d B_t\\
            &\leq \big(-2\alpha |X_t^{\xi}|^{2+\delta}-2\theta|X_t^{\xi}|^2 +2\theta |X_t^{\xi}|(\mathbf{E}[|X_t^{\xi}|^2])^{\frac{1}{2}}+2L|X_t^{\xi}|+2\beta+\overline{\xs}^2\big)\d t\\
            &\ \ \ \  \qquad\qquad\qquad\qquad  \qquad\qquad\qquad\qquad  \qquad\qquad\qquad\qquad  +2X_t^{\xi}\sigma(X_t^{\xi})\d B_t.
        \end{split}
    \end{equation}
    Then Young's inequality gives
    \begin{equation*}
        \d |X_t^{\xi}|^2\leq \big(-(2+\theta)|X_t^{\xi}|^2+\theta \mathbf{E}[|X_t^{\xi}|^2]+3^{2\delta^{-1}+1}(2\alpha)^{-2\delta^{-1}}+L^2+2\beta+\overline{\xs}^2\big)\d t+2X_t^{\xi}\sigma(X_t^{\xi})\d B_t.
    \end{equation*}
    Thus, applying It{\^o}'s formula to $e^{2t}|X_t^{\xi}|^2$ and taking expectations yield
    \begin{equation*}
    e^{2t}\mathbf{E}[|X_t^{\xi}|^2]\leq \mathbf{E}[|\xi|^2]+2^{-1}\big(3^{2\delta^{-1}+1}(2\alpha)^{-2\delta^{-1}}+L^2+2\beta+\overline{\xs}^2\big)e^{2t},
    \end{equation*}
    which implies \eqref{eq:p-moment estimate} for $p=2$.

    For $p>2$, applying It{\^o}'s formula to $|X_t^{\xi}|^p$ together with \eqref{eq:0802-2} and Young's inequality gives
    \begin{equation*}
        \begin{split}
            \d |X_t^{\xi}|^p&=\frac{p}{2} |X_t^{\xi}|^{p-2}\d |X_t^{\xi}|^2+\frac{1}{2}\frac{p}{2}\big(\frac{p}{2}-1\big)|X_t^{\xi}|^{p-4}\d \langle |X^{\xi}|^{2}\rangle_t\\
            &=\big(-p|X_t^{\xi}|^{p-2}X_t^{\xi}V'(X_t^{\xi})-p\theta|X_t^{\xi}|^p +p\theta |X_t^{\xi}|^{p-2} X_t^{\xi}\mathbf{E}[X_t^{\xi}]+\frac{p(p-1)}{2}|X_t^{\xi}|^{p-2}\xs^2(X_t^{\xi})\big)\d t\\
            &\qquad\qquad\qquad\qquad\qquad\qquad\qquad\qquad\qquad\qquad\qquad\qquad\qquad +p|X_t^{\xi}|^{p-2}X_t^{\xi}\xs(X_t^{\xi})\d B_t\\
            &\leq \big(-p\alpha |X_t^{\xi}|^{p+\delta}-\theta |X_t^{\xi}|^{p}+\theta \mathbf{E}[|X_t^{\xi}|^p]+C(p)(|X_t^{\xi}|^{p-2}+|X_t^{\xi}|^{p-1})\big)\d t\\
            &\qquad\qquad\qquad\qquad\qquad\qquad\qquad\qquad\qquad\qquad +p|X_t^{\xi}|^{p-2}X_t^{\xi}\xs(X_t^{\xi})\d B_t\\
            &\leq \big(-(p+\theta)|X_t^{\xi}|^p+\theta \mathbf{E}[|X_t^{\xi}|^p]+C(p)\big)\d t+p|X_t^{\xi}|^{p-2}X_t^{\xi}\xs(X_t^{\xi})\d B_t,
        \end{split}
    \end{equation*}
    where $C(p)=C(p,\alpha,\beta,\delta,\overline{\xs}, L)$ is a constant, depending only on $(p,\alpha,\beta,\delta,\overline{\xs}, L)$, may be different line by line. Similarly,  applying It{\^o}'s formula to $e^{pt}|X_t^{\xi}|^p$ and taking expectations, we derive \eqref{eq:p-moment estimate}.
\end{proof}

Indeed, we want to show the solutions are $p$-integrable for all $p\geq 2$ if the initial conditions are just square integrable.
For this purpose, we need the following extrinsic-law case.
Let $C([s,\8);\PP_2(\R))$ be the space of continuous functions from $[s,\8)$ to $\PP_2(\R)$.
For any $\mu\in C([s,\8);\PP_2(\R))$ and $\xi\in L^2(\xO,\FF_s,\bP)$, we denote $\{X_t^{s,\mu,\xi}\}_{t\geq s}$ by the solution of the following extrinsic-law-dependent SDE
\begin{equation}\label{eq:extrinsic-law-sde}
    \d X_t=-V'(X_t)\d t-(W'*\mu_t)(X_t)\d t+\xs(X_t)\d B_t, \ \ X_s=\xi, \ \ t\geq s.
\end{equation}
Similarly, we denote the solution $\{X_t^{0,\mu,\xi}\}_{t\geq 0}$ by $\{X_t^{\mu,\xi}\}_{t\geq 0}$.

\begin{lemma}\label{lem:ldsde-coming-down}
Suppose Assumption \ref{asp:potential} holds.
If $\mu\in C([0,\8);\PP_2(\R))$, and if $\xi\in L^2(\xO,\FF_0,\bP)$, then for any $p\geq2$, there exists a constant $C(p)=C(p,\alpha,\beta,\delta,\overline{\xs}, \xk, \theta,L)>0$ such that the unique solution $\{X_t^{\mu,\xi}\}_{t\geq0}$ of \eqref{eq:extrinsic-law-sde} satisfies
\begin{equation*}
\big(\bE\big[|X_t^{\mu,\xi}|^p\big]\big)^{1/p}\leq
C(p)\Big(t^{-\frac{1}{\delta}}+t^{\frac{\xk}{2}}+\sup_{s\in[0,t]}\|\mu_s\|_2^{\frac{2}{2+\delta}}\Big)\ \text{ for all $t>0$}.
\end{equation*}
\end{lemma}

\begin{proof}
For $t\geq0$, set
\begin{equation*}
Z_t:=\int_0^t\xs(X_s^{\mu,\xi})\d B_s,\ \ \ Y_t:=X_t^{\mu,\xi}-Z_t.
\end{equation*}
Then \eqref{eq:extrinsic-law-sde} implies that $Y_t$ solves the following random differential equation,
\begin{equation*}
Y_t=\xi+\int_0^t\Big(-V'(Y_s+Z_s)-\theta(Y_s+Z_s)+\theta\int_{\R}x\d \mu_s(x)\Big)\d s.
\end{equation*}
Hence, $|Y_t|^2$ is differentiable in $t$.
Then by Assumption \ref{asp:potential} (ii)-(iv), we have
\begin{equation*}
\begin{split}
\frac{\d}{\d t}(|Y_t|^2)
&=2Y_t\Big(-V'(Y_t+Z_t)-\theta(Y_t+Z_t)+\theta\int_{\R}x\d \mu_t(x)\Big)\\
&\leq -2\alpha |Y_t|^{2+\delta}+2\beta-2\theta |Y_t|^{2}+2\theta|Y_t|\big(|Z_t|+\theta^{-1}|V'(Z_t)|\big)+2\theta |Y_t|\|\mu_t\|_2\\
&\leq -2\alpha |Y_t|^{2+\delta}+\theta\|\mu_t\|_2^2+2\beta+\xg+\gamma|Z_t|^{2\kappa},
\end{split}
\end{equation*}
where $\xg=3\theta(1+\theta^{-2}L^2)$.
For $t\geq0$, let
\[
y_t:=|Y_t|^{-\delta},\ \ \ \mu_t^*:=\sup_{s\in[0,t]}\|\mu_s\|_2,\ \ \ Z_t^*:=\sup_{s\in[0,t]}|Z_s|.
\]
Then for any $t\geq 0$ with $y_t<\8$, i.e., $Y_t\neq 0$, we have
\begin{equation*}
\begin{split}
\dif{y_t}{t}&=-\frac{\delta}{2}|Y_t|^{-(2+\delta)}\frac{\d}{\d t}(|Y_t|^2)
\geq \alpha\delta-\frac{\alpha\delta}{2}\Big(\frac{\theta}{\xa}(\mu_t^*)^2+\frac{2\xb+\xg}{\xa}+\frac{\xg}{\xa}(Z_t^*)^{2\kappa}\Big)y_t^{\frac{2+\delta}{\delta}}.
\end{split}
\end{equation*}
Hence, for any $s\in [0,t]$,
\begin{equation}\label{eq:0825-1}
\dif{y_s}{s}\geq \frac{\alpha\delta}{2}>0, \ \text{ if } \ y_s\leq\Big(\frac{\theta}{\xa}(\mu_t^*)^2+\frac{2\xb+\xg}{\xa}+\frac{\xg}{\xa}(Z_t^*)^{2\kappa}\Big)^{-\frac{\delta}{2+\delta}}.
\end{equation}
Then we have the following claim.

\noindent\textbf{Claim:} We conclude that
\begin{equation}\label{eq:lower-bd-y_t}
y_t\geq \min\bigg\{\frac{\alpha\delta}{2}t,\ \Big(\frac{\theta}{\xa}(\mu_t^*)^2+\frac{2\xb+\xg}{\xa}+\frac{\xg}{\xa}(Z_t^*)^{2\kappa}\Big)^{-\frac{\delta}{2+\delta}}\bigg\}, \ \text{ for any } \ t\geq 0.
\end{equation}

\noindent\textbf{Proof of Claim:} For any fixed $t\geq 0$, set
\[
A_t:=\Big(\frac{\theta}{\xa}(\mu_t^*)^2+\frac{2\xb+\xg}{\xa}+\frac{\xg}{\xa}(Z_t^*)^{2\kappa}\Big)^{-\frac{\delta}{2+\delta}}.
\]
The proof of the claim is divided into the following two cases.
\begin{equation*}
\begin{split}
    &\textbf{Case 1:} \ y_t\geq \frac{\alpha\delta}{2}t, \ \text{ if } \ y_s\leq A_t, \ \text{ for all } \ s\in [0,t];\\
    &\textbf{Case 2:} \ y_t\geq A_t, \ \text{ if there exists $s\in [0,t]$ such that } \ y_s\geq A_t.
\end{split}
\end{equation*}
Case 1 is a direct consequence of \eqref{eq:0825-1}. To prove Case 2, we assume otherwise that $y_t<A_t$. Note that $y_{\cdot}$ is continuous in $[s,t]$ and $y_s\geq A_t$, let
\[
r:=\inf\{v\in [s,t]: y_u\leq A_t, \text{ for all } u\in [v,t]\},
\]
we conclude that $s\leq r<t$ and
\begin{equation*}
    y_u\leq A_t, \text{ for all } u\in [r,t], \ \text{ and } \ A_t=y_r>y_t.
\end{equation*}
Then the mean value theorem indicates 
\[
\frac{\d y_u}{\d u}\Big|_{u=u_0}=\frac{y_t-y_r}{t-r}<0, \ \text{ for some } \ u_0\in [r,t],
\]
which gives a contradiction to \eqref{eq:0825-1} and the fact that $y_{u_0}\leq A_t$. We complete the proof of the claim.

By \eqref{eq:lower-bd-y_t}, we have, for any $p\geq 1$ and $t>0$,
\begin{equation*}
\begin{split}
|X_t^{\mu,\xi}|^p&\leq 2^{p-1}(|Y_t|^p+|Z_t|^p)\\
&\leq 2^{p-1}(y_t^{-\frac{p}{\delta}}+(Z_t^*)^p)\\
&\leq 2^{p-1}\max\bigg\{\Big(\frac{\alpha\delta}{2}t\Big)^{-\frac{p}{\delta}}, \Big(\frac{\theta}{\xa}(\mu_t^*)^2+\frac{2\xb+\xg}{\xa}+\frac{\xg}{\xa}(Z_t^*)^{2\kappa}\Big)^{\frac{p}{2+\delta}}\bigg\}+2^{p-1}(Z_t^*)^p\\
&\leq C(p)\Big(t^{-\frac{p}{\delta}}+(\mu_t^*)^{\frac{2p}{2+\delta}}+1+(Z_t^*)^{\xk p}\Big).
\end{split}
\end{equation*}
Assumption \ref{asp:potential} (v) and the Burkholder-Davis-Gundy (BDG) inequality imply that there exists $C(\xk,p)>0$ such that
\begin{equation*}
\mathbf{E}\big[(Z_t^*)^{\xk p}\big]\leq C(\xk,p)\mathbf{E}\bigg[\Big(\int_0^t|\xs(X_s^{\mu,\xi},\mu_s)|^2\d s\Big)^{\frac{\xk p}{2}}\bigg]\leq C(\xk,p)\cdot(\overline{\xs}t)^{\frac{\xk p}{2}}.
\end{equation*}
Thus, for any $p\geq 2$, there is some constant $C(p)>0$ satisfying, for all $t>0$,
\begin{equation*}
\begin{split}
\big(\bE\big[|X_t^{\mu,\xi}|^p\big]\big)^{1/p}\leq
C(p)\Big(t^{-\frac{1}{\delta}}+t^{\frac{\xk}{2}}+\sup_{s\in[0,t]}\|\mu_s\|_2^{\frac{2}{2+\delta}}\Big).
\end{split}
\end{equation*}
\end{proof}

Next, we give the $p$-th moment estimates for the solutions of \eqref{eq:mvsystem} when the initial conditions are square integrable.

\begin{lemma}\label{lem:mvsde-coming-down}
Suppose Assumption \ref{asp:potential} holds.
If $\xi\in L^2(\xO,\FF_0,\bP)$, then $X_t^{\xi}\in L^p(\xO,\FF_t,\bP)$ for all $p\geq2$, $t>0$.
Moreover, for any $p\geq2$, there exists a constant $C(p)\geq1$ such that, for all $t\geq1$,
\begin{equation*}
\big(\bE\big[|X_t^{\xi}|^p\big]\big)^{1/p}\leq\e^{-(t-1)}\big(\bE[|\xi|^2]\big)^{1/2}+C(p).
\end{equation*}
\end{lemma}

\begin{proof}
Denote $\mu_t:=\LL(X_t^{\xi})$ for $t\geq0$, and by Lemma \ref{lem:p-moment estimate} and Lemma \ref{lem:ldsde-coming-down}, we have for any $t>0$,
\begin{equation}\label{pf:mvsde-coming-down-1}
\big(\bE\big[|X_t^{\xi}|^p\big]\big)^{1/p}\leq C(p)\Big(t^{-\frac{1}{\delta}}+t^{\frac{\kappa}{2}}+\big(\bE[|\xi|^2]\big)^{\frac{1}{2+\xd}}\Big).
\end{equation}
Then for any $t\geq 1$, Lemma \ref{lem:p-moment estimate} and \eqref{pf:mvsde-coming-down-1} yield
\[
\big(\bE\big[|X_t^{\xi}|^p\big]\big)^{1/p}\leq\e^{-(t-1)}\big(\bE\big[|X_1^{\xi}|^p\big]\big)^{1/p}+C(p)\leq e^{-(t-1)}C(p)\big(\bE[|\xi|^2]\big)^{\frac{1}{2+\xd}}+C(p).
\]
Apply again Young's inequality by noticing $\frac{1}{(2+\xd)/2}+\frac{1}{(2+\xd)/\xd}=1$, we complete the proof.
\end{proof}

Now we are ready to prove the coming down from infinity under the hyper-dissipative condition.

\begin{proposition}\label{thm:mvsde-attractor}
Suppose Assumption \ref{asp:potential} holds.
Let $P_t^*:\PP_2(\R)\to\PP_2(\R)$ be the semigroup of the McKean-Vlasov SDE \eqref{eq:mvsystem}. Then for any $t>0$, $P_t^*$ is a compact map.
Moreover, there exists a nonempty set $K\subset \PP_{\8}(\R)$ such that
\begin{enumerate}[label=\textnormal{(\roman*)}]
\item $K$ is a connected compact set in $\PP_2(\R)$;
\item $P_t^*K\subset K$ for all $t\geq 2$;
\item $K$ is order bounded in $\PP_2(\R)$, i.e., there exist $\underline{\mu},\overline{\mu}\in\PP_2(\R)$ such that $\underline{\mu}\leq_{\st}K\leq_{\st}\overline{\mu}$;
\item $K$ absorbs all bounded sets in $\PP_2(\R)$, i.e., for any bounded set $B\subset \PP_2(\R)$, there exists $T>0$ such that $P_t^*B\subset K$ for all $t\geq T$.
\end{enumerate}
\end{proposition}

\begin{proof}
We first show the compactness of $P_t^*$ for any $t>0$.
For any bounded set $B$ in $\PP_2(\R)$, it follows from \eqref{pf:mvsde-coming-down-1} in Lemma \ref{lem:mvsde-coming-down} that $P_t^*B$ is a bounded subset of $\PP_p(\R)$ for all $p>2$.
By Corollary \ref{coro:bounded-order-bounded}, $P_t^*B$ is order bounded in $\PP_2(\R)$ with respect to the stochastic order.
Then Proposition \ref{prop:order-bounded-precompact} implies $P_t^*B$ is relatively compact in $\PP_2(\R)$, so $P_t^*$ is compact.

Next, we prove (i)-(iv). Set
\begin{equation}\label{pf:mvsde-attractor}
K_p:=\Big\{\mu\in\PP_p(\R): \|\mu\|_p\leq \frac{e}{e-1}C(p)\Big\}, \ \text{ and } \ K:=\bigcap_{p\geq2}K_p,
\end{equation}
where the constant $C(p)$ comes from Lemma \ref{lem:mvsde-coming-down}.
Clearly, $\xd_0\in K$, so $K$ is nonempty.
We now show (i)-connectedness.
Because $K$ is convex and the operation of convex combinations is continuous in $\PP_2(\R)$, we obtain $K$ is connected.
Then we want to show (iii).
Since $K$ is bounded in $\PP_p(\R)$ for all $p>2$, it follows from Corollary \ref{coro:bounded-order-bounded} that $K$ is order bounded.
By Proposition \ref{prop:order-bounded-precompact}, to prove (i)-compactness, we only need to show $K$ is closed. Suppose a $\PP_2$-Cauchy sequence $\{\mu_n\}_{n\geq 1}\subset K$ converges to a probability measure $\mu\in \PP_2(\R)$, and it remains to check $\mu\in K$, i.e., $\mu\in K_p$ for all $p\geq 2$. By \cite[Theorem 6.9]{Villani2009}, $\mu_n$ weakly converges to $\mu$. For any $p\geq 2$, since $\{\mu_n\}_{n\geq 1}\subset K_p$, then 
\[
\begin{split}
    \int_{\R}|x|^p\d \mu(x)&=\lim_{N\to \8}\int_{\R}(|x|^p\wedge N)\d \mu(x)\\
    &=\lim_{N\to \8}\lim_{n\to \8}\int_{\R}(|x|^p\wedge N)\d \mu_n(x)\\
    &\leq \sup_{n\geq 1}\int_{\R}|x|^p\d \mu_n(x)\\
    &\leq \frac{e}{e-1}C(p).
\end{split}
\]
This gives $\mu\in K_p$. Then $\mu\in K$ as $p\geq 2$ is arbitrary.

The statements (ii)(iv) can be easily checked by Lemma \ref{lem:mvsde-coming-down} and the construction of $K$ in \eqref{pf:mvsde-attractor}. The proof is completed.
\end{proof}

\subsection{Proof of Theorem \ref{thm:global-convergence} and Corollary \ref{coro:unique-convergence}}
\label{subsec:proof-global-convergence}

We first give the following lemma, which is a direct consequence of Feng-Qu-Zhao \cite[Theorem 4.8]{Feng-Qu-Zhao2023}. Let $C_b(\R;\PP_2(\R))$ be the collection of all bounded continuous paths $\{\nu_t\}_{t\in \R}$ in $\PP_2(\R)$.

\begin{lemma}\label{lem:time-inhomogeneous}
    Suppose Assumption \ref{asp:potential} holds, and let $\{\nu_t\}_{t\in \R}\in C_b(\R;\PP_2(\R))$. Consider the following time-inhomogeneous SDE
    \begin{equation}\label{eq:sde-1}
        \d X_t=-V'(X_t)\d t-(W'*\nu_t)(X_t)\d t+\xs(X_t)\d B_t.
    \end{equation}
     Let $X_t^{s,\mu}$ be the unique solution of equation \eqref{eq:sde-1} at time $t$ with starting time $s$ and initial distribution $\mu$ and $P^*(t,s)$ be the dual semigroup of equation \eqref{eq:sde-1}, i.e., $P^*(t,s)\mu:=\LL(X_t^{s,\mu})$. Then the equation \eqref{eq:sde-1} has a unique entrance measure $\{\rho^{\nu}_t\}_{t\in \R}$ in $C_b(\R;\PP_2(\R))$, i.e., $P^*(t,s)\rho^{\nu}_s=\rho^{\nu}_t$ for all $t\geq s$ and there exist $C,\lambda>0$ such that for all $\mu\in \PP_2(\R)$, $t\geq s$, we have
     \begin{equation}
         d_2\big(\LL(X_t^{s,\mu}), \rho^{\nu}_t\big)\leq C(1+\|\mu\|_2^2)e^{-\lambda (t-s)},
     \end{equation}
    where $d_2(\cdot,\cdot)$ is defined by
      \[
      d_2(\mu_1,\mu_2):=\int_{\R}(1+|x|^2)\d |\mu_1-\mu_2|(x), \ \text{ for all } \ \mu_1,\mu_2\in \PP_2(\R).
      \]
    In particular,
     \begin{enumerate}[label=\textnormal{(\roman*)}]
        \item if $\{\nu_t\}_{t\in \R}$ is an entrance measure of \eqref{eq:mvsystem}, i.e., $P^*_t\nu_s=\nu_{t+s}$ for all $t\geq 0, s\in \R$, then $\rho^{\nu}_t=\nu_t$ for all $t\in \R$;
        \item if $\{\nu_t\}_{t\in \R}$ is a singleton $\{\nu\}\subset\PP_2(\R)$, then  the unique entrance measure $\{\rho^{\nu}_t\}_{t\in \R}$ is also a singleton $\{\rho\}\subset \PP_2(\R)$ such that $\rho$ is the unique invariant measure of \eqref{eq:sde-1} in $\PP_2(\R)$.
     \end{enumerate}
\end{lemma}

The next lemma unveils the totally ordered structure of the set of invariant measures.

\begin{lemma}\label{lem:comparison-invariant-measure}
Suppose Assumption \ref{asp:potential} holds, and let $\mu_1,\mu_2$ be two invariant measures of the equation \eqref{eq:mvsystem}.
If $\int_{\R}x\d\mu_1(x)\leq\int_{\R}x\d\mu_2(x)$, then $\mu_1\leq_{\st}\mu_2$.
\end{lemma}

\begin{proof}
Let $m_i=\int_{\R}x\d\mu_i(x), i=1,2$.
Now we consider the following SDEs
\begin{equation}\label{eq:SDE-i}
    \d X_t^{i}=\big(-V'(X_t^{i})-\theta X_t^{i}+\theta m_i\big)\d t+\xs(X_t^{i})\d B_t,\quad i=1,2.
\end{equation}
Obviously, $\mu_i$ is an invariant measures of \eqref{eq:SDE-i}. Denote by $X_t^{i,x}$ the solution of \eqref{eq:SDE-i} with initial value $x$ and starting time $0$, $i=1,2$. Since $m_1\leq m_2$, it follows from 
the comparison theorem for usual SDEs \cite[Theorem 1.2]{Geib1994} 
that
\begin{equation}
\label{pf:comparison-invariant-measure-3}
\LL(X_t^{1,0})\leq_{\st}\LL(X_t^{2,0})\ \text{ for all $t\geq0$}.
\end{equation}
By Lemma \ref{lem:time-inhomogeneous}, $\mu_i$ is the unique invariant measure of \eqref{eq:SDE-i} and
\begin{equation}
\label{pf:comparison-invariant-measure-4}
d_2(\LL(X_t^{i,0}),\mu_i)\to0\ \text{ as $t\to\8$, for $i=1,2$},
\end{equation}
 Note that $\WW_2(\mu,\nu)\leq \sqrt{2d_2(\mu,\nu)}$ (see e.g. \cite[Theorem 6.15]{Villani2009}). Then we conclude that
\begin{equation}\label{pf:comparison-invariant-measure-7}
\WW_2(\LL(X_t^{i,0}),\mu_i)\to0\ \text{ as $t\to\8$, for $i=1,2$}.
\end{equation}
Together with \eqref{pf:comparison-invariant-measure-3}\eqref{pf:comparison-invariant-measure-7}, the closedness of stochastic order in $\PP_2(\R)$ implies that $\mu_1\leq_{\st}\mu_2$.
\end{proof}

\begin{remark}\label{rem:comparison-invariant-measure}
In general, we can consider two SDEs as below,
\begin{align}
\d X_t^{1}=\big(-V'(X_t^{1})-\theta X_t^{1}+\theta m_1\big)\d t+\xs(X_t^{1})\d B_t,
\label{rem:comparison-invariant-measure-1}\\
\d X_t^{2}=\big(-V'(X_t^{2})-\theta X_t^{2}+\theta m_2\big)\d t+\xs(X_t^{2})\d B_t.
\label{rem:comparison-invariant-measure-2}
\end{align}
By the proof of Lemma \ref{lem:comparison-invariant-measure}, whenever $m_1\leq m_2$, the invariant measures $\mu_1$, $\mu_2$ of the equations \eqref{rem:comparison-invariant-measure-1}, \eqref{rem:comparison-invariant-measure-2} are order related, precisely $\mu_1\leq_{\st}\mu_2$. 
\end{remark}

Now we are in a position to prove Theorem \ref{thm:global-convergence} and Corollary \ref{coro:unique-convergence}.

\begin{proof}[Proof of Theorem \ref{thm:global-convergence}]
(i). By Lemma \ref{lem:comparison-invariant-measure}, we see $\MM$ is totally ordered with respect to the stochastic order.
Take $K\subset\PP_{\8}(\R)$ as the set in \eqref{pf:mvsde-attractor}, and Proposition \ref{thm:mvsde-attractor} (iv) gives $\MM\subset K\subset\PP_{\8}(\R)$.
The finiteness of $\MM$ is postponed to the end of the proof.

(ii). 
By \cite[Theorem 3.5]{Liu-Qu-Yao-Zhi2024}, the semigroup $P_t^*:\PP_2(\R)\to\PP_2(\R)$ generated by the equation \eqref{eq:mvsystem} is an order-preserving semiflow.
Proposition \ref{thm:mvsde-attractor} gives a eventually positively invariant, order bounded, closed set $K\subset\PP_2(\R)$, and this set $K$ attracts all bounded subsets of $\PP_2(\R)$.
Now, apply Proposition \ref{T:attractor}, whose statements (i)(ii) say that, there are two invariant measures $\underline{\nu},\overline{\nu}\in\PP_2(\R)$ such that the global attractor $\xo(K)\subset[\underline{\nu},\overline{\nu}]$. By the definition of global attractor, one has $\MM\subset\xo(K)$. Obviously, $\underline{\nu}=\inf_{\leq_{\st}}\MM$ and $\overline{\nu}=\sup_{\leq_{\st}}\MM$.
Next, we need to show the global convergence to the order interval $[\underline{\nu},\overline{\nu}]$.
Since $\xo(K)$ is the global attractor, the inclusion $\xo(K)\subset[\underline{\nu},\overline{\nu}]$ tells us, for all bounded sets $B\subset\PP_2(\R)$, $\WW_2(P_t^*B,[\underline{\nu},\overline{\nu}])\to0$ as $t\to\8$.

(iii). This is a consequence of Proposition \ref{T:attractor} (iii).

(i)-finiteness.
Now we prove that $\MM$ is a finite set.
Let $\{\mu_m\}_{m\in \R}$ be probability measures given in \eqref{eq:density-function}, i.e., 
\begin{equation*}
\dif{\mu_m}{x}\propto\frac{1}{\xs^2(x)}\exp\bigg\{-2\int_0^x\frac{V'(y)+\theta y-\theta m}{\xs^2(y)}\d y\bigg\}.
\end{equation*}
By a simple calculation, $\mu_m$ solves the following Fokker-Planck equation 
\[
\mathcal{L}^*_m\mu_m=0,
\]
where $\mathcal{L}_m^*$ is the dual operator of the following $\mathcal{L}_m$:
\[
\mathcal{L}_m\varphi(x)=\frac{1}{2}\sigma^2(x)\partial^2_x \varphi(x)+\big(-V'(x)-\theta x+\theta m\big)\partial_x\varphi(x).
\]
Hence, $\mu_m$ is an invariant measure of the following SDE
\begin{equation}\label{eq:m-sde}
\d X_t=\big(-V'(X_t)-\theta X_t+\theta m\big)\d t+\xs(X_t)\d B_t,
\end{equation}
and it is unique (see Meyn-Tweedie \cite{Meyn-Tweedie1993-book,Meyn-Tweedie1993} or Lemma \ref{lem:time-inhomogeneous}).
Thus, a measure $\mu$ is an invariant measure of \eqref{eq:mvsystem} if and only if
\begin{equation}\label{pf:self-consistency}
\mu=\mu_m,\ \text{ and }\ \int_{\R} (x-m)\d \mu(x)=0\ \text{ for some $m\in \R$}.
\end{equation}
It is equivalent to
\[
\mu=\mu_m,\ \text{ and }\ \int_{\R}\frac{x-m}{\xs^2(x)}\exp\bigg\{-2\int_0^x\frac{V'(y)+\theta y-\theta m}{\xs^2(y)}\d y\bigg\}\d x=0\ \text{ for some $m\in \R$}.
\]
Consider the following self-consistency function,
\begin{equation}\label{eq:self-consistency-function}
F(m):=\int_{\R}\frac{x-m}{\xs^2(x)}\exp\bigg\{-2\int_0^x\frac{V'(y)+\theta y-\theta m}{\xs^2(y)}\d y\bigg\}\d x.
\end{equation}
Then the number of zeros of $F$ is exactly the number of invariant measures of the equation \eqref{eq:mvsystem}, and each zero of $F$ is the mean value of the corresponding invariant measure.
We assume there are infinitely many invariant measures, and argue by contradiction.
By the statement (ii), the set of invariant measures $\MM$ is contained in an order interval $[\underline{\nu},\overline{\nu}]$.
So the boundedness of $[\underline{\nu},\overline{\nu}]$ in $\PP_2(\R)$ (see Proposition \ref{prop:order-bounded-precompact}) entails boundedness of the zero set of $F$ in $\R$.
Since invariant measures are infinitely many, the zero set of $F$ has an accumulation point.
By Assumption \ref{asp:potential} (iii)(v), we have $F$ is an analytic function, and this implies $F$ is identically zero.
This is impossible, and we complete the proof.
\end{proof}

\begin{remark}\label{R:connected global attractor}
In Remark \ref{R:invariant measures cannot attract bdd sets}, we show that the set of invariant measures cannot attract all the bounded subsets of $\PP_2(\R)$. Here, by Proposition \ref{T:attractor} (i)(ii) and Proposition \ref{thm:mvsde-attractor} (i), there exists a connected compact invariant set (the global attractor) within $[\underline{\nu},\overline{\nu}]$ which attracts all the bounded subsets of $\PP_2(\R)$.
\end{remark}

\begin{proof}[Proof of Corollary \ref{coro:unique-convergence}]
The uniqueness of invariant measure enatils that $\underline{\nu}=\nu=\overline{\nu}$.
Then the antisymmetry of stochastic order implies that $[\underline{\nu},\overline{\nu}]=\{\nu\}$.
Hence, this corollary is a direct consequence of Theorem \ref{thm:global-convergence} (ii).
\end{proof}

\section{Basins of attraction}\label{sec:local-convergence}

This section attempts to prove Theorem \ref{T:general V}, Theorem \ref{thm:local-convergence}, Corollary \ref{coro:segment-convergence}, Theorem \ref{thm:double-well}, Theorem \ref{thm:multi-well} and Theorem \ref{thm:double-well-vanish}.
In Section \ref{subsec:number-invariant-measure}, we prove that the number of invariant measures of the equation \eqref{eq:mvsystem} is bounded by the zero-crossing number of the function $V'$.
Using the fact that solutions can be traced backward along connecting orbits, we establish  a local attractor theorem in Section \ref{subsec:local attractor} to show that the forward limit of a connecting orbit is locally attracting.
Then, we finish the proof of our main results, especially the verification of the local dissipative condition under the assumptions of Theorem \ref{T:general V}.

\subsection{Number of invariant measures}
\label{subsec:number-invariant-measure}

In this subsection, we provide the prior information on the number of invariant measures.
We show the number of invariant measures is at most the zero-crossing number of the derivative $V'$ of confining potential, regardless of the ranges of $\theta$ and $\xs$ (see Proposition \ref{prop:number-invariant-measure}).
This helps to make our main theorems quantitative (Theorem \ref{thm:double-well}-\ref{thm:double-well-vanish}), and overcomes difficulties caused by qualitative results and missing parameter in the previous literature on phase transitions (see e.g., \cite{Alecio2023,Tugaut2014}).
One the other hand, this proposition allows $V'$ vanishes on some interval (Theorem \ref{thm:double-well-vanish}), which is out of the settings of many existing works.
Also, it provides an insight that the essential factor in the maximal number of invariant measures is the zero-crossing number of $V'$, instead of the number of simple zeros of $V'$.

Let $f:\R\to\R$ be a function.
Recall that its zero-crossing number $Z(f)$ is the maximal integer $N$ such that, there exists an increasing sequence $x_1<x_2<\cdots<x_N<x_{N+1}$ satisfying $f(x_n)f(x_{n+1})<0$ for all $n=1,2,\dots,N$.
It is easy to check that, if $V'$ has $(2n-1)$ simple zeros, then $Z(V')=2n-1$.
Next, we will see the zero-crossing number of the function $V'$ is essential to determine the number of invariant measures.
We begin with a simple case where the zero-crossing number is one.

\begin{lemma}\label{lem:unique-zero-crossing}
Suppose Assumption \ref{asp:potential} holds.
If the function $f:\R\to\R$ is continuous with polynomial growth and has zero-crossing number $Z(f)=1$, then the function $F:\R\to\R$ defined by
\[
F(m):=\int_{\R}f(x)\exp\bigg\{-2\int_0^x\frac{V'(y)+\theta y-\theta m}{\xs^2(y)}\d y\bigg\}\d x
\]
has exactly one zero.
\end{lemma}

\begin{proof}
Since $f$ has polynomial growth, $F$ is well-defined.
Given $Z(f)=1$, either case of the following two is true.
\begin{itemize}
\item There is some $x_0\in\R$ such that $f(x)\leq0$ on $(-\8,x_0)$ and $f(x)\geq0$ on $(x_0,\8)$;
\item There is some $x_0\in\R$ such that $f(x)\geq0$ on $(-\8,x_0)$ and $f(x)\leq0$ on $(x_0,\8)$.
\end{itemize}
It suffices to consider the first case, and by definition of the zero-crossing number, there are $x_-<x_0<x_+$ such that
\[
f(x_-)<0,\ \ \ f(x_+)>0.
\]
Now define another function $\tilde F:\R\to\R$ by
\[
\tilde F(m):=\int_{\R}f(x)\exp\bigg\{-2\int_{x_0}^x\frac{V'(y)+\theta y-\theta m}{\xs^2(y)}\d y\bigg\}\d x,
\]
and we first show $\tilde F$ has only one zero.
Denote $\xS(x):=\int_{x_0}^x\frac{1}{\xs^2(y)}\d y$, and by Assumption \ref{asp:potential} (v), we observe $\xS(x)<0$ on $(-\8,x_0)$, $\xS(x)>0$ on $(x_0,\8)$.
A straightforward differentiation gives
\begin{align*}
\tilde F'(m)&=\int_{\R}2\theta\xS(x)f(x)\exp\bigg\{-2\int_{x_0}^x\frac{V'(y)+\theta y-\theta m}{\xs^2(y)}\d y\bigg\}\d x\\
&=\int_{\R}2\theta\xS(x)f(x)\exp\bigg\{-2\int_{x_0}^x\frac{V'(y)+\theta y}{\xs^2(y)}\d y\bigg\}\exp\big\{2\theta m\xS(x)\big\}\d x.
\end{align*}
For $m\geq0$, since $\xS(x)f(x)\geq0$ for all $x\in\R$, we have
\begin{equation}\label{pf:unique-zero-crossing-1}
\begin{aligned}
\tilde F'(m)&\geq\int_{x_0}^{\8}2\theta\xS(x)f(x)\exp\bigg\{-2\int_{x_0}^x\frac{V'(y)+\theta y}{\xs^2(y)}\d y\bigg\}\exp\big\{2\theta m\xS(x)\big\}\d x\\
&\geq\int_{x_0}^{\8}2\theta\xS(x)f(x)\exp\bigg\{-2\int_{x_0}^x\frac{V'(y)+\theta y}{\xs^2(y)}\d y\bigg\}\d x=:\xl_+,
\end{aligned}
\end{equation}
where $\xl_+$ is a constant independent of $m$.
Noticing $x_+>x_0$, $f(x_+)>0$, $\xS(x_+)>0$ and $\xS f$ is continuous, it follows $\xl_+>0$.
And in the same manner, for $m<0$, we see
\begin{equation}\label{pf:unique-zero-crossing-2}
\tilde F'(m)\geq\int_{-\8}^{x_0}2\theta\xS(x)f(x)\exp\bigg\{-2\int_{x_0}^x\frac{V'(y)+\theta y}{\xs^2(y)}\d y\bigg\}\d x=:\xl_->0.
\end{equation}
Combining \eqref{pf:unique-zero-crossing-1} and \eqref{pf:unique-zero-crossing-2} gives, for all $m\in\R$,
\[
\tilde F'(m)\geq\min\{\xl_+,\xl_-\}>0,
\]
so the function $\tilde F$ has only one zero.
To show $F$ has exactly one zero, we just simply note that $\tilde F(m)=F(m)G(m)$ for a positive function $G$ given by
\[
G(m)=\exp\bigg\{2\int_0^{x_0}\frac{V'(y)+\theta y-\theta m}{\xs^2(y)}\d y\bigg\}.
\]
Hence, the function $F$ has only one zero as well.
\end{proof}

The following lemma extends to the general case of finite zero-crossings.

\begin{lemma}\label{lem:multiple-zero-crossing}
Suppose Assumption \ref{asp:potential} holds.
If the function $f:\R\to\R$ is continuous with polynomial growth and has zero-crossing number $Z(f)=n$, then the function $F:\R\to\R$ defined by
\[
F(m):=\int_{\R}f(x)\exp\bigg\{-2\int_0^x\frac{V'(y)+\theta y-\theta m}{\xs^2(y)}\d y\bigg\}\d x
\]
has at most $n$ zeros.
\end{lemma}

\begin{proof}
By Lemma \ref{lem:unique-zero-crossing}, the conclusion is true for $n=1$, and we prove by induction.
Assume the conclusion is true for $n$.
If the function $f$ is continuous and $Z(f)=n+1$, then we may assume there is an increasing sequence $-\8=x_{-1}<x_0<\dots<x_n<x_{n+1}=\8$ such that, $f$ changes its sign on any two adjacent intervals $(x_{k-1},x_k)$, $k=0,1,\dots,n+1$.
Define $\tilde F:\R\to\R$ by
\[
\tilde F(m):=\int_{\R}f(x)\exp\bigg\{-2\int_{x_0}^x\frac{V'(y)+\theta y-\theta m}{\xs^2(y)}\d y\bigg\}\d x.
\]
Taking the derivative of $\tilde F$ yields
\begin{align*}
\tilde F'(m)&=\int_{\R}2\theta\xS(x)f(x)\exp\bigg\{-2\int_{x_0}^x\frac{V'(y)+\theta y-\theta m}{\xs^2(y)}\d y\bigg\}\d x\\
&=G(m)\cdot\int_{\R}2\theta\xS(x)f(x)\exp\bigg\{-2\int_0^x\frac{V'(y)+\theta y-\theta m}{\xs^2(y)}\d y\bigg\}\d x,
\end{align*}
where
\[
\xS(x)=\int_{x_0}^x\frac{1}{\xs^2(y)}\d y,\quad G(m)=\exp\bigg\{2\int_0^{x_0}\frac{V'(y)+\theta y-\theta m}{\xs^2(y)}\d y\bigg\}.
\]
Note the zero crossing number of $\xS f$ is $n$ and $G$ is positive, so $\tilde F'$ has at most $n$ zeros by the induction.
This implies $\tilde F$ has at most $(n+1)$ zeros, and since $F=G\tilde{F}$, the function $F$ has at most $(n+1)$ zeros.
We complete the proof.
\end{proof}

Now a bound of the number of invariant measures is given as below.

\begin{proposition}\label{prop:number-invariant-measure}
Suppose Assumption \ref{asp:potential} holds.
If the function $V':\R\to\R$ has zero-crossing number $Z(V')=2n-1$, then the equation \eqref{eq:mvsystem} has at most $(2n-1)$ invariant measures.
\end{proposition}

\begin{proof}
Consider the self-consistency function defined in \eqref{eq:self-consistency-function},
\[
F(m):=\int_{\R}\frac{x-m}{\xs^2(x)}\exp\bigg\{-2\int_0^x\frac{V'(y)+\theta y-\theta m}{\xs^2(y)}\d y\bigg\}\d x.
\]
Then the number of zeros of $F$ is exactly the number of invariant measures of the equation \eqref{eq:mvsystem}.
Use integration by parts, and we observe
\[
F(m)=-\int_{\R}\frac{V'(x)}{\theta\xs^2(x)}\exp\bigg\{-2\int_0^x\frac{V'(y)+\theta y-\theta m}{\xs^2(y)}\d y\bigg\}\d x.
\]
Assumption \ref{asp:potential} (iv)(v) tell us $\theta>0$ and $\xs^2(x)>0$ for all $x\in\R$, so $Z(V'/\theta\xs^2)=Z(V')=2n-1$.
Lemma \ref{lem:multiple-zero-crossing} implies that $F$ has at most $(2n-1)$ zeros, and thus, the equation \eqref{eq:mvsystem} has at most $(2n-1)$ invariant measures.
\end{proof}

\subsection{Local attractors of order-preserving semigroups on \texorpdfstring{$\PP_2(\R)$}{P2(R)}}
\label{subsec:local attractor}
To investigate the local dynamics near the invariant measures, we will present a theorem for order-preserving semigroups that ensures the local attraction of the forward limit  of a connecting orbit.
We first apply Proposition \ref{prop:number-invariant-measure} and the main result of \cite{Liu-Qu-Yao-Zhi2024} to obtain the existence of connecting orbits.

\begin{proof}[Proof of {\rm (i)-(iii)} in Theorem \ref{thm:local-convergence}]
By Proposition \ref{prop:number-invariant-measure}, the equation \eqref{eq:mvsystem} has at most $(2n-1)$ invariant measures.
Then, by applying \cite[Theorem 1.1]{Liu-Qu-Yao-Zhi2024}, it remains only to show the  pairwise disjointness of  $\{B(\xd_{a_k},r_k)\}_{k=1}^n$.
In fact, \eqref{eq:local-convergence-condition} implies
\begin{equation}\label{eq:new0805}
r_{i}+r_{j}\leq |a_i-a_j|, \ \text{ for all } \ 1\leq i\neq j\leq n.
\end{equation}
Suppose on the contrary that there exists $\nu\in B(\delta_{a_i},r_{i})\cap B(\delta_{a_{j}},r_{j})$. Then,
\begin{equation*}
    |a_i-a_j|=\WW_2(\delta_{a_i}, \delta_{a_j})\leq \WW_2(\delta_{a_i}, \nu)+\WW_2(\nu, \delta_{a_j})<r_i+r_j,
\end{equation*}
which gives a contradiction to \eqref{eq:new0805}.
\end{proof}

Based on the existence of connecting orbits, we give the following theorem ensuring the local attraction of the forward limit of a connecting orbit.

\begin{theorem}\label{T:locally attracting}
{\rm (Local attractor).} 
Let $P^*$ be an order-preserving semiflow on $\PP_2(\R)$.
Assume that there exist a positively invariant bounded open set $B\subset\PP_2(\R)$ and an order-bounded closed set $K\subset\PP_2(\R)$ such that $K$ attracts $B$. Assume further there is only one invariant measure $\mu_{0}$ in $\overline{B}$. 
Then the following hold.
\begin{enumerate}[label=\textnormal{(\roman*)}]
    \item If $\mu_0\geq_{\st}\omega(B)$ and there exists another invariant measure $\mu_{-1}<_{\st}\mu_0$ with an increasing connecting orbit $\{\nu^{-1,0}_t\}_{t\in\mathbb{R}}$ from $\mu_{-1}$ to $\mu_{0}$ such that $\nu^{-1,0}_{t_{-1}}\leq_{\st}\omega(B)$ for some $t_{-1}\in\mathbb{R}$, then $\mu_{0}$ attracts $B$.
    \item If $\mu_{0}\leq_{\st}\omega(B)$ and there exists another invariant measure $\mu_{1}>_{\st}\mu_{0}$ with a decreasing connecting orbit $\{\nu^{1,0}_t\}_{t\in\mathbb{R}}$ from $\mu_1$ to $\mu_{0}$ such that $\nu^{1,0}_{t_{1}}\geq_{\st}\omega(B)$ for some $t_{1}\in\mathbb{R}$, then $\mu_{0}$ attracts $B$.
    \item If there exist another two invariant measures $\mu_{-1}<_{\st}\mu_{0}$ and $\mu_{1}>_{\st}\mu_{0}$ with an increasing connecting orbit $\{\nu^{-1,0}_t\}_{t\in\mathbb{R}}$ from $\mu_{-1}$ to $\mu_{0}$ and a decreasing connecting orbit $\{\nu^{1,0}_t\}_{t\in\mathbb{R}}$ from $\mu_1$ to $\mu_{0}$ such that, $\nu^{-1,0}_{t_{-1}}\leq_{\st}\omega(B)$ for some $t_{-1}\in\mathbb{R}$ and $\nu^{1,0}_{t_{1}}\geq_{\st}\omega(B)$ for some $t_{1}\in\mathbb{R}$, then $\mu_{0}$ attracts $B$.
\end{enumerate}
\end{theorem}

\begin{proof}
(i). By Proposition \ref{prop:order-bounded-precompact}, $K$ is compact in $\PP_2(\R)$. Since the compact set $K$ attracts $B$, Lemma \ref{L:omega-limit-set} (iii) entails that $\omega(B)$ is nonempty, compact, invariant. 
Considering the closed set
$$S:=\{\nu^{-1,0}_t:\ t\geq t_{-1}\}\cup\{\mu_{0}\}\cup\omega(B)\subset\PP_2(\R),$$
it is obvious that $S$ is positively invariant.
Since $B$ is positively invariant, we have $\omega(B)\subset\overline{B}$.
Hence, $\mu_{0}$ is the unique invariant measure in $S$.
Observing that $S$ has a lower bound $\nu^{-1,0}_{t_{-1}}\in S$ and an upper bound $\mu_{0}\in S$, Proposition  \ref{T:attractor} (i) (vi) and Remark \ref{R:attractor-change space-history} (i) entail that $\omega(S)=\{\mu_{0}\}$ is the global attractor of $P^*|_{S}$, where $P^*|_{S}$ means the semiflow $P^*$ restricted to $S$.
The invariance of $\omega(B)$ and $\omega(B)\subset S$ entail that 
$$\omega(B)=\omega(\omega(B))\subset\omega(S).$$
Thus, the nonemptiness of $\omega(B)$ implies that $\omega(B)=\{\mu_{0}\}$.
Therefore, by Lemma \ref{L:omega-limit-set} (iii), $\mu_0$ attracts $B$.
This proves (i).

(ii). Considering the semiflow $P^*$ restricted to the positively invariant closed set 
$$S:=\{\nu^{1,0}_t:\ t\geq t_{1}\}\cup\{\mu_{0}\}\cup\omega(B)\subset\PP_2(\R)$$ 
with lower bound $\mu_0\in S$ and upper bound $\nu^{1,0}_{t_1}\in S$,
the proof of (ii) is analogous to (i).

(iii). Considering the semiflow $P^*$ restricted to the positively invairant closed set 
$$S:=\{\nu^{-1,0}_t:\ t\geq t_{-1}\}\cup\{\mu_{0}\}\cup\omega(B)\cup\{\nu^{1,0}_t:\ t\geq t_{1}\}\subset\PP_2(\R)$$ 
with lower bound $\nu^{-1,0}_{t_{-1}}\in S$ and upper bound $\nu^{1,0}_{t_1}\in S$, the proof of (iii) is analogous to (i).
\end{proof}

In the following, we will give key results to verify the core condition in Theorem \ref{T:locally attracting} that, $\xo(B(\xd_{a_k},r_k))$ has upper/lower bounds in connecting orbits (Proposition \ref{pro:bounds-of-limit-set} and Proposition \ref{prop:bounds of omega set}).
In Proposition \ref{pro:bounds-of-limit-set} and Proposition \ref{prop:bounds of omega set}, we will fix the connecting orbit as the laws of the solution flow of \eqref{eq:mvsystem}, which gives a time-inhomogeneous SDE.
Noting that the solution of this SDE can be traced backward along the connecting orbit, and by the comparison theorem, we can find a measure in the connecting orbit at a point of time sufficiently close to $-\8$, which is a lower bound or upper bound of $\xo(B(\xd_{a_k},r_k))$ with respect to stochastic order.

Firstly, we need to present a comparison principle for extrinsic-law-dependent SDEs \eqref{eq:extrinsic-law-sde}.

\begin{lemma}\label{thm:ldsde-comparison}
Suppose Assumption \ref{asp:potential} holds. Let $\mu,\nu\in C([s,\8);\PP_2(\R))$ and $T>0$.
Assume
\begin{equation}\label{ineq:mulenu}
    \int_{\R}x\d \mu_t(x)\leq \int_{\R}x\d \nu_t(x), \ \text{ for all } \ t\in [s,s+T].
\end{equation}
If $\xi,\eta\in L^2(\xO,\FF_s,\bP)$ and $\LL(\xi)\leq_{\st}\LL(\eta)$, then $\LL(X_t^{s,\mu,\xi})\leq_{\st}\LL(Y_t^{s,\nu,\eta})$ for all $t\in [s,s+T]$.
\end{lemma}

\begin{proof}
    Since $\LL(\xi)\leq_{\st}\LL(\eta)$, Strassen's theorem (see Lindvall \cite[Theorem IV.2.4]{Lindvall1992} or \cite[equation (3)]{Lindvall1999}) yields that there exist $\tilde{\xi}, \tilde{\eta}\in L^2(\xO,\FF_s,\bP)$ such that
    \begin{equation}
      \LL(\tilde{\xi})=\LL(\xi), \ \ \LL(\tilde{\eta})=\LL(\eta), \ \ \text{ and } \ \ \tilde{\xi}\leq \tilde{\eta}, \ \text{ $\mathbf{P}$-a.s.}
    \end{equation}
    It suffices to show that $\LL(X_t^{s,\mu,\tilde{\xi}})\leq_{\st}\LL(Y_t^{s,\nu,\tilde{\eta}})$ for all $t\in [s,s+T]$.

    Note that 
    \begin{equation}\label{eq:ldsde-comparison}
    \begin{aligned}
    \d X^{s,\mu,\tilde{\xi}}_t&=\Big[-V'(X^{s,\mu,\tilde{\xi}}_t)-\theta X^{s,\mu,\tilde{\xi}}_t+\theta \int_{\R}x\d \mu_t(x)\Big]\d t+\xs(X^{s,\mu,\tilde{\xi}}_t)\d B_t;\\
    \d Y^{s,\nu,\tilde{\eta}}_t&=\Big[-V'(Y^{s,\nu,\tilde{\eta}}_t)-\theta Y^{s,\nu,\tilde{\eta}}_t+\theta \int_{\R}x\d \nu_t(x)\Big]\d t+\xs(Y^{s,\nu,\tilde{\eta}}_t)\d B_t.
    \end{aligned}
    \end{equation}
    Since $\theta>0$, according to the comparison theorem for usual SDEs \cite[Theorem 1.2]{Geib1994}, we have
    \begin{equation*}
        X_t^{s,\mu,\tilde{\xi}}\leq Y^{s,\nu,\tilde{\eta}}_t \ \text{ for all } \ t\in [s,s+T], \text{ $\mathbf{P}$-a.s.} 
    \end{equation*}
    The desired result is proved.
\end{proof}

\begin{proposition}
\label{pro:bounds-of-limit-set}
    Suppose Assumption \ref{asp:potential} holds. Assume that a bounded subset $B$ in $\PP_2(\R)$ is positively invariant under $P^*$ and there exists a connecting orbit $\{\nutd_t\}_{t\in \R}$ from an invariant measure $\nu$ to another invariant measure $\pi$. If 
    \begin{equation}\label{ineq: 0502}
        \int_{\R}x\d \nu(x)<\inf_{\mu\in B}\int_{\R}x\d \mu(x), \ \ \Big(resp. \ \int_{\R}x\d \nu(x)>\sup_{\mu\in B}\int_{\R}x\d \mu(x)\Big),
    \end{equation}
    then there exists $t_0\in \R$ such that $\nutd_t\leq_{\st} \omega(B)$ $($resp. $\nutd_t\geq_{\st} \omega(B)${$)$} for all $t\leq t_0$.
\end{proposition}

\begin{proof}
    We first prove the ``$<$" case in \eqref{ineq: 0502}, and the case ``$>$" can be proved similarly. Now consider the following time-inhomogeneous SDE
    \begin{equation}\label{SDE-0502}
        \d X_t=-V'(X_t)\d t-(W'*\nutd_t)(X_t)\d t+\xs(X_t)\d B_t.
    \end{equation}
    Let $X_t^{s,\mu}$ be the unique solution of equation \eqref{SDE-0502} at time $t$ with starting time $s$ and initial distribution $\mu$ and $P^*(t,s)$ be the dual semigroup of equation \eqref{SDE-0502}, i.e., $P^*(t,s)\mu:=\LL(X_t^{s,\mu})$. Since $\{\nutd_t\}_{t\in \R}$ is a connecting orbit, then Lemma \ref{lem:time-inhomogeneous} (i) yields that it is an entrance measure of  equation \eqref{SDE-0502}, i.e., $P^*(t,s)\nutd_s=\nutd_t$ for all $t\geq s$.
    Note that $\sup_{t\in \R}\|\nutd_t\|_2<\infty$ (by definition, the connecting orbit is bounded in $\PP_2(\R)$), by Lemma \ref{lem:time-inhomogeneous}, we have there exist $C>0, \lambda>0$ such that for any $\mu\in \PP_2(\R)$ and $t\geq s$,
    \begin{equation}\label{eq:0806-1}
        d_2\big(\LL(X_t^{s,\mu}), \nutd_t\big)\leq C(1+\|\mu\|_2^2)e^{-\lambda (t-s)},
    \end{equation}
     where $d_2$ is given 
     by
        \begin{equation}\label{def:d2-distance}
        d_2(\nu_1,\nu_2)=\int_{\R}(1+|x|^2)\d |\nu_1-\nu_2|(x).
    \end{equation}
    According to \cite[Theorem 6.15]{Villani2009}, we conclude that
    \begin{equation*}
        \WW_2\big(\LL(X_t^{s,\mu}), \nutd_t\big)\leq \sqrt{2d_2\big(\LL(X_t^{s,\mu}), \nutd_t\big)}\leq \sqrt{2C(1+\|\mu\|_2^2)e^{-\lambda (t-s)}}, \ \text{ for all } \ \mu\in\PP_2(\R), \ t\geq s.
    \end{equation*}
    Hence, for all bounded sets $\{\mu_s\}_{s\leq t}$ in $\PP_2(\R)$, we have
     \begin{equation}\label{eq:0806-4}
         \lim_{s\to -\8} \WW_2\big(\LL(X_t^{s,\mu_s}), \nutd_t\big)=0, \ \text{ for any bounded set $\{\mu_s\}_{s\leq t}$ in $\PP_2(\R)$}.
     \end{equation}
    Note that $\{\nutd_t\}_{t\in \R}$ is a connecting orbit from $\nu$ to $\pi$ in $\PP_2(\R)$, then
    \[
    \limsup_{t\to -\infty}\Big|\int_{\R}x\d\nutd_t(x)-\int_{\R}x\d\nu(x)\Big|\leq \limsup_{t\to -\infty}\WW_1(\nutd_t,\nu)\leq \lim_{t\to -\infty}\WW_2(\nutd_t,\nu)=0.
    \]
    Hence, by \eqref{ineq: 0502}, there exists $t_0\in \R$ such that 
    \begin{equation}\label{eq:0803-1}
        \int_{\R}x\d \nutd_t(x)<\inf_{\mu\in B}\int_{\R}x\d \mu(x), \ \text{ for all } \ t\leq t_0.
    \end{equation}

    Next, we prove that $\nutd_t\leq_{\st} \omega(B)$ for all $t\leq t_0$.
    By Theorem \ref{thm:global-convergence} (ii), the order interval $[\underline{\nu},\overline{\nu}]$ attracts the bounded set $B$.
    Then Proposition \ref{prop:order-bounded-precompact} and Lemma \ref{L:omega-limit-set} (iii) tell us $\xo(B)$ is nonempty.
    For any $\rho\in \omega(B)$, it follows from Lemma \ref{L:omega-limit-set} (i) that there exist $\{(t_k, \rho_k)\}_{k\geq 1}\subset \R^+\times B$ such that $t_k\to \infty$ as $k\to \infty$ and 
    \begin{equation}\label{eq:1026-1}
       \lim_{k\to \infty}\WW_2(P^*_{t_k}\rho_k, \rho)=0.
    \end{equation}
    Fix $t\leq t_0$, and we consider the following two SDEs starting from $t-t_k$:
    \begin{equation*}
    \begin{aligned}
    \d X_s&=-V'(X_s)\d s-(W'*\nu_s)(X_s)\d s+\xs(X_s)\d B_s,& \LL(X_{t-t_k})&=\rho_k,\\
    \d Y_s&= -V'(Y_s)\d t-(W'*P^*_{s-t+t_k}\rho_k)(Y_s)\d s+\xs(Y_s)\d B_s,& \LL(Y_{t-t_k})&=\rho_k.
    \end{aligned}
    \end{equation*}
    Note that $\rho_k\in B$ and $B$ is positively invariant, then Assumption \ref{asp:potential} (iv), \eqref{eq:0803-1} and Lemma \ref{thm:ldsde-comparison} yield
    \begin{equation*}
        \LL(X_{t}^{t-t_k,\rho_k})\leq_{\st} \LL(Y_{t}^{t-t_k,\rho_k})=P^*_{t_k}\rho_k.
    \end{equation*}
    On the other hand, \eqref{eq:0806-4} and boundedness of $B$ imply $\WW_2(\LL(X_{t}^{t-t_k,\rho_k}),\nutd_t)\to0$ as $k\to\8$. Thus, \eqref{eq:1026-1} and the closedness of the stochastic order give $\nutd_t\leq_{\st} \rho$. Hence, $\nutd_t\leq_{\st} \omega(B)$ as $\rho\in \omega(B)$ is arbitrary.
\end{proof}

Now for any fixed $\mu\in \PP_2(\R)$, we consider the following SDE:
\begin{equation}\label{SDE 1}
\d X_t=-V'(X_t)\d t-(W'*\mu)(X_t)\d t+\xs(X_t)\d B_t.
\end{equation}
Under Assumption \ref{asp:potential}, we have SDE \eqref{SDE 1} has a unique invariant measure $\rho^{\mu}\in \PP_{2}(\R)$ (see e.g., Meyn-Tweedie \cite{Meyn-Tweedie1993-book,Meyn-Tweedie1993} or Lemma \ref{lem:time-inhomogeneous} (ii)). Define a measure-iterating map $\Psi: \PP_2(\R)\to \PP_2(\R)$ by
\begin{equation}\label{eq:psi}
    \Psi(\mu)=\rho^{\mu} \ \text{(the unique invariant measure of SDE \eqref{SDE 1})}.
\end{equation}
According to \cite[(67)]{Liu-Qu-Yao-Zhi2024}, we have the following property for $\Psi$.
\begin{lemma}\label{thm:Psi-invariant-neighbourhood}
    Suppose that Assumption \ref{asp:potential} holds. If the equation \eqref{eq:mvsystem} is locally dissipative at $a\in \R$ with configuration $(r_a,g_a)$, then we have $\Psi \big(\overline{B(\delta_a,r_a)}\big)\subset B(\delta_a,r_a)$.
\end{lemma}

\begin{remark}\label{rem:no-invariant-measure}
    Note that $\mu$ is an invariant measure of the equation \eqref{eq:mvsystem} if and only if $\Psi(\mu)=\mu$. Then Lemma \ref{thm:Psi-invariant-neighbourhood} implies that there is no invariant measure in $\partial \overline{B(\delta_a,r_a)}:=\overline{B(\delta_a,r_a)}\setminus B(\delta_a,r_a)$.
\end{remark}

\begin{proposition}\label{prop:bounds of omega set}
    Assume all assumptions in Theorem \ref{thm:local-convergence} hold. Let $\{\nu^{k,\uparrow}_t\}_{t\in \R}$ {$($}resp. $\{\nu^{k,\downarrow}_t\}_{t\in \R}${$)$} be the increasing {$($}resp. decreasing{$)$} connecting orbit from $\nu_{2k}$ to $\nu_{2k+1}$ {$($}resp. $\nu_{2k-1}${$)$} in Theorem \ref{thm:local-convergence} (iii). Then there exists $t_k\in \R$ such that, for all $t\leq t_k$,
    \begin{equation*}
        \nu^{k,\uparrow}_{t}\leq_{\st} \omega\big(B(\delta_{a_{k+1}},r_{k+1})\big), \  \text{ {$\big($}resp. $\nu^{k,\downarrow}_{t}\geq_{\st} \omega\big(B(\delta_{a_k},r_k)\big)${$\big)$}}. 
    \end{equation*}
\end{proposition}

\begin{proof}
    We only prove there is some $t_k\in\R$ such that $\nu^{k,\uparrow}_{t}\leq_{\st} \omega\big(B(\delta_{a_{k+1}},r_{k+1})\big)$ for all $t\leq t_k$, and the proof of the other case is similar. 
    Note that $B(\delta_{a_{k+1}},r_{k+1})$ is positively invariant under $P^*$, by virtue of Proposition \ref{pro:bounds-of-limit-set}, it suffices to show that
    \[
    \int_{\R}x\d \nu_{2k}(x)<\inf\Big\{\int_{\R}x\d \mu(x): \mu\in B(\delta_{a_{k+1}},r_{k+1})\Big\}.
    \]
    Let $\{\mu_m\}_{m\in \R}$ be probability measures given in \eqref{eq:density-function}. According to \eqref{pf:self-consistency}, a measure $\mu$ is an invariant measure of \eqref{eq:mvsystem} if and only if $\mu=\mu_{m}$ and $\int x\d \mu_m(x)=m$ for some $m\in \R$. By (i) of Theorem \ref{thm:local-convergence}, there exist $m_1<m_2<\cdots<m_{2n-1}$ such that 
    \begin{equation}\label{eq:invariant-nu-to-mu_m}
        \nu_i=\mu_{m_i}, \ \text{ and } \ \int_{\R}x\d\nu_i(x)=m_i, \ \text{ for all } \ i=1,2,\cdots,2n-1.
    \end{equation}
    We have the following claim.
    
    \noindent\textbf{Claim:} For all $k=1,2,\cdots,n-1$, $a_k+r_k<m_{2k}<a_{k+1}-r_{k+1}$.

    \noindent\textbf{Proof of Claim:} We only prove $m_{2k}<a_{k+1}-r_{k+1}$, and the other one can be obtained similarly. 

    We first show that $m_{2k}\notin [a_{k+1}-r_{k+1},a_{k+1}+r_{k+1}]$. Otherwise, since
    \begin{equation*}
        \bigg\{\int_{\R}x\d\mu(x): \mu\in \overline{B(\delta_{a_{k+1}},r_{k+1})}\bigg\}=[a_{k+1}-r_{k+1},a_{k+1}+r_{k+1}],
    \end{equation*}
    there exists $\mu\in \overline{B(\delta_{a_{k+1}},r_{k+1})}$ such that $\int x\d\mu(x)=m_{2k}$. By the definition of $\Psi$ (see \eqref{eq:psi}) and Lemma \ref{thm:Psi-invariant-neighbourhood}, it follows that $\Psi(\mu)=\mu_{m_{2k}}=\nu_{2k}\in \Psi(\overline{B(\delta_{a_{k+1}},r_{k+1})})\subset B(\delta_{a_{k+1}},r_{k+1})$. This gives a contradiction to \eqref{eq:nu notin ball}. Hence $m_{2k}\notin [a_{k+1}-r_{k+1},a_{k+1}+r_{k+1}]$.

    Note also, by Theorem \ref{thm:local-convergence} (i)(ii) that $\nu_{2k}<_{\st}\nu_{2k+1}$ and $\nu_{2k+1}\in B(\delta_{a_{k+1}},r_{k+1})$, we conclude that 
    \[
    m_{2k}=\int x\d \nu_{2k}(x)\leq\int x\d \nu_{2k+1}(x)\leq a_{k+1}+r_{k+1}.
    \]
    This together with $m_{2k}\notin [a_{k+1}-r_{k+1},a_{k+1}+r_{k+1}]$ yields $m_{2k}<a_{k+1}-r_{k+1}$. The Claim is proved.

    By Claim, it follows
    \[
    m_{2k}=\int_{\R}x\d \nu_{2k}(x)<\inf\Big\{\int_{\R}x\d \mu(x): \mu\in \overline{B(\delta_{a_{k+1}},r_{k+1})}\Big\}=a_{k+1}-r_{k+1}.
    \]
    This finishes the proof.
\end{proof}

\subsection{Proof of Main Theorems}

\label{subsec:proof-local-convergence}

We have obtained Theorem \ref{thm:global-convergence} and Corollary \ref{coro:unique-convergence} in Section \ref{subsec:proof-global-convergence}, and proved (i)–(iii) of Theorem \ref{thm:local-convergence} in Section \ref{subsec:local attractor}.
In this section, we complete the proof of the remaining main theorems and conclude with a conjecture on a saddle-point structure in the double-well granular media equation.

\begin{proof}[Proof of \textnormal{(iv)(v)} in Theorem \ref{thm:local-convergence}]
(iv). For $k=1$, according to the statement (ii), the open ball $B(\xd_{a_1},r_1)$ is positively invariant under $P^*$.
By Theorem \ref{thm:local-convergence} (i)(ii) and Remark \ref{rem:no-invariant-measure}, there is only one invariant measure $\nu_1$ in $\overline{B(\xd_{a_1},r_1)}$.
Theorem \ref{thm:global-convergence} (ii) and Theorem \ref{thm:local-convergence} (i) entail that $[\nu_1,\nu_{2n-1}]$ attracts $B(\xd_{a_1},r_1)$ and hence $\xo(B(\xd_{a_1},r_1))\subset[\nu_1,\nu_{2n-1}]$.
Then $\nu_1\leq_{\st}\xo(B(\xd_{a_1},r_1))$.
Let $\{\nu_t^{1,\downarrow}\}_{t\in\R}$ be a decreasing connecting orbit from $\nu_2$ to $\nu_1$ given in Theorem \ref{thm:local-convergence} (iii).
Proposition \ref{prop:bounds of omega set} shows that there exists $t_1\in \R$ such that $\nu_{t_1}^{1,\downarrow}\geq_{\st} \xo(B(\xd_{a_1},r_1))$.
Note that, by \cite[Theorem 3.5]{Liu-Qu-Yao-Zhi2024}, the semigroup $P_t^*:\PP_2(\R)\to\PP_2(\R)$ generated by the equation \eqref{eq:mvsystem} is an order-preserving semiflow.
Then Theorem \ref{T:locally attracting} (ii) implies that $\nu_1$ attracts $B(\xd_{a_1},r_1)$ under $P^*$, which means $\sup_{\mu\in B(\xd_{a_k},r_k)}\WW_2(P_t^*\mu,\nu_1)\to0$ as $t\to\8$.

For $k=2,3,\dots,n-1$, let $\{\nu_t^{k,\downarrow}\}_{t\in\R}$ be a decreasing connecting orbit from $\nu_{2k}$ to $\nu_{2k-1}$ and $\{\nu_t^{k,\uparrow}\}_{t\in\R}$ be an increasing connecting orbit from $\nu_{2k-2}$ to $\nu_{2k-1}$ given in Theorem \ref{thm:local-convergence} (iii).
It follows from the statement (ii) that $B(\xd_{a_k},r_k)$ is positively invariant under $P^*$.
By Theorem \ref{thm:local-convergence} (i)(ii) and Remark \ref{rem:no-invariant-measure}, there is only one invariant measure $\nu_{2k-1}$ in $\overline{B(\xd_{a_k},r_k)}$.
We also see the order interval $[\nu_1,\nu_{2n-1}]$ attracts $B(\xd_{a_k},r_k)$ by Theorem \ref{thm:global-convergence} (ii) and Theorem \ref{thm:local-convergence} (i).
By Proposition \ref{prop:bounds of omega set}, we can find some $t_k\in\R$ such that $\nu_{t_k}^{k,\uparrow}\leq_{\st}\xo(B(\xd_{a_k},r_k))\leq_{\st}\nu_{t_k}^{k,\downarrow}$.
Theorem \ref{T:locally attracting} (iii) suggests that $\sup_{\mu\in B(\xd_{a_k},r_k)}\WW_2(P_t^*\mu,\nu_{2k-1})\to0$ as $t\to\8$.

The last case is $k=n$.
We just note that $\nu_{2n-1}\geq_{\st}\xo(B(\xd_{a_n},r_n))$ by Theorem \ref{thm:global-convergence} (ii) and Theorem \ref{thm:local-convergence} (i).
Then Theorem \ref{T:locally attracting} (i) shows the result valid as in the proof of case $k=1$.

(v). Noticing the conditions in Proposition  \ref{T:attractor} are verified in the proof of Theorem \ref{thm:global-convergence}, we apply Proposition \ref{T:attractor} (iv) to finish the proof.
If $\mu\leq_{\st}\nu$ for some $\nu\in B(\xd_{a_1},r_1)$, then from Proposition  \ref{T:attractor} (iv), we have $\WW_2(P_t^*\mu,\nu_1)\to0$.
If $\mu\geq_{\st}\nu$ for some $\nu\in B(\xd_{a_n},r_n)$, the convergence result follows as well.

Now, it remains to show the basin of attraction of $\nu_1$ (resp. $\nu_{2n-1}$) contains $\cup_{a\leq0}B(\xd_{a_1+a},r_1)$ (resp. $\cup_{a\geq0}B(\xd_{a_n+a},r_n)$).
We only need to show for any $\mu\in\cup_{a\leq0}B(\xd_{a_1+a},r_1)$, there exists $\nu\in B(\xd_{a_1},r_1)$ such that $\mu\leq_{\st}\nu$.
And the other case of $\nu_{2n-1}$ can be proved similarly.
For $a\in\R$, let $\xt^a$ be the translation map by $a$ on $\R$, say $\xt^a(x):=x+a$.
Let $\xt_{\#}^a$ be the pushforward of $\xt^a$ acting on probability measures, that is, $\d(\xt_{\#}^a\nu)(x):=\d\nu(x-a)$.
Then it follows $\xt_{\#}^a\nu\leq_{\st}\nu$ for $a\leq0$, and $\xt_{\#}^a\nu\geq_{\st}\nu$ for $a\geq0$.
Note $\xt_{\#}^a(B(\xd_{a_1},r_1))=B(\xd_{a_1+a},r_1)$, and hence, for any $\mu\in B(\xd_{a_1+a},r_1)$ with $a\leq0$, there exists $\nu\in B(\xd_{a_1},r_1)$ such that $\mu=\xt_{\#}^a\nu\leq_{\st}\nu$.
\end{proof}

\begin{proof}[Proof of Corollary \ref{coro:segment-convergence}]
    (i). This follows from \eqref{eq:invariant-nu-to-mu_m}. 

    (ii). Note from \eqref{eq:m-sde} that for any $m\in \R$, $\mu_m$ is the unique invariant measure of the following SDE:
    \begin{equation}\label{SDE-0805}
        \d X^m_t=\big(-V'(X^m_t)-\theta X^m_t+\theta m\big)\d t+\xs(X^m_t)\d B_t.
    \end{equation}
    We prove (ii) in the following three cases.

    \noindent\textbf{Case $k=1$:} Let $\{\nu^{1,\downarrow}_t\}_{t\in \R}$ be the decreasing connecting orbit from $\nu_2=\mu_{m_2}$ to $\nu_1$ in Theorem \ref{thm:local-convergence} (iii). Then $\WW_2(\nu_t^{1,\downarrow},\mu_{m_2})\to0$ as $t\to-\8$.
    Fixing any $m\in (-\8, m_2)$, we conclude that there exists $t_0\in \R$ such that
    \begin{equation*}
        \alpha^{1,\downarrow}_t:=\int_{\R}x\d \nu^{1,\downarrow}_t(x)\geq m, \ \text{ for all } \ t\leq t_0.
    \end{equation*}
    Similar to \eqref{SDE-0502}, consider the following time-inhomogeneous SDE:
    \begin{equation}\label{SDE-0805-1}
        \d Y_t=\big(-V'(Y_t)-\theta Y_t+\theta\alpha^{1,\downarrow}_t\big)\d t+\xs(Y_t)\d B_t.
    \end{equation}
    Then it is easy to check that $\{\nu^{1,\downarrow}_t\}_{t\in \R}$ is an entrance measure of \eqref{SDE-0805-1}. Let $X_t^{m;s,0}$ and $Y^{s,0}_t$ be the unique solution at time $t$ start from $s\leq t$ with initial data 0 of equation \eqref{SDE-0805} and \eqref{SDE-0805-1} respectively. Then Lemma \ref{lem:time-inhomogeneous} entails that there exist $C>0,\lambda>0$ such that for all $t\geq s$,
    \[
    d_2\big(\LL(X_t^{m;s,0}), \mu_m\big)\leq Ce^{-\lambda (t-s)}, \ \text{ and } \ d_2\big(\LL(Y^{s,0}_t), \nu_t^{1,\downarrow}\big)\leq Ce^{-\lambda (t-s)}.
    \]
    Similar to the proof of \eqref{eq:0806-4}, we conclude that
    \begin{equation}
       \lim_{s\to-\8}\WW_2\big(\LL(X_{t_0}^{m;s,0}), \mu_m\big)=0, \ \text{ and } \ \lim_{s\to-\8}\WW_2\big(\LL(Y^{s,0}_{t_0}), \nu^{1,\downarrow}_{t_0}\big)=0.
    \end{equation}
    On the other hand, Lemma \ref{thm:ldsde-comparison} yields
    \[
    \LL(X_{t_0}^{m;s,0})\leq_{\st}\LL(Y^{s,0}_{t_0}), \ \text{ for all } \ s\leq t_0.
    \]
    Then the closedness of the stochastic order gives $\mu_m\leq_{\st} \nu^{1,\downarrow}_{t_0}$. Note that 
    \[
    \WW_2(P_t^*\nu^{1,\downarrow}_{t_0}, \nu_1)=\WW_2(\nu^{1,\downarrow}_{t_0+t}, \nu_1)\to 0, \ \text{ as } \ t\to\8,
    \]
    then Proposition \ref{T:attractor} (v) shows that $\WW_2(P^*_t\mu_m, \nu_1)\to 0$ as $t\to\8$.

    \noindent\textbf{Case $1<k<n$:} Let $\{\nu^{k,\downarrow}_t\}_{t\in \R}$ and $\{\nu^{k-1,\uparrow}_t\}_{t\in \R}$ be the corresponding decreasing and increasing connecting orbits from $\nu_{2k}=\mu_{m_{2k}}$ to $\nu_{2k-1}$, and $\nu_{2k-2}=\mu_{m_{2k-2}}$ to $\nu_{2k-1}$ respectively in Theorem \ref{thm:local-convergence} (iii). Fix any $m\in (m_{2k-2},m_{2k})$, the connecting orbits indicate that there exists $t_0\in \R$ such that
    \begin{equation*}
        \int_{\R}x\d \nu^{k,\downarrow}_t(x)\geq m, \text{ and } \ \int_{\R}x\d \nu^{k-1,\uparrow}_t(x)\leq m \ \text{ for all } \ t\leq t_0.
    \end{equation*}
    Similar to that in Case $k=1$, we have $\nu^{k-1,\uparrow}_{t_0}\leq_{\st} \mu_m\leq_{\st} \nu^{k,\downarrow}_{t_0}$. Note that
    \[
    \WW_2(P^*_t\nu^{k-1,\uparrow}_{t_0},\nu_{2k-1})\to0, \ \text{ and } \ \WW_2(P^*_t\nu^{k,\downarrow}_{t_0},\nu_{2k-1})\to0, \ \text{ as } \ t\to \8,
    \]
    and the order-preserving property of $P^*$ entails that 
    \[
    P^*_t\nu^{k-1,\uparrow}_{t_0}=\nu^{k-1,\uparrow}_{t_0+t}\leq_{\st} P^*_t\mu_m\leq_{\st} \nu^{k,\downarrow}_{t_0+t}=P^*_t\nu^{k,\downarrow}_{t_0}, \ \text{ for all } \ t\geq 0.
    \]
    Thus, the relative compactness of $\{P_t^*\mu_m\}_{t\geq0}$, the closedness and antisymmetry of the stochastic order entail that $\WW_2(P^*_t\mu_m,\nu_{2k-1})\to0$ as $t\to\infty$.

    \noindent\textbf{Case $k=n$:} Let $\{\nu^{n-1,\uparrow}_t\}_{t\in \R}$ be the increasing connecting orbit from $\nu_{2n-2}=\mu_{m_{2n-2}}$ to $\nu_{2n-1}$ and fix any $m\in (m_{2n-2},\8)$. Similar to Case $k=1$, there exists $t_0\in \R$ such that $\mu_m\geq_{\st} \nu^{n-1,\uparrow}_{t_0}$. Then applying Proposition \ref{T:attractor} (v) again, we conclude $\WW_2(P^*_t\mu_m,\nu_{2n-1})\to0$ as $t\to\8$.
\end{proof}

Before proving Theorem \ref{T:general V}, we first give the following key proposition to verify the local dissipativity condition.

\begin{proposition}\label{prop:local-dissipation-at-a}
    Assume that $W(x)=\frac{\theta}{2}x^2$ and the confining potential function $V\in C^2(\R;\R)$ satisfy the following conditions:
    \begin{enumerate}
        \item [(i)] There exists $a\in \mathbb{R}$ such that $V'(a)=0$ and $V''(a)>0$;
        \item [(ii)] $\liminf_{|x|\to\infty}\frac{V'(x)}{x|x|^{\delta}}>0$ for some $\delta>0$.
    \end{enumerate}
    Then, there exists $r_a > 0$ such that for any $0 < r \leq r_a$, there are $\theta_r > 0$ and $\sigma_r > 0$ such that for all $\theta \geq \theta_r$ and $\sigma^2(x) \leq \sigma_r$ for all $x \in \R$, the equation \eqref{eq:mvsystem} is locally dissipative at $a$ with radius $r$.
\end{proposition}
\begin{proof}
     Without loss of generality, we assume $0<\delta<1$. Then we first show that there exist $C_1,C_2, C_3>0$ such that
    \begin{equation}\label{eq:lower-bdd-V'}
        \frac{V'(x+a)}{x}\geq C_1|x|^{\delta}-C_2|x|^{\frac{\delta}{2}}+C_3, \ \ \text{for any $x\neq0$}.
    \end{equation}
    By (ii), it holds that $\liminf_{|x|\to\infty}\frac{V'(x+a)}{x|x|^{\delta}}>0$. Thus, there exists $R>0$ and $C_1>0$ such that 
    \begin{equation*}
        \frac{V'(x+a)}{x}\geq C_1|x|^{\delta}, \ \text{ for all } \ |x|\geq R.
    \end{equation*}
    Let $C_3=\frac{V''(a)}{2}>0$.
    Notice that $\lim_{x\to 0}\frac{V'(x+a)}{x}=V''(a)>0$, there exist $0<\epsilon_a<R$ and  such that
    \begin{equation*}
        \frac{V'(x+a)}{x}\geq C_3, \ \text{ for all } \ |x|\leq \epsilon_a.
    \end{equation*}
    Take
    \begin{equation*}
        C_2=\epsilon_a^{-\frac{\delta}{2}}\bigg(\sup_{\epsilon_a\leq |x|\leq R}\bigg|\frac{V'(x+a)}{x}\bigg|+C_1R^{\delta}+C_3\bigg).
    \end{equation*}
    One can verify \eqref{eq:lower-bdd-V'} in three cases: $\abs{x}\geq R$, $\epsilon_a<|x|< R$ and $|x|\leq \epsilon_a$, separately. 
    Then by \eqref{eq:lower-bdd-V'} we have
    \begin{equation}\label{E:V inequal}
        -2xV'(x+a)=-2|x|^2\frac{V'(x+a)}{x}\leq -2C_1|x|^{2+\delta}+2C_2|x|^{2+\frac{\delta}{2}}-2C_3|x|^2, \ \text{ for all } \ x\in\R.
    \end{equation}
    Let $\overline{\xs}:=\sup_{x\in\R}|\xs(x)|$. Choose
    \begin{equation*}
        g_a(z,w)=2C_1z^{1+\frac{\delta}{2}}-2C_2z^{1+\frac{\delta}{4}}+(2C_3+2\theta)z-2\theta z^{\frac{1}{2}}w^{\frac{1}{2}}-\overline{\xs}^2.
    \end{equation*}
    Obviously, $g_a$ is continuous and $g_a(z,\cdot)$ is decreasing for any $z\geq 0$ and hence
    \begin{equation}\label{E:0 r2 inf}
       \inf_{0\leq w\leq r^2}g_a(z,w)=g_a(z,r^2), \ \text{ for any } \ r>0.
    \end{equation}  
    Also, \eqref{E:V inequal} implies that
    \begin{equation}\label{eq:1210-1}
        -2xV'(x+a)-2x(W'*\mu)(x)+|\xs(x+a)|^2\leq -g_a(|x|^2,\|\mu\|_2^2).
    \end{equation}
    For any $r>0$, a direct computation yields 
    \[
    g_a(r^2,r^2)=2r^2\big(C_1r^{\delta}-C_2r^{\frac{\delta}{2}}+C_3\big)-\overline{\xs}^2,
    \]
    \[
    \frac{\partial g_a(z,r^2)}{\partial z}\bigg|_{z=r^2}=2\big(1+\frac{\delta}{2}\big)C_1r^{\delta}-2\big(1+\frac{\delta}{4}\big)C_2r^{\frac{\delta}{2}}+2C_3+\theta,
    \]
    and for all $z\geq 0$,
    \begin{equation}\label{eq:second-derivation-g_a}
        \frac{\partial^2 g_a(z,r^2)}{\partial z^2}=\delta\big(1+\frac{\delta}{2}\big)C_1z^{-(1-\frac{\delta}{2})}-\frac{\delta}{2}\big(1+\frac{\delta}{4}\big)C_2z^{-(1-\frac{\delta}{4})}+\frac{\theta r}{2}z^{-\frac{3}{2}}.
    \end{equation}
    Then there exists $r_a>0$ such that for any $0<r\leq r_a$, we have
    \begin{equation}\label{E:r2 c1>0}
        2r^2\big(C_1r^{\delta}-C_2r^{\frac{\delta}{2}}+C_3\big)>0,
           \end{equation}
    and 
           \begin{equation}\label{E:ga partial >0}
           \frac{\partial g_a(z,r^2)}{\partial z}\bigg|_{z=r^2}=2\big(1+\frac{\delta}{2}\big)C_1r^{\delta}-2\big(1+\frac{\delta}{4}\big)C_2r^{\frac{\delta}{2}}+2C_3+\theta>0.
           \end{equation}
 
    Now, we fix $0<r\leq r_a$. 
    By \eqref{eq:second-derivation-g_a}, we have $\lim_{z\downarrow 0}\frac{\partial^2 g_a(z,r^2)}{\partial z^2}=\infty$. Hence, there exists $\epsilon_r>0$ such that
    \begin{equation*}
        \frac{\partial^2 g_a(z,r^2)}{\partial z^2}>0, \ \text{ for all } \ 0<z<\epsilon_r.
    \end{equation*}
    On the other hand,
    \begin{equation*}
        \frac{\partial^2 g_a(z,r^2)}{\partial z^2}=\frac{\delta}{8}z^{(\frac{\delta}{4}-1)}\big(4(2+\delta)C_1z^{\frac{\delta}{4}}-(4+\delta)C_2\big)+\frac{\theta r}{2}z^{-\frac{3}{2}}>0, \ \text{ for all } \ z\geq C_{\delta}:=\Big(\frac{(4+\delta)C_2}{4(2+\delta)C_1}\Big)^{\frac{4}{\delta}}.
    \end{equation*}
   Let
    \begin{equation}\label{eq:theta-r}
        \theta_r=\frac{\delta}{r}\big(1+\frac{\delta}{4}\big)C_2C_{\delta}^{\frac{3}{2}}\epsilon_r^{-(1-\frac{\delta}{4})}.
    \end{equation} 
    For  $\theta\geq \theta_r$, by \eqref{eq:second-derivation-g_a}, one has
    \begin{equation*}
        \frac{\partial^2 g_a(z,r^2)}{\partial z^2}\geq \delta\big(1+\frac{\delta}{2}\big)C_1z^{-(1-\frac{\delta}{2})}>0, \ \text{ for all } \ \epsilon_r\leq z\leq C_{\delta}.
    \end{equation*}
    Thus, for all $\theta\geq \theta_r$, we conclude that
    \begin{equation}\label{eq:1210-2}
        \frac{\partial^2 g_a(z,r^2)}{\partial z^2}>0, \ \text{ for all } \ z>0.
    \end{equation}
    So $g_a(\cdot,r^2)$ is convex. Thus, by \eqref{E:ga partial >0},
    \begin{equation}\label{E:ga z r2>0}
        \frac{\partial g_a(z,r^2)}{\partial z}\geq \frac{\partial g_a(r^2,r^2)}{\partial z}>0, \ \text{ for all } \ z\geq r^2.
    \end{equation}
    \eqref{E:r2 c1>0} implies that there exists $ \sigma_r$ such that
    \begin{equation}\label{eq:sigma-r}
       0< \sigma_r<2r^2\big(C_1r^{\delta}-C_2r^{\frac{\delta}{2}}+C_3\big).
    \end{equation}
    Then by \eqref{E:ga z r2>0}, we have
    \begin{equation}\label{eq:1210-3}
        g_a(z,r^2)\geq g_a(r^2,r^2)>\xs_r-\overline{\xs}^2\geq 0, \ \text{ for all } \ z\geq r^2.
    \end{equation}
    Therefore, for any $0<r\leq r_a$, let $\theta_r, \sigma_r$ be given as in \eqref{eq:theta-r} and \eqref{eq:sigma-r}.
    It follows from \eqref{E:0 r2 inf}, \eqref{eq:1210-1}, \eqref{eq:1210-2},  \eqref{eq:1210-3}  that for all $\theta\geq \theta_r, \ \overline{\xs}^2\leq \sigma_r$, equation \eqref{eq:mvsystem} is locally dissipative at $a$ with radius $r$.
\end{proof}

\begin{remark}\label{rem:hyper-dissipation-local}
    If $V'$ is hyper-dissipative (see Assumption \ref{asp:potential} (iii)), i.e., there exist $\xa>0$, $\xb>0$, $\xd>0$ such that, for all $x,y\in\R$,
\[
-(x-y)(V'(x)-V'(y))\leq-\xa\abs{x-y}^{2+\xd}+\xb.
\]
Then we have
\[
-xV'(x)\leq -\xa\abs{x}^{2+\xd}+ |x||V'(0)|+ \xb,
\]
and hence
\[
\frac{V'(x)}{x|x|^{\delta}}=\frac{xV'(x)}{|x|^{2+\delta}}\geq \xa-\frac{|V'(0)|}{|x|^{1+\delta}}-\frac{\beta}{|x|^{2+\delta}}.
\]
Thus, $\lim_{|x|\to\infty}\frac{V'(x)}{x|x|^{\delta}}=\alpha>0$ and the condition (ii) in Proposition \ref{prop:local-dissipation-at-a} holds true.
\end{remark}

\begin{proof}[Proof of Theorem \ref{T:general V}]
    According to Remark \ref{rem:hyper-dissipation-local}, we have
    \begin{equation}\label{eq:new1210}
        \lim_{x\to\infty}V'(x)=\infty, \ \ \lim_{x\to-\infty}V'(x)=-\infty.
    \end{equation}
    Since $V'$ has only finitely many zeros and each of them is simple, the number of zeros  must be odd.
    Denoting them by $a_1<b_1<a_2<b_2<\cdots<a_{n-1}<b_{n-1}<a_n$, it then follows that 
        \begin{equation}\label{eq:new1210-2}
        V''(a_i)>0, \ 1\leq i\leq n, \ \text{ and } \ V''(b_j)<0, \ 1\leq j\leq n-1.
    \end{equation}
    By virtue of Proposition \ref{prop:local-dissipation-at-a}, Remark \ref{rem:hyper-dissipation-local} and \eqref{eq:new1210-2} that, there exist 
    \[
    0<r<\frac{1}{2}\min_{1\leq i\leq n-1}\abs{a_{i+1}-a_i}, \ \ \theta_0>0, \ \text{ and } \ \xs_0>0
    \]
    such that, for all $\theta\geq \theta_0$ and $\overline{\xs}\leq \xs_0$, equation \eqref{eq:mvsystem} is locally dissipative at $a_{k}$ with radius $r$ for all $1\leq k\leq n$. 
    Then the desired result follows from Theorem \ref{thm:local-convergence}.
\end{proof}

\begin{proof}[Proof of Theorem \ref{thm:double-well}]
    According to Theorem \ref{thm:local-convergence}, it suffices to show that equation \eqref{eq:double-well} are locally dissipative at $-1,1$ with configurations $(r,g_{-1}), (r,g_1)$ for $r=\frac{9-\sqrt{17}}{8}$ and for some $g_{\pm 1}$.

    Note that 
    \[
    V'(x)=x^3-x, \ W(x)=\frac{\theta}{2}x^2.
    \]
    Then we construct $g_{\pm 1}$ as follows.

    \noindent\textbf{Construction of $g_{1}$:} For any $x\in \R, \mu\in \PP_2(\R)$, by a simple calculation, we have
    \begin{equation*}
        -2xV'(x+1)-2x(W'*\mu)(x)+|\sigma(x+1)|^2\leq -\big(2|x|^4-6|x|^3+(4+2\theta)|x|^2-2\theta |x|\|\mu\|_2-\overline{\sigma}^2\big),
    \end{equation*}
    where 
    \begin{equation*}
        \overline{\sigma}:=\sup_{x\in \R}|\sigma(x)|.
    \end{equation*}
    Then we choose $g_{1}$ as follows: for any $z,w\geq 0$,
    \begin{equation}\label{1127-7}
       g_{1}(z,w)=2z^2-6z^{\frac{3}{2}}+(4+2\theta)z-2\theta z^{\frac{1}{2}}w^{\frac{1}{2}}-|\overline{\xs}|^2.
    \end{equation}
    Obviously, $g_1$ is continuous and $g_1(z,\cdot)$ is decreasing for any $z\geq 0$ and hence
    \begin{equation*}
       \inf_{0\leq w\leq r^2}g_1(z,w)=g_1(z,r^2), \ \text{ for any } \ r>0.
    \end{equation*}
    Note that
    \begin{equation*}
       \frac{\partial g_1(z,w)}{\partial z}=4z-9z^{\frac{1}{2}}+4+2\theta-\theta w^{\frac{1}{2}}z^{-\frac{1}{2}},
    \end{equation*}
    and
    \begin{equation*}
       \frac{\partial^2 g_1(z,w)}{\partial z^2}=4-\frac{9}{2}z^{-\frac{1}{2}}+\frac{1}{2}\theta w^{\frac{1}{2}}z^{-\frac{3}{2}}.
    \end{equation*}
    Then by calculation, we have for any $w>0$, $g_1(\cdot,w)$ is convex if and only if 
    \begin{equation}\label{1127-6}
       \theta\geq \frac{27}{16w^{\frac{1}{2}}}.
    \end{equation}
    Note that for any $0<r<1$, 
    \begin{equation*}
       g_{1}(r^2,r^2)=2r^2(r-1)(r-2)-\overline{\xs}^2,
    \end{equation*}
    and
    \begin{equation*}
       \frac{\partial g_1}{\partial z}(r^2,r^2)=4r^2-9r+4+\theta>\theta-1.
    \end{equation*}
    Now for any fixed $0<r<1$, if we choose
    \begin{equation*}
       \theta\geq\frac{27}{16r}>1,\ \text{ and } \ \overline{\xs}^2<2r^2(r-1)(r-2),
    \end{equation*}
    then
    \begin{equation*}
       g_1(r^2,r^2)>0,\ \text{ and } \ \frac{\partial g_1}{\partial z}(r^2,r^2)>0.
    \end{equation*}
    By \eqref{1127-6}, we have $g_1(\cdot, r^2)$ is convex. Hence, 
    \begin{equation*}
       \frac{\partial g_1}{\partial z}(z,r^2)\geq \frac{\partial g_1}{\partial z}(r^2,r^2)>0 \ \text{ for all } \ z\geq r^2,
    \end{equation*}
    and thus
    \begin{equation*}
       g_1(z,r^2)\geq g_1(r^2,r^2)>0 \ \text{ for all } \ z\geq r^2.
    \end{equation*}
    Hence, for any $0<r<1$ and choosing $g_1$ as in \eqref{1127-7}, equation \eqref{eq:double-well} is locally dissipative at $1$ with configuration $(r,g_1)$ whenever
    \begin{equation}\label{eq:1007-1}
       \theta\geq\frac{27}{16r},\ \text{ and } \ \overline{\xs}^2<2r^2(r-1)(r-2).
    \end{equation}
    Note that 
    \begin{equation}\label{eq:1007-2}
        \max_{0<r<1}2r^2(r-1)(r-2)=2r^2(r-1)(r-2)\Big|_{r=\frac{9-\sqrt{17}}{8}}=\frac{51\sqrt{17}-107}{256}.
    \end{equation}
    Combining \eqref{eq:1007-1}\eqref{eq:1007-2}, under the condition \eqref{eq:con-double-well} in Theorem \ref{thm:double-well},
    equation \eqref{eq:double-well} is locally dissipative at $1$ with configuration $(r,g_1)$ where $r=\frac{9-\sqrt{17}}{8}$ and $g_1$ is given as in \eqref{1127-7}.

    \noindent\textbf{Construction of $g_{-1}$:} Similarly, let $g_{-1}=g_1$ be given as in \eqref{1127-7}, we can show that equation \eqref{eq:double-well} is also locally dissipative at $-1$ with configuration $\big(\frac{9-\sqrt{17}}{8},g_{-1}\big)$.
\end{proof}

\begin{proof}[Proof of Theorem \ref{thm:multi-well}]
Take
\begin{align*}
g_0(z,w)&=2z^3-10z^2+(8+2\theta)z-2\theta z^{\frac12}w^{\frac12}-\overline{\xs}^2,& r_0&=\frac{\sqrt{15-3\sqrt{13}}}{3},\\
g_2(z,w)&=2z^3-20z^{\frac52}+70z^2-100z^{\frac32}+(48+2\theta)z-2\theta z^{\frac12}w^{\frac12}-\overline{\xs}^2,& r_2&=r_0,\\
g_{-2}(z,w)&=g_2(z,w),& r_{-2}&=r_0,
\end{align*}
where $\overline{\sigma}:=\sup_{x\in \R}|\sigma(x)|$. Similar to the calculation in the proof of Theorem \ref{thm:double-well}, we can prove that the equation \eqref{eq:multi-well} is locally dissipative at $0, \pm2$.
The proof is finished by applying Theorem \ref{thm:local-convergence}.
\end{proof}

\begin{proof}[Proof of Theorem \ref{thm:double-well-vanish}]
    By a simple calculation, we have
   \begin{equation*}
       -2xV'(x+1)\leq -2x^2(x+1)(x+2)+\frac{1}{4}.
   \end{equation*}
   Then we have
   \begin{equation*}
        -2xV'(x+1)-2x(W'*\mu)(x)+|\sigma(x+1)|^2\leq -\big(2|x|^4-6|x|^3+(4+2\theta)|x|^2-2\theta |x|\|\mu\|_2-\overline{\sigma}^2-\frac{1}{4}\big),
    \end{equation*}
    where $\overline{\sigma}:=\sup_{x\in \R}|\sigma(x)|$.
    Choose $g_1$ as follows: for any $z,w\geq 0$,
    \begin{equation}\label{eq:g_1-dw-vanish}
       g_{1}(z,w)=2z^2-6z^{\frac{3}{2}}+(4+2\theta)z-2\theta z^{\frac{1}{2}}w^{\frac{1}{2}}-|\overline{\xs}|^2-\frac{1}{4}.
    \end{equation}
    Repeating the proof of that in Theorem \ref{thm:double-well}, we have the equation \eqref{eq:double-well-vanish} is locally dissipative at $1$ with configuration $(r, g_1)$ with $r=\frac{9-\sqrt{17}}{8}$ whenever
    \[
    \theta\geq \frac{27(9+\sqrt{17})}{128}, \ \text{ and } \ \overline{\xs}^2+\frac{1}{4}<\frac{51\sqrt{17}-107}{256},
    \]
    i.e., \eqref{eq:para-double-well-vanish} holds. Similarly, choose $g_{-1}=g_1$ as in \eqref{eq:g_1-dw-vanish}, the equation \eqref{eq:double-well-vanish} is locally dissipative at $-1$ with configuration $\big(\frac{9-\sqrt{17}}{8},g_{-1}\big)$. Then the desired results follow from Theorem \ref{thm:local-convergence} directly.
\end{proof}

Motivated by our results, we conclude the paper with the following conjecture.

In this work, by exploiting the order-preserving property of the associated semiflow, we have obtained a partial characterization of the basins of attraction of invariant measures, both in the sense of the stochastic order and the topology. 
In particular, our results reveal the existence of unbounded open subsets contained in the basins of attraction of minimal/maximum invariant measures.

In the theory of monotone dynamical systems, a classical \emph{saddle-point structure}
(or \emph{saddle-type behavior}) arises when a system possesses three equilibria
$a<b<c$, where $a$ and $b$ are stable.
More precisely, the state space can be decomposed into three disjoint invariant sets:
the basin of attraction $B_1$ of $a$, the basin of attraction $B_2$ of $b$,
and an unordered invariant set $M$ containing $c$.
The set $M$ separates the two basins $B_1$ and $B_2$ and is usually referred to as the
\emph{separatrix} (see e.g., Jiang-Liang-Zhao \cite{Jiang-Liang-Zhao2004}).

\begin{center}
\textbf{Conjecture.}
\end{center}
In the double-well granular media equation, the three invariant measures admit an
analogous saddle-point structure.
More precisely, the unstable invariant measure $\mu_0$ is expected to be contained
in an unordered invariant separating set that acts as a separatrix between the basins of
attraction of the two stable invariant measures.

 \section*{Acknowledgement}
 We acknowledge the financial supports of EPSRC grant ref (ref. EP/S005293/2), Royal Society through the Award of
 Newton International Fellowship (ref. NIF/R1/221003) and National Natural Science Foundation of China (No. 12501184, No. 12171280).

\begin{appendices}

\section{Nonexistence of Strong Ordered Pairs for Stochastic Order}\label{sec:no strong order}

In addition to closedness of the stochastic order, almost all classical results in the theory of monotone dynamical systems are established on the fundamental assumption that the state space $X$ is not only endowed with a partial order relation ``$\leq$'', but is also a ``strongly ordered space'' (see e.g., Hirsch \cite{Hirsch84,Hirsch88}).
A strongly ordered space implies that the interior of the partial order relation ``$\leq$'' (referred to as the strong ordering ``$\ll$'') is a nonempty relation on $X$, and its closure is ``$\leq$''.
The strong ordering ``$\ll$'' is a relation for which two points ``$x\ll y$'' means that there are open neighbourhoods $U,V$ in $X$ of $x,y$ respectively such that $U\leq V$, where the term ``$U\leq V$'' refers to $u\leq v$ for any $u\in U$ and $v\in V$. 
If $x\ll y$, we call $(x,y)$ is a \emph{strong ordered pair}.
Based on the tool of order topology, especially the topology generated by the interiors of order intervals enclosed by strong ordered pairs, the theory of monotone dynamical systems and its applications to cooperative irreducible ODEs, delay equations and parabolic equations  have undergone extensively investigation (see e.g., \cite{Smith95,HS05,Z17}).

Unfortunately, in the appendix, we will show that the strong ordering relation generated by the stochastic order ``$\leq_{\st}$'' in $\PP_2(\R)$ is empty, meaning that NO strong ordered pairs exist. We actually prove a stronger conclusion -- for any $\mu\in\PP_2(\R)$, every open set $B\subset\PP_2(\R)$ has an element that is order unrelated with $\mu$ (see Proposition \ref{prop:no-strong-order}). 
This implies that classical theory of monotone dynamical systems are inapplicable to order-preserving semigroups on $\PP_2(\R)$. Consequently, exploration of properties of the stochastic order and the results of monotone dynamical systems, as discussed in Section \ref{subsec:order-bounded-equivalent}, \ref{subsec:global attractor} and \ref{subsec:local attractor}, are indispensable.

\begin{proposition}\label{prop:no-strong-order}
    Let $\mu\in \PP_2(\R)$ and $U$ be an open set in $\PP_2(\R)$.
    Then there exists $\nu\in U$ such that $\mu, \nu$ are order unrelated.
    In particular, any order interval in $\PP_2(\R)$ has no interior point.
\end{proposition}

\begin{proof}
    Since $U$ is open, there exist $\nu_0\in U$ and $\delta>0$ such that 
    \begin{equation*}
        B(\nu_0,2\xd):=\{\nu\in \PP_2(\R): \WW_2(\nu,\nu_0)<2\delta\}\subset U.
    \end{equation*}

    \noindent\textbf{Claim.} There exist $c>0$ and $\nu_1\in B(\nu_0,\xd)$ such that
    \begin{equation}\label{1117-2}
        \nu_1((-\8,-r])>0, \ \ \nu_1([r,\8))>0, \ \text{ for all } \ r\geq c.
    \end{equation}
    
    \noindent\textbf{Proof of Claim:} Note that there exists $c>0$ such that $\nu_0((-c,c))>0$. Let
    \begin{equation*}
        \rho(x)=\frac{e^{-|x|}}{\int_{\R}|x|^2e^{-|x|}\d x}, \ \ \rho_{\epsilon}(x):=\epsilon^{-1}\rho(\epsilon^{-1}x), \ \epsilon\in (0,1).
    \end{equation*}
    For any $\epsilon\in(0,1)$, we have
    \begin{equation*}
        \begin{split}
            \alpha_{\epsilon}&:=\int_{(-\8,-c]\cup[c,\8)}\rho_{\epsilon}(x)\d x \\
            &\leq \frac{1}{c^2}\int_{(-\8,-c]\cup[c,\8)}|x|^2\rho_{\epsilon}(x)\d x\\
            &\leq \frac{1}{c^2}\int_{\R}|x|^2\rho_{\epsilon}(x)\d x\\
            &=\frac{\epsilon^2}{c^2}\int_{\R}|x|^2\rho(x)\d x\\
            &=\frac{\epsilon^2}{c^2}.
        \end{split}
    \end{equation*}
    Choose $\epsilon_0\in(0,1)$ such that $\epsilon_0^2<c^2\nu_0((-c,c))$ and we define, for any $0<\epsilon<\epsilon_0$,
    \begin{equation*}
        \d\nu_{1,\epsilon}(x)=
        \begin{cases}
            \frac{\nu_0((-c,c))-\alpha_{\epsilon}}{\nu_0((-c,c))}\d\nu_0(x), & \text{on } (-c,c),\\
            \d\nu_0(x)+\rho_{\epsilon}(x)\d x, & \text{on } (-\8,-c]\cup [c,\8).
        \end{cases}
    \end{equation*}
    Since $\nu_0\in \PP_2(\R)$, it is easy to check that $\nu_{1,\epsilon}\in \PP_2(\R)$ satisfying \eqref{1117-2} for all $0<\epsilon<\epsilon_0$. Note that for any $f\in C_b(\R)$, where $C_b(\R)$ is the set of all bounded continuous functions from $\R$ to $\R$,
    \begin{equation}\label{1117-1}
        \begin{split}
            &\ \ \ \ \bigg|\int_{\R}f(x)\d\nu_{1,\epsilon}(x)-\int_{\R}f(x)\d\nu_0(x)\bigg|\\
            &\leq \frac{\alpha_{\epsilon}}{\nu_0((-c,c))}\int_{(-c,c)}|f(x)|\d\nu_0(x)+\int_{(-\8,-c]\cup[c,\8)}|f(x)|\rho_{\epsilon}(x)\d x\\
            &\leq 2|f|_{\infty}\alpha_{\epsilon}\\
             &\leq \frac{2|f|_{\infty}}{c^2}\epsilon^2\to 0 \ \text{ as } \ \epsilon\to 0,
        \end{split}
    \end{equation}
    where $|f|_{\infty}=\sup_{x\in \R}|f(x)|$.
    Hence $\nu_{1,\epsilon}$ weakly converges to $\nu_0$ as $\epsilon\to 0$, which is denoted by $\nu_{1,\epsilon}\xrightarrow{w} \nu_0$ as $\epsilon\to 0$. Moreover, similar to \eqref{1117-1}, we have
    \begin{equation}\label{eq:0228-1}
            \begin{split}
                &\ \ \ \ \lim_{\epsilon\to 0}\bigg|\int_{\R}|x|^2\d\nu_{1,\epsilon}(x)-\int_{\R}|x|^2\d\nu_0(x)\bigg|\\
                &\leq\lim_{\epsilon\to 0}\Big(c^2\alpha_{\epsilon}+\int_{(-\8,-c]\cup[c,\8)}|x|^2\rho_{\epsilon}(x)\d x\Big) \\
                &\leq \lim_{\epsilon\to 0} 2\epsilon^2=0.
            \end{split}
    \end{equation}
    By Theorem 6.9 in \cite{Villani2009}, we have $\lim_{\epsilon\to 0}\WW_2(\nu_{1,\epsilon},\nu_0)=0$ due to $\nu_{1,\xe}\goto{w}\nu_0$ and \eqref{eq:0228-1}.
    Then there exists $0<\epsilon_1<\epsilon_0$ such that $\nu_1:=\nu_{1,\epsilon_1}\in B(\nu_0,\xd)$ satisfying \eqref{1117-2}.
    Thus, we have proved the Claim.

    If $\mu$ and $\nu_1$ are order unrelated, the proof is completed by choosing $\nu=\nu_1$. Otherwise, without loss of generality, we assume $\mu\leq_{\st}\nu_1$. Then we have
    \begin{equation}\label{eq:0816-0}
        \mu((-\8,-r])\geq \nu_1((-\8,-r]), \ \ \mu([r,\8))\leq \nu_1([r,\8)), \ \text{ for all } \ r\geq c.
    \end{equation}
    We will finish the proof by considering the following two cases.

    \noindent\textbf{Case 1.} $\mu([r_0,\8))<\nu_1([r_0,\8))$ for some $r_0\geq c$. 
    By \eqref{1117-2} and \eqref{eq:0816-0}, one has
    \begin{equation}\label{eq:0816-1}
        \mu((-\8,-r])\geq\nu_1((-\8,-r])>0, \ \text{ for all } \ r\geq c,
    \end{equation}
    Note that $\lim_{r\to \infty}\mu((-\8,-r])=0$,
    then for any positive integer $n\geq c$, there exists $r_n>n+1$ such that
    \begin{equation*}
        0<2\mu((-\8,-r_n])<\nu_1((-\8, -n]).
    \end{equation*}
    For any $n\geq c$, define
    \begin{equation}\label{eq:0816-2}
        \d\nu_{2,n}(x)=
        \begin{cases}
            \d\nu_1(x), & \text{on } (-n,\8),\\
            \big(\nu_1((-\8, -n])-2\mu((-\8,-r_n])\big)\d x, & \text{on } (-(n+1),-n],\\
            0,   & \text{on } (-r_n,-(n+1)],\\
            2\d\mu(x), & \text{on } (-\8,-r_n].
        \end{cases}
    \end{equation}
    It is easy to check that $\nu_{2,n}\in \PP_2(\R)$ for all $n\geq c$. Moreover, $\nu_{2,n}\xrightarrow{w} \nu_1$ as $n\to \infty$ since for any $f\in C_b(\R)$,
    \begin{equation*}
        \begin{split}
            &\ \ \ \ \bigg|\int_{\R}f(x)\d\nu_{2,n}(x)-\int_{\R}f(x)\d\nu_1(x)\bigg|\\
            &\leq \int_{(-\8, -n]}|f(x)|\d\nu_{2,n}(x)+\int_{(-\8, -n]}|f(x)|\d\nu_1(x)\\
            &\leq 2|f|_{\infty}\nu_1((-\8, -n])\to 0 \ \text{ as } \ n\to \infty.
        \end{split}
    \end{equation*}
    Note also that $\mu,\nu_1\in \PP_2(\R)$, then for any $n\geq c$,
    \begin{equation*}
        \begin{split}
            &\ \ \ \ \bigg|\int_{\R}|x|^2\d\nu_{2,n}(x)-\int_{\R}|x|^2\d\nu_1(x)\bigg|\\
            &\leq \int_{(-\8, -n]}|x|^2\d\nu_{2,n}(x)+\int_{(-\8, -n]}|x|^2\d\nu_1(x)\\
            &\leq (n+1)^2\nu_1((-\8, -n])+2\int_{(-\8,-r_n]}|x|^2\d\mu(x)+\int_{(-\8, -n]}|x|^2\d\nu_1(x)\\
            &\leq 2\int_{(-\8,-r_n]}|x|^2\d\mu(x)+5\int_{(-\8, -n]}|x|^2\d\nu_1(x)\to 0 \ \text{ as } \ n\to \infty.
        \end{split}
    \end{equation*}
    Hence, it follows from Theorem 6.9 in \cite{Villani2009} again that $\lim_{n\to \infty}\WW_2(\nu_{2,n},\nu_1)=0.$
    Thus, there exists $n_0\geq c$ such that $\nu:=\nu_{2,n_0}\in B(\nu_1,\xd)\subset B(\nu_0,2\xd)$. By \eqref{eq:0816-1}, the construction of $\nu_{2,n_0}$ in \eqref{eq:0816-2} and the choice of $r_0$ in Case 1, we have
    \begin{equation*}
        \begin{split}
            0<\mu&((-\8,-r])<2\mu((-\8,-r])=\nu((-\8,-r]), \text{ for all } \ r\geq r_{n_0},\\
            &\mu([r_0,\8))<\nu_1([r_0,\8))=\nu([r_0,\8)).
        \end{split}
    \end{equation*}
    This, together with Lemma \ref{lem:stochastic-order-distribution}, indicates that $\mu,\nu$ are order unrelated.

    \noindent\textbf{Case 2.} $\mu([r,\8))=\nu_1([r,\8))$ for all $r\geq c$.
    By \eqref{1117-2}, one has 
    \begin{equation}\label{y-081601}
    \mu([r,\8))=\nu_1([r,\8))>0 \ \text{ for all } \ r\geq c.
    \end{equation}
    Note that 
    \begin{equation*}
            \lim_{r\to \infty}\mu((-\8,-r])=\lim_{r\to \infty}\mu([r,\8))=0,
    \end{equation*}
    then it follows from \eqref{eq:0816-1} and \eqref{y-081601} that, for any positive integer $n\geq c$, there exists $r_n>n+1$ such that
    \begin{equation*}
        0<2\mu((-\8,-r_n])<\nu_1((-\8, -n]), \ \text{ and } \ 0<2\mu([r_n,\8))<\nu_1([n,\8)).
    \end{equation*}
    For any $n\geq c$, define
    \begin{equation*}
        \d\nu_{3,n}(x)=
        \begin{cases}
            \d\nu_1(x), & \text{on } (-n,n),\\
            \big(\nu_1([n,\8))-2\mu([r_n,\8))\big)\d x, & \text{on } [n,n+1),\\
            \big(\nu_1((-\8, -n])-2\mu((-\8,-r_n])\big)\d x, & \text{on } (-(n+1),-n],\\
            0, & \text{on } (-r_n,-(n+1)]\cup [n+1,r_n),\\
            2\d\mu(x), & \text{on } (-\8,-r_n]\cup [r_n,\8).        
        \end{cases}
    \end{equation*}
    Similarly, we can conclude that $\nu_{3,n}\in\PP_2(\R)$ and $\lim_{n\to \infty}\WW_2(\nu_{3,n},\nu_1)=0.$ Choose $n_0\geq c$ large enough such that $\nu:=\nu_{3,n_0}\in B(\nu_1,\xd)\subset B(\nu_0,2\xd)$, then we have for all $r\geq r_{n_0}$,
     \begin{equation*}
        \begin{split}
            0<&\mu((-\8,-r])<2\mu((-\8,-r])=\nu((-\8,-r]), \ \text{ and } \\
            0<&\mu([r,\8))<2\mu([r,\8))=\nu([r,\8)).
        \end{split}
    \end{equation*}
    Hence, Lemma \ref{lem:stochastic-order-distribution} gives $\mu,\nu$ are order unrelated.
\end{proof}

\end{appendices}

\clearpage
\phantomsection
\addcontentsline{toc}{section}{References}
\linespread{1.0}
\selectfont
\bibliographystyle{siam}
\bibliography{reference}

\end{document}